\documentclass[12pt,a4paper,twoside]{book}
\usepackage{a4wide}
\usepackage[latin2]{inputenc}
\usepackage{amssymb}
\usepackage{amsmath}
\usepackage{ipe}
\usepackage[mathscr]{eucal}
\usepackage[dvips]{graphicx}
\usepackage{makeidx}

\newcommand{\xauthor}{Jan Hubi\v cka}

\makeatletter

\def\Erdos{Erd\H{o}s}
 \def\dr{\downarrow\!\! }  
\def\Embed#1#2{\Phi^{#1}_{#2}}
\def\HP{{\cal P_{\cal U}}}
\def\HPleq{\leq_{\cal U}}
\def\K{{\cal K}}
\def\LeqUp{\preceq}

\def\U{{\cal U}}
\def\Urysohn{{\mathbb U}}
\def\Rado{{\cal R}}
\def\Convex{{\cal C}}

\def\Words{{\cal W}}
\def\Functions{{\cal F}}
\def\Intervals{{\cal I}}
\def\Bintree{{\cal B}}
\def\Periodic{{\cal S}}
\def\Grammar{{\cal G}}
\def\X{{\cal X}}
\def\Q{\mathbb Q}
\def\Lifts{{\cal L}}
\def\Piece{{\cal P}}
\def\N{\mathbb N}
\def\Z{\mathbb Z}
\def\Poset{\mathbb P}
\def\h{h}
\def\d{d}
\DeclareMathSymbol{\Oddel}
   {\mathbin}{symbols}{"6A}

\def\He{{\cal P_\in}}

\def\HF{{{\cal P}_f}}

\def\Model{{\mathfrak {M}}}

\def\RO{{\overrightarrow {\R}}}
\def\N{{\mathbb N}}
\def\Z{{\mathbb Z}}
\def\Q{{\mathbb Q}}
\def\R{{\cal {R}}}
\def\RD{{\cal {O}}}

\def\T{{\cal T}}
\def\M{{\cal M}}
\def\F{{\cal F}}
\def\Kk{{K_k}}

\def\Poset{{\cal P}}
\def\P{{\cal P}}
\def\Forb{{\mathrm {Forb}}}
\def\Paths{{\overrightarrow{\cal P}}}
\def\Polynoms{{\cal O}}

\def\HE#1,#2,#3{(\,#2\,\Oddel\,#3\,)}
\def\Sur{\mathbb S}

\def\Forb{\mathop{Forb_h}\nolimits}
\def\Forbi{\mathop{Forb_e}\nolimits}
\def\Age{\mathop{Age}\nolimits}
\def\Rel{\mathop{Rel}\nolimits}
\def\CSP{\mathop{CSP}\nolimits}

\def\relsys#1{\mathbf #1}

\def\relS{R_{\mathbf{S}}}
\def\relA{R_{\mathbf{A}}}
\def\extsys#1{\mathbf #1'}

\def\rel#1#2{R_{\mathbf{#1}}^{#2}}

\def\ext#1#2{X_{\mathbf{#1}}^{#2}}
\def\extl#1#2{X_{#1}^{#2}}

\def\Fraisse{Fra\"{\i}ss\' e}

\DeclareMathSymbol{\Oddel}
   {\mathbin}{symbols}{"6A}
\def\HE#1,#2,#3{(\,#2\,\Oddel\,#3\,)}
\def\TV{{\mathcal {TV}}}
\def\K{\mathcal {K}}

\def\apple{\heartsuit}

\def\Collapse{\lower-0.6ex\hbox{$\varphi$}\kern-.3em/\kern-.25em _n}

\def\TA{\HE 1,\emptyset,\emptyset}
\def\TB{\HE 2,{\{\TA\}},\emptyset}
\def\TC{\HE 3,\emptyset,{\{\TB\}}}
\def\TD{\HE 4,\emptyset,{\{\TA,\TB,\TC\}}}
\def\Sur{\mathbb S}

\newtheorem{thm}{Theorem}[chapter]
\newtheorem{fact}[thm]{Fact}
\newtheorem{claim}[thm]{Claim}
\newtheorem{defn}[thm]{Definition}
\newtheorem{corollary}[thm]{Corollary}
\newenvironment{example}[1][Example.]{\begin{trivlist}
\item[\hskip \labelsep {\bfseries #1}]}{\end{trivlist}}

\newtheorem{prop}[thm]{Proposition}
\newtheorem{prob}[thm]{Problem}

\newtheorem{lem}[thm]{Lemma} 

\def\Growth{\mathop{\operator@font {Growth}}}

\providecommand{\remarkname}{Remark}
\newenvironment{remark}[1][\remarkname]{\par
  \normalfont \topsep6\p@\@plus6\p@\relax
  \trivlist
  \item[\hskip\labelsep \bf {#1\ignorespaces.}]\ignorespaces
}{%
\endtrivlist
\par
}

\DeclareRobustCommand{\qed}{%
  \ifmmode \mathqed
  \else
    \leavevmode\unskip\penalty9999 \hbox{}\nobreak\hfill
    \quad\hbox{$\square$}%
  \fi
}
\let\QED@stack\@empty
\let\qed@elt\relax
\newcommand{\pushQED}[1]{%
  \toks@{\qed@elt{#1}}\@temptokena\expandafter{\QED@stack}%
  \xdef\QED@stack{\the\toks@\the\@temptokena}%
}
\newcommand{\popQED}{%
  \begingroup\let\qed@elt\popQED@elt \QED@stack\relax\relax\endgroup
}
\def\popQED@elt#1#2\relax{#1\gdef\QED@stack{#2}}
\newcommand{\qedhere}{%
  \begingroup \let\mathqed\math@qedhere
    \let\qed@elt\setQED@elt \QED@stack\relax\relax \endgroup
}

\providecommand{\proofname}{Proof}
\newenvironment{proof}[1][\proofname]{\par
  \pushQED{\qed}%
  \normalfont \topsep6\p@\@plus6\p@\relax
  \trivlist
  \item[\hskip\labelsep \bf {#1\ignorespaces.}]\ignorespaces
}{%
\popQED\endtrivlist
\par
}

\begin{document}

%

\begin{titlepage}
\begin{center}
\large
Charles University in Prague \\[0mm]
Faculty of Mathematics and Physics \\[0mm]
Department of Applied Mathematics \\[12mm]
{\bf \huge Doctoral Thesis} \\[8mm]
\begin{figure}[h] \centering \scalebox{.5}{\includegraphics{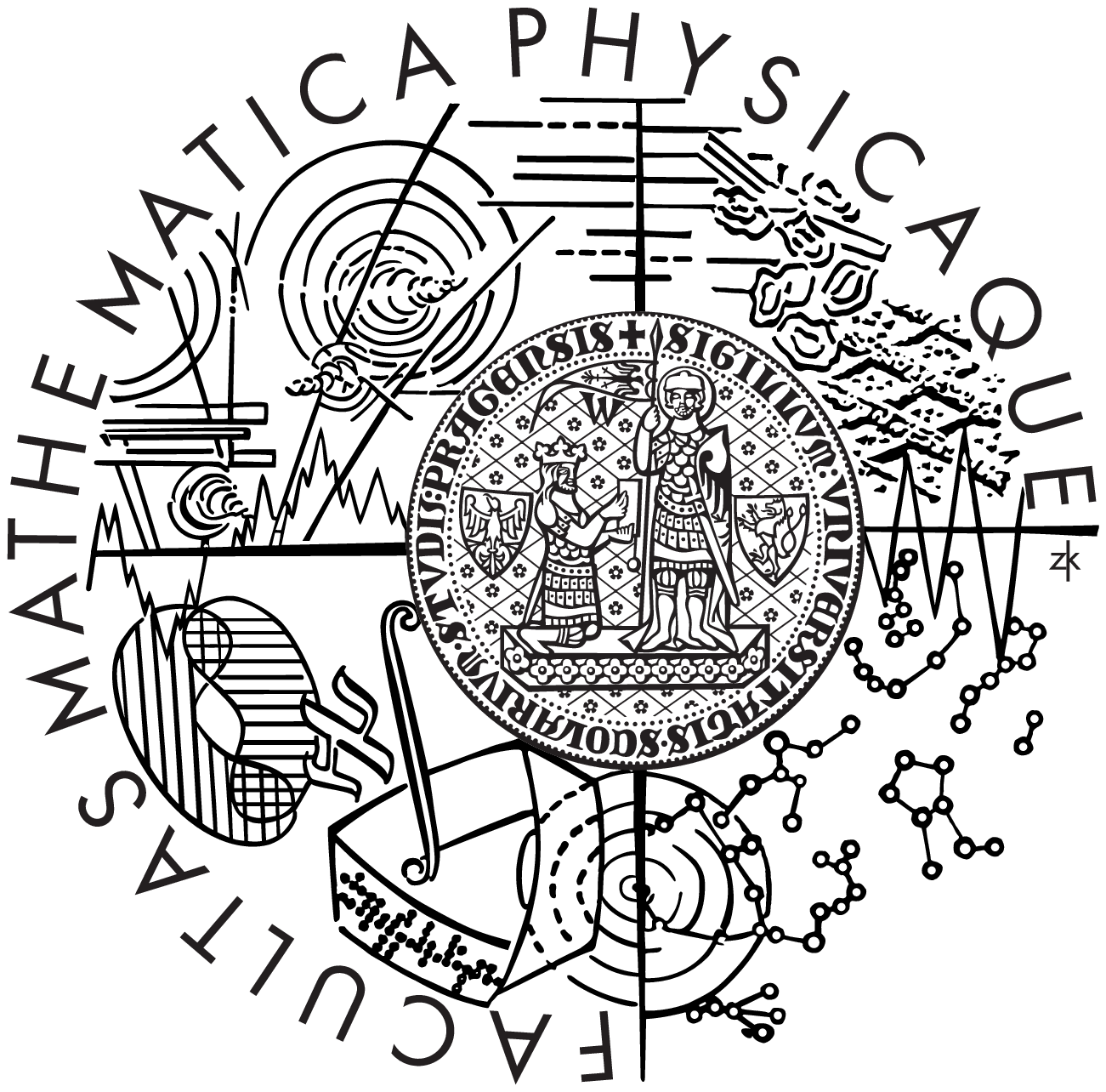}} \end{figure}
\vspace{10mm}
{\huge \xauthor} \\[8mm]
{\huge \textbf{Combinatorial Properties of Finite Models}} \\[15mm]
Supervisor \\[2mm]
{\Large Prof. RNDr. Jaroslav Ne\v set\v ril, DrSc.} \\[10mm]
Study program \\[2mm]
Computer Science \\[0mm]
Discrete Models and Algorithms
\end{center}
\end{titlepage}
\eject
~
\eject
\noindent
\textbf{\Large Acknowledgment}

\vspace{5mm}

\noindent
I would like to thank my advisor, Prof. Jaroslav Ne\v set\v ril, for his
lead, patience and for fixing sometimes unbelievable mistakes I was doing
(seemingly) intentionally and perpetually.  I would like to thank my
consultant Patrice Ossona de Mendez and 
Albert Atserias,
Manuel Bodirsky,
Peter J. Cameron,
David Howard,
Yarred Nigussie,
Ale\v s Pultr,
Robert \v S\'amal,
Norbert W.~Sauer,
William T.~Trotter and
Anatoly M.~Vershik,
for remarks and questions that
motivated my research. Robert \v S\'amal also noticed several issues in
the introduction chapter.  I am very grateful to Andrew Goodall who survived
several horrors in Madrid, Prague, T\'abor and \v Cesk\'y Krumlov and gave a
lot of feedback for several chapters of the thesis.  I also enjoyed
working at the Department of Applied Mathematics and the Institute for
Theoretical Computer Science for eight years.  Finally I would like to thank an
anonymous naked woman who sleep-walking entered our room and stole a pillow in
summer 2001 and inspired me for my first mathematical result.

\vfill
\noindent
I hereby declare that I have written this thesis on my own, and the references include all the sources of information I have exploited. I agree with the lending of this thesis.

\vspace{10mm}
\noindent
Prague, \today{} \hfill{} \xauthor

\eject
~
\eject
\tableofcontents
\cleardoublepage
~
\eject
\noindent {\bf N\' azev pr\' ace:} Kombinatorika kone\v cn\'ych model\accent23u.

\noindent {\bf Autor:} Jan Hubi\v cka

\noindent {\bf Katedra (\'ustav):} Katedra aplikovan\'e matematiky

\noindent {\bf Vedouc\'\i{} doktorsk\'e pr\'ace:} Prof. RNDr. Jaroslav Ne\v set\v ril, DrSc.

\noindent {\bf e-mail vedouc\' \i{}ho:} {\tt nesetril@kam.mff.cuni.cz}

\noindent {\bf Abstrakt:}
 V t\' eto pr\' aci se v\v enujeme univerz\' aln\' \i{}m struktur\' am
pro vno\v ren\' \i{} i homomorfismy a sjednocujeme v\' ysledky t\'ykaj\'\i{}c\'\i{} se obou t\v echto
pojm\accent23u. Uk\'a\v zeme, \v ze mnoh\'e z~univerz\'aln\'i{}ch a
ultrahomogenn\'\i{}ch struktur jsou reprezentovateln\'e pomoc\'i jednoduch\'ych
kone\v cn\'ych technik. O takov\' ych struktur\'ach \v r\'\i{}k\'ame, \v ze
maj\' i{} kone\v cnou prezentaci. Na z\'aklad\v e klasick\'e reprezentace n\'ahodn\' eho grafu
(R. Rado) hled\' ame kone\v cn\'e prezentace pro zn\'am\'e ultrahomogenn\'\i{} struktury.
Podle klasifika\v cn\'\i{}ho programu najdeme prezentace v\v sech ultrahomogenn\'\i{}ch
neorientovan\'ych graf\accent23u, turnaj\accent23u a \v c\'aste\v cn\'ych uspo\v r\'ad\'an\'\i{}.
Uk\'a\v zeme tak\'e prezentaci racion\'aln\'\i{}ho Urysohnova prostoru a n\v ekter\'ych orientovan\'ych graf\accent23u.

V\v enujeme se tak\'e zn\'am\'ym struktur\'am, kter\'e lze pova\v zovat za kone\v cn\'e prezentace.
Uv\'ad\'i{}me p\v rehled struktur, kter\'e popisuj\'\i{} \v c\'aste\v cn\'a uspo\v r\' ad\' an\'\i{}
a u nich\v z m\accent23u\v zeme dok\'azat jejich univerzalitu (nap\v r\'\i{}klad uspo\v r\'ad\'an\'\i{} mno\v zin slov,
geometrick\'ych objekt\accent23u, polynom\accent23u, \v ci homomorfismov\'e uspo\v r\'ad\'an\'\i{} struktur).

Uk\'a\v zeme nov\'y kombinatorick\'y d\accent23ukaz existence univerz\'aln\'ich struktur pro t\v r\'\i{}dy
struktur definovan\'ych pomoc\'\i{} zak\'azan\'ych homomorfism\accent23u. Z tohoto d\accent23ukazu plyne
nov\'a konstrukce homomorfismov\'ych dualit a souvislost s Urysohnov\'ym prostorem. 

\noindent {\bf Kl\'\i{}\v cov\'a slova}: ultrahomogenita, univerzalita, rela\v cn\'\i{} struktury, Urysohn\accent23uv metrick\'y prostor, homomorfismov\' e usp\v r\'ad\'an\'\i{}.
~
\vfill

\noindent {\bf Title:} Combinatorial Properties of Finite Models

\noindent {\bf Author:} Jan Hubi\v cka

\noindent {\bf Department:} Department of applied mathematics

\noindent {\bf Supervisor:} Prof. RNDr. Jaroslav Ne\v set\v ril, DrSc.

\noindent {\bf Supervisor's e-mail address:} {\tt nesetril@kam.mff.cuni.cz}

\noindent {\bf Abstract:} We study countable embedding-universal and
homomorphism-universal structures and unify results related to both of these
notions.  We show that many universal and ultrahomogeneous structures allow a
concise description (called here a finite presentation). Extending classical
work of Rado (for the random graph), we find a finite presentation for each of
the following classes: homogeneous
undirected graphs, homogeneous tournaments and homogeneous partially ordered
sets. We also give a finite presentation of the rational Urysohn metric space and
some homogeneous directed graphs.

We survey well known structures that are finitely
presented.  We focus on structures endowed with natural partial orders and
prove their universality.  These partial orders include partial orders on sets
of words, partial orders formed by geometric objects, grammars, polynomials and
homomorphism orders for various combinatorial objects.

We give a new combinatorial proof of the existence of embedding-universal
objects for homomorphism-defined classes of structures. This relates countable embedding-universal structures
to homomorphism dualities (finite homomorphism-universal structures) and
Urysohn metric spaces.  Our explicit construction also allows us to show
several properties of these structures.


~
\vfill


\chapter{Introduction and motivation}
\label{introdukcehomorder}

It is an old mathematical idea  to reduce the study of a particular class of
objects to a certain single ``universal'' object. It is hoped that this object
might be used to study the given (infinite) set of individual problems in a more
systematic and perhaps even more efficient way. For example, the universal object may
have interesting additional properties (such as symmetries and
ultrahomogeneity) which in turn can be used to classify finite problems.
In this thesis we shall study embedding-universal and homomorphism-universal
relational structures. 

A {\em relational structure} $\relsys{A}$ is a pair $(A,(\rel{A}{i}:i\in I))$
where $\rel{A}{i}\subseteq A^{\delta_i}$ (i.e. $\rel{A}{i}$ is a $\delta_i$-ary
relation on $A$). The family $(\delta_i: i\in I)$ is called the {\em
type} $\Delta$. The type is usually fixed and understood from the context.
(Note that we consider relational structures only, and no function symbols.) If
the set $A$ is finite we call {\em $\relsys A$ a finite structure}.  We consider only
countable or finite structures.

A {\em homomorphism} $f:\relsys{A}\to \relsys{B}=(B,(\rel{B}{i}:i\in I))$ is a
mapping $f:A\to B$ such that $(x_1,x_2,\ldots, x_{\delta_i})\in
\rel{A}{i}$ implies $(f(x_1),f(x_2),\ldots,f(x_{\delta_i}))\in \rel{B}{i}$ for each $i\in
I$.
 If $f$ is one-to-one then $f$ is called a {\em monomorphism}. A monomorphism
$f$ such $(x_1,x_2,\ldots, x_{\delta_i})\in \rel{A}{i}$ if and only if $(f(x_1),f(x_2),\ldots,f(x_{\delta_i}))\in \rel{B}{i}$ for each $i\in I$ is called an {\em
embedding}. 

The existence of a homomorphism $f:\relsys{A}\to \relsys{B}$ will be also denoted
by $\relsys{A}\to \relsys{B}$.  

An embedding $f:\relsys{A}\to\relsys{B}$ that is onto is called an {\em isomorphism}.
An isomorphism $f:\relsys{A}\to \relsys{A}$ is called an {\em automorphism}.

A relational structure $\relsys{A}$ is a {\em substructure of the relational
structure $\relsys{B}$} when the identity mapping is a monomorphism from $\relsys{A}$
to $\relsys{B}$.  A relational structure $\relsys{A}$ is an {\em
induced substructure of the relational structure $\relsys{B}$} when the identity mapping is an embedding from $\relsys{A}$ to $\relsys{B}$.

Several well-known mathematical structures will be discussed.  We consider
these structures to be special cases of relational structures. However, when
convenient, we use standard graph-theoretic notation.

An {\em undirected graph} (or simply a {\em graph}) is a tuple $G=(V,E)$ such that
$E$ is a set of subsets of $V$ of size $2$. It corresponds to a
symmetric relational structure $\relsys{A}=(A,\relA)$ of type $\Delta=(2)$ defined by $A=V$ and $(u,v)\in
\relA$ if and only if $\{u,v\}\in E$.

A {\em directed graph} is a tuple $G=(V,E)$, such that $E$ is a set of 2-tuples
of $V$.
  This corresponds to a relational structure
$\relsys{A}=(A,\relA)$ of type $\Delta=(2)$ defined by $A=V$ and $\relA=E$.

A directed graph  $G=(V,E)$ is an {\em oriented graph} if and only if there are
no vertices $v_1,v_2\in V$ such that both edges $(v_1,v_2)$ and $(v_2,v_1)$ are
in $E$.  (An oriented graph can be constructed by assigning an orientation  to
every edge of an undirected graph.) 

Finally, a {\em partially ordered set} is pair $(P,\leq_P)$ such that $\leq_P$
is a reflexive, weakly antisymmetric, and transitive binary relation on $P$.  
It corresponds to a relational structure $\relsys{A}=(A,\relA)$
defined by $A=P$ and $\relA=\hbox{$\leq_P$}$.

For a family $\F$ of finite relational structures, denote by
$\Forbi(\F)$ the class of all (finite or countable) relational structures $\relsys{A}$ for which there is no embedding $\relsys{F}\to \relsys{A}$ for any $\relsys{F}\in \F$. 

Similarly, put $\Forb(\F)= \{\relsys{A}:
\relsys{F}\nrightarrow \relsys{A}\mathrm{~for~}\relsys{F}\in \F\}$.

For a given family $\F$ of finite relational structures, the class $\Forb(\F)$
can be equivalently seen as the class $\Forbi(\F')$ where $\F'$ consists of all
structures $\relsys{A}$ such that there is a structure $\relsys{B}\in \F$ and a
homomorphism $\relsys{B}\to\relsys{A}$ that is onto.  If $\F$ consists of
finitely many structures of finite type then $\F'$ is finite too.

We will also use the same notation when speaking about graphs (or directed)
graphs.  For $\F$ a family of countable or oriented graphs, the classes $\Forb(\F)$
and $\Forbi(\F)$ will consist of countable graphs (or directed graphs) only.

In most cases, when considering classes $\Forbi(\F)$ and $\Forb(\F)$, we will be 
interested in families $\F$ consisting of connected structures only.  A
structure $\relsys{A}$ is {\em connected} if for every proper subset $B$ of
vertices of $\relsys{A}$ there is some tuple $(x_1,x_2,\ldots, x_{\delta_i})\in\rel{A}{i}$,
$i \in I$, containing both vertices in $B$ and vertices in $A\setminus B$.

For a given class $\K$ of relational structures we say 
that the structure $\relsys{U}$ is an {\em embedding-universal} (or, simply {\em universal}) structure for $\K$ if $\relsys{U}\in \K$ and every structure $\relsys A\in \K$ can be found as an induced
substructure of $\relsys U$ (or in other words, there exists an embedding from
$\relsys A$ to $\relsys U$).

Similarly we say the structure $\relsys{U}$ is a {\em homomorphism-universal} (sometimes also called {\em hom-universal}) structure for the class $\K$ if $\relsys{U}\in \K$ and for
every structure $\relsys A\in \K$ there exists a homomorphism $\relsys A \to
\relsys U$.

Universal structures can be seen as a representative of the maximum equivalence
class of the following quasi-orders:
$$\relsys{A}\leq_e\relsys{B} \hbox{ if and only if there exists an embedding from } \relsys{A} \hbox{ to } \relsys{B},$$
$$\relsys{A}\leq_h\relsys{B} \hbox{ if and only if there exists a homomorphism from } \relsys{A} \hbox{ to } \relsys{B}.$$

The partial order $\leq_e$ is called the {\em embedding order} and the partial
order $\leq_h$ is called the {\em homomorphism order}.

The notions of embedding-universality and homomorphism-universality have both
been extensively studied and we shall outline many related results and
applications in this chapter.  We shall also concentrate on similarities between
these terms. This is a novel approach since the two notions have been
traditionally studied in different contexts.  In particular, for both notions
of universality we shall try to answer the following questions:
\begin{itemize}
\item Given a class $\K$ of countable relational structures,  is there a universal structure for the class $\K$?
\item Given a relational structure $\relsys{U}$,  is $\relsys{U}$ a universal structure for some class $\K$?
\item What are the known examples of universal structures? 
\end{itemize}
We also outline some of the numerous applications of these notions.

\section{Ultrahomogeneous and generic structures}
\label{Genericsection}

By far the most extensively studied universal structures are the ones satisfying one
additional property:

\begin{defn}
A structure $\relsys A$ is {\em ultrahomogeneous} (sometimes also called {\em homogeneous})
if every isomorphism of two induced finite substructures of $\relsys A$ can be
extended to an automorphism of $\relsys A$. 

A structure $\relsys{G}$ is {\em generic} for the class $\K$ if it is (embedding-)universal for $\K$ and
ultrahomogeneous.
\end{defn}

\begin{figure}
\centerline{\includegraphics{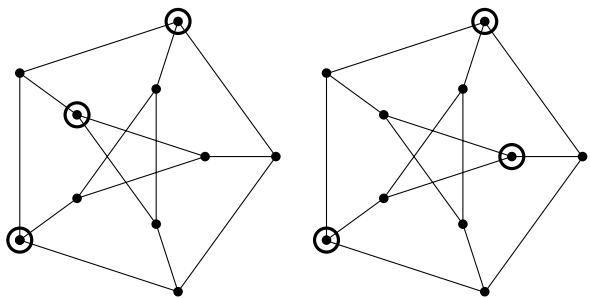}}
\caption{Two kinds of independent set in the Petersen graph.}
\label{petersenhom}
\end{figure}
\begin{figure}
\centerline{\includegraphics{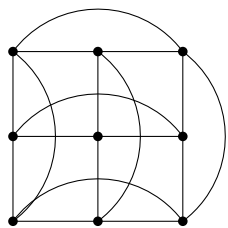}}
\caption{Ultrahomogeneous graph $L(K_{3,3})$.}
\label{homgraf1}
\end{figure}
Ultrahomogeneity of a structure is a very strong property implying a maximal
degree of symmetry.  In particular it implies vertex-transitivity as well as
edge-transitivity.

The strength of ultrahomogeneity can be demonstrated by the example of finite
graphs.  It is easy to see that completely symmetric graphs (complete graphs
and independent sets) are ultrahomogeneous.  Less trivial examples are
difficult to find.  For example, the Petersen graph is known for its symmetry.
It is both vertex-transitive and edge-transitive. In addition, every partial
isomorphism of two 5-cycles in the graph can be extended to an automorphism.
Still it fails to be ultrahomogeneous, because it has two different types of independent
set of size three, as depicted in Figure \ref{petersenhom}.  The first independent set is formed by
neighbors of a single vertex, while the second independent set is not in the neighborhood of any vertex as it can be extended to an independent set of size 4.
Consequently any partial isomorphism mapping the first independent set to the
second cannot be extended to an automorphism.  Still, non-trivial examples of
ultrahomogeneous finite graphs do exist.  Consider the one depicted in Figure
\ref{homgraf1}.

We shall focus almost exclusively on infinite ultrahomogeneous structures.  A
well-known example of a ultrahomogeneous structure is the order of rationals
$(\Q,\leq)$.  The ultrahomogeneity of $(\Q,\leq)$ follows easily from the
definition. Furthermore, every countable linear order can be embedded in
$(\Q,\leq)$ by a monotone embedding (this result is attributed to Cantor).  Consequently, $(\Q,\leq)$ is also the generic structure for the class
of all (countable) linear orders (and all monotone embeddings).

How many structures similar to $(\Q,\leq)$ can we find?
It is important that ultrahomogeneous structures are characterized by properties of finite
substructures.  To show that, we need to first introduce some additional
notions.
\begin{defn}
For a countable relational structure $\relsys{U}$, we denote by $\Age(\relsys{U})$
the class of all finite structures isomorphic to a substructure of $\relsys{U}$.

For a class $\K$ of countable relational structures, we denote by
$\Age(\relsys{\K})$ the class of all finite structures isomorphic to a
substructure of some $\relsys{A}\in \K$.  
\end{defn}

The key property of the age of any ultrahomogeneous structure is described by the
following concept.
\begin{figure}
\centerline{\includegraphics{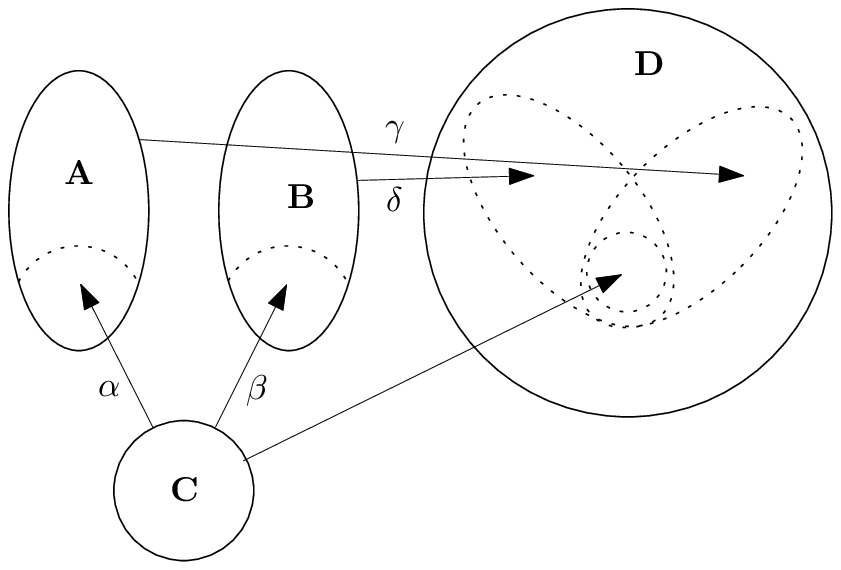}}
\caption{Amalgamation of $(\relsys{A},\relsys{B},\relsys{C}, \alpha, \beta)$.}
\label{amalgamfig0}
\end{figure}
\begin{defn}
\label{amalgamationclassdef}
Let $\relsys{A},\relsys{B},\relsys{C}$ be relational structures, $\alpha$ an embedding of $\relsys{C}$ into $\relsys{A}$, and $\beta$ an embedding of $\relsys{C}$ into $\relsys{B}$.  
An {\em amalgamation of $(\relsys{A}, \relsys{B}, \relsys{C}, \alpha, \beta)$} is any triple $(\relsys{D},\gamma,\delta)$, where $\relsys{D}$ is a relational structure, $\gamma$ an embedding $\relsys{A}\to \relsys{D}$ and $\delta$ an embedding $\relsys{B}\to\relsys{D}$ such that $\gamma\circ\alpha = \delta\circ\beta$.
\end{defn}
Less formally, an amalgamation ``glues together'' the structures $\relsys{A}$ and
$\relsys{B}$ into a single substructure of $\relsys{D}$ such that copies of $\relsys{C}$ coincide.
See Figure \ref{amalgamfig0}.

The age of the generic linear order $(\Q,\leq)$ consists of all
finite linear orders.  It is easy to see that, given finite linear orders
$\relsys{A}$, $\relsys{B}$, $\relsys{C}$ and embeddings $\alpha:\relsys{C}\to
\relsys{A}$ and $\beta:\relsys{C}\to\relsys{B}$, one can construct an amalgamation
$(\relsys{D},\gamma,\delta)$ where $\relsys{D}$ is a linear order on
$|A|+|B|-|C|$ vertices and $\gamma,\delta$ are order-preserving mappings such
that $\gamma\alpha=\delta\beta$ on $\relsys{C}$.

Often the vertex set of structures $\relsys{A}$, $\relsys{B}$ and $\relsys{C}$
can be chosen in such a way that the embeddings $\alpha$ and $\beta$ are identity
mappings.  In this case, for brevity, we will call an amalgamation of
$(\relsys{A},\relsys{B},\relsys{C},\alpha,\beta)$ simply an {\em amalgamation of
$\relsys{A}$ and $\relsys{B}$ over $\relsys{C}$}.  Similarly for an amalgamation
$(\relsys{D}, \gamma, \delta)$ of a given $(\relsys{A}, \relsys{B}, \relsys{C}, \alpha, \beta)$ we are often interested in the structure
$\relsys{D}$ alone.  In this case we shall call the structure $\relsys{D}$ an amalgamation of $(\relsys{A}, \relsys{B}, \relsys{C}, \alpha, \beta)$ (omitting the embeddings $\gamma$ and $\delta$).

The notion of amalgamation gives a lot of freedom in the way the structures can be
combined.  As we shall consider only hereditary (closed for taking induced substructures) classes in all our results, we can assume that $\relsys{D}$
contains only the vertices needed by $\gamma$ and $\delta$.  That is,
$$D=\gamma(A)\cup \delta(B).$$

Sometimes we use more strict versions of amalgamation.  We say that
an amalgamation is {\em strong} when $\gamma(x)=\delta(x')$ if and only if $x\in
\alpha(C)$ and $x'\in \beta(C)$.  Less formally, a strong amalgamation glues together
$\relsys{A}$ and $\relsys{B}$ with an overlap no greater than the copy of $\relsys{C}$ itself.

It is easy to observe that in the case of linear orders a strong amalgamation
is always possible.  However, we can restrict the notion even further. A strong
amalgamation is {\em free} if there are no relations of $\relsys{D}$ spanning both
vertices of $\gamma(A)$ and $\delta(B)$ that are not images of some relations
of structure $\relsys{A}$ or $\relsys{B}$ via the embedding $\gamma$ or $\delta$, respectively.

Obviously, in the case of linear orders a free amalgamation exists only in
very degenerate cases, since new relations need to be introduced between
the vertex sets $\gamma (A\setminus \alpha (C))$ and $\delta (B\setminus \beta(C))$. 

Strong and free amalgamation are important notions used to prove additional
properties of structures.  We shall give numerous examples of uses of 
free amalgamation later.

The ages of ultrahomogeneous structures are described by the following
definition and result.

\begin{defn}
A class $\K$ of finite relational structures is called an {\em
amalgamation class} (sometimes also a \Fraisse{} class) if the following
conditions hold:
\begin{enumerate}
\item ({\em Hereditary property}) For every $\relsys{A}\in \K$ and induced substructure $\relsys{B}$ of $\relsys{A}$ we have $\relsys{B}\in \K$.
\item ({\em Amalgamation property}) 
For $\relsys{A},\relsys{B},\relsys{C}\in \K$ and $\alpha$ an embedding of
$\relsys{C}$ into $\relsys{A}$, $\beta$ an embedding of $\relsys{C}$ into
$\relsys{B}$, there exists $(\relsys{D},\gamma,\delta), \relsys{D}\in \K$, that is an amalgamation of
$(\relsys{A}, \relsys{B}, \relsys{C}, \alpha, \beta)$.
\item $\K$ is closed under isomorphism.
\item $\K$ has only countably many mutually non-isomorphic structures. 
\end{enumerate}
\end{defn}
\begin{thm}[\Fraisse{} \cite{F, Hoges}]
\label{fraissethm}
(a) A class $\K$ of finite structures is the age of a countable ultrahomogeneous structure $\relsys{G}$ if and only if $\K$ is an amalgamation class. 

(b) If the conditions of (a) are satisfied then the structure $\relsys{G}$ is unique up to isomorphism. 
\end{thm}

An amalgamation class is commonly defined with one additional
property (see \cite{Hoges}). A class $\K$ has the {\em joint embedding property} if
for every $\relsys{A}, \relsys{B}\in \K$ there exists $\relsys{C}\in \K$ such
that $\relsys{C}$ contains both $\relsys{A}$ and $\relsys{B}$ as induced
substructures.  We will allow empty structures and assume that there is a unique
empty structure up to isomorphism (as in \cite{CameronP}).  In this setting, the
joint embedding property is just a special case of the amalgamation property.
For given $\relsys{A}$ and $\relsys{B}$ in $\K$ consider the amalgamation of
$(\relsys{A},\relsys{B},\relsys{C},\alpha,\beta)$ where $\relsys{C}$ is an
empty structure.

We should note that in the proof of Theorem \ref{fraissethm} the structure
$\relsys{G}$ is constructed by induction, i.e., by countably many amalgamations and joint
embeddings of structures in the class $\K$.  No explicit description of the
structure is given. For this reason the structure $\relsys G$ is often called a
{\em \Fraisse{} limit} of $\K$ and denoted by $\lim \K$.

We say that structure $\relsys{A}$ is {\em younger} than structure $\relsys{B}$
if $\Age(\relsys{A})$ is a subset of $\Age(\relsys{B})$.
As we shall show, every ultrahomogeneous structure $\relsys{G}$ has the
property that it is (em\-bed\-ding{}-)universal for the class $\K$ of all countable structures younger
than $\relsys{G}$.  It follows that all ultrahomogeneous structures are also
universal and generic for the class $\K$. (Thus we use the letter $\relsys{G}$ to denote
this structure throughout this section.)

Theorem \ref{fraissethm} (\Fraisse{}'s theorem) has many applications in
proving the existence of ultrahomogeneous (and generic) structures. For a given
class $\K$ it is usually trivial to show that $\K$ is hereditary, isomorphism
closed and countable.   Thus the task of showing the existence of
a particular ultrahomogeneous structure usually reduces to giving a method of
constructing amalgamations.

Even very simple amalgamation classes give rise to very interesting
structures.  A popular example of a generic structure is the graph $\Rado$,
generic for the class of all countable graphs.  The class of all finite graphs
is an amalgamation class (and in fact it is an example of an amalgamation class
where a free amalgamation always exists).  The existence of $\R$ follows from Theorem
\ref{fraissethm} and is surprising in itself --- there are uncountably many
non-isomorphic graphs ``packed together'' as induced substructures in the
single countable object.  The graph $\Rado$, known as the {\em Rado graph}, has
several striking properties. We will use it as our primary motivating
example throughout this chapter. 


We have shown how to find an ultrahomogeneous structure. Now let us focus on the
opposite problem. Given a structure $\relsys{G}$, can we tell if it is
ultrahomogeneous?  Instead of showing that $\Age(\relsys{G})$ is an
amalgamation class, it is often easier to use the following alternative
characterization of ultrahomogeneous structures.
\begin{defn}
\label{extensionprop}
A structure $\relsys{A}$ has the {\em extension property} if the following holds.  If
structures $\relsys{B}$ and $\relsys{C}$ are members of the $\Age(\relsys{A})$
such that $\relsys{B}$ is an induced substructure of $\relsys{C}$ and $|C|=|B|+1$,
then every embedding $\varphi:\relsys{B} \to \relsys{A}$ can be extended to an embedding
$\varphi':\relsys{C} \to \relsys{A}$.
\end{defn}
Since the age is always hereditary, it is possible to omit the condition
$|C|=|B|+1$ from Definition \ref{extensionprop}.   This condition is however convenient in
proofs that the given structure $\relsys{G}$ has the extension property.  Observe
that in the case of $(\Q,\leq)$ the extension property is equivalent to
property that for every $a,b\in \Q$ such that $a<b$ there exists $c$ such that
$a<c<b$ (that is, the density of $(\Q,\leq)$) and that there are no maximal or minimal elements in $(\Q,\leq)$.

The extension property can be also seen as a property of a class.  We say that
structure $\relsys{G}$ has the {\em extension property for class $\K$} when
$\relsys{G}$ has the extension property and $\Age(\relsys{G})=\Age(\K)$.

It follows directly from the definitions that all ultrahomogeneous structures
have the extension property.  In the opposite direction, we can show the
following lemma.

\begin{lem}[see e.g. \cite{Hoges}]
  \label{extuniv}
Let $\relsys{G}$ be a structure with the extension property. Then the following holds.
\begin{enumerate}
\item Up to isomorphism, $\relsys{G}$ is uniquely determined by its age (i.e., every countable structure $\relsys{B}$ with the extension property
such that $\Age(\relsys{G})=\Age(\relsys{B})$ is isomorphic to $\relsys{G}$).
\item $\relsys{G}$ is ultrahomogeneous.
\item $\relsys{G}$ is universal for the class of all countable structures younger than $\relsys{G}$.
\end{enumerate}
\end{lem}
\begin{proof}[Proof (sketch)]
We outline an argument proving 1. to demonstrate the model-theoretic tool
known as {\em zig-zag} (or back-and-forth).

Fix relational structures $\relsys{G}$ and $\relsys{B}$ with the extension property such that $\Age(\relsys{G})=\Age(\relsys{B})$.
The procedure to build an isomorphism $\varphi:\relsys G\to \relsys B$ is as
follows:

Assume that the vertices of both $\relsys G$ and $\relsys B$ are natural numbers
(or equivalently enumerate vertices of both vertex sets).  First set
$\varphi(0)=0$.  In the next step, construct an preimage of $1$ in $\relsys B$
using the extension property of $\relsys G$ (that is, find a vertex $v$ in $\relsys G$ such that
the tuple consisting of elements $0$ and $v$ is in $\rel{G}{i}$ if and only if
the corresponding tuple consisting of $0$ and $1$ is in $\rel{B}{i}$). Put
$\varphi(v)=1$.  In the next step choose the first vertex $v'$ in $\relsys G$
such that $\varphi(v')$ is not defined yet and use the extension property of
$\relsys B$ to define an image $v''$.  Put $\varphi(v')=v''$ and continue
analogously.  By alternating $\relsys G$ and $\relsys B$ the process exhausts
both the vertices of $\relsys G$ and of $\relsys B$, thereby constructing an isomorphism. 

To prove 2.~one can build an isomorphism in the same way as above.  The only difference is
that it is necessary to start with a partially given isomorphism instead of
an empty one.

To show 3.~one can use a similar argument: just build the isomorphism in one
direction.
\end{proof}

We illustrate the usefulness of Lemma \ref{extuniv} by proving the following famous result:

\begin{thm}[\Erdos{} and R\'enyi \cite{ErdosReni}]
\label{randomgr}
There is a countable graph $\Rado'$ with the property that a countable random
graph (edges chosen independently with probability $1\over 2$) is almost surely
isomorphic to $\Rado'$.
\end{thm}

Countable random graphs can be constructed inductively by
adding vertices one at a time. When vertex $v$ is added, the edge to any
older vertex $v'$ is added independently of all the other vertices with probability $1\over 2$.

The theorem claims that this random process of constructing a graph almost surely
leads to the same result after countably many steps.  Compared to the finite case,
this result is very counter-intuitive.  In fact Theorem \ref{randomgr} allows us to speak
about ``the countable random graph.''

We use the extension property to prove Theorem \ref{randomgr} and moreover show
that $\Rado'$ is generic for the class of all countable graphs and thus is
isomorphic to $\Rado$.  The class of all finite graphs a is a very simple class
allowing the following convenient reformulation of the extension property.

\begin{fact}
\label{grafext}
A graph $G=(V,E)$ has the {\em extension property} for the class of all finite graphs
if for every $J,D$ finite disjoint subsets of $V$, there exists a vertex $v\in V$
joined by an edge to every vertex in $J$ and no vertex in $D$.
\end{fact}

\begin{proof}[Proof of Theorem \ref{randomgr}]
We consider random graphs on vertex sets formed by the set $\N$ of natural numbers.

First we show that with probability 1 a countable random graph has the
extension property.  To apply Fact \ref{grafext} we need to prove that, for
every choice of $J$ and $D$ (disjoint and finite subsets of $\N$), with
probability 1 there is vertex $v$ joined to every vertex in $J$ and no vertex
in $D$. 

First fix the choice of $J$ and $D$ and we prove that with probability 0
there is no such vertex $v$.
The probability that a given vertex $v$ will be joined to every vertex in $J$
and no vertex in $D$ is
 $${1\over {2^{|J|+|D|}}}=c.$$ Since edges are
constructed independently, the probability that $k$ vertices will all fail to
satisfy the extension property is $(1-c)^k$. Since there are infinitely many
choices of the vertex $v$, the probability that all fail is
$\lim_{k\to\infty}(1-c)^k = 0$.

It follows that for every feasible choice of $J$ and $V$, a countable random
graph fails with probability 0 and there are only countably many choices. By
standard probabilistic reasoning (that the union of countably many null sets is
null) it follows that a countable random graph fails to have the extension property
with probability 0.  

By Lemma \ref{extuniv} we know that with probability 1 a countable random graph
is generic for the class of all countable graphs.  By Theorem \ref{fraissethm}
there is up to isomorphism a unique such graph $\Rado$.  We put
$\Rado'$=$\Rado$.
\end{proof}

The correspondence between random structures and generic structures can be
carried beyond the class of undirected graphs.  Precisely the same argument can
be used for oriented and directed graphs. See also \cite{V} for the
construction of the random metric space and proof of its equivalence with the
generic metric space.  In general it can be shown that if $\relsys{G}$ is a
countable ultrahomogeneous relational structure then almost all countable
structures younger than $\relsys{G}$ are isomorphic to $\relsys{G}$ (see
\cite{cameron}).

\subsection{Known ultrahomogeneous structures}
\label{classection}




%
It is natural to ask which ultrahomogeneous structures exist.  This leads to
the celebrated classification programme of ultrahomogeneous structures that we
outline now. 

The first important result in the area was the classification of ultrahomogeneous
partial orders (given by Schmerl in 1979 \cite{Schmerl}, see also \cite{Cherlin} for a simple proof).

\begin{thm}[Schmerl \cite{Schmerl}]
\label{schmerlthm}
Every countable ultrahomogeneous partial order is isomorphic to one of the
following: 

\begin{enumerate}
\item A (possibly infinite) antichain.
\item A (possibly infinite) union of copies of the ordered rationals $(\Q,\leq)$, elements in distinct copies being incomparable (antichain of chains).
\item A union indexed by $(\Q,\leq)$ of antichains $A_q$ all of the same (finite or countably infinite) size, and ordered by $x\leq y$ if and only if there is some $q<r,x\in A_q$ and $y\in A_r$ (chain of antichains). 
\item The generic partial order for the class of all countable partial orders.
\end{enumerate}
\end{thm}

The classification of all ultrahomogeneous graphs was given by Lachlan and
Woodrow in 1984 \cite{lachlan}.  This classification is a lot more difficult result than
Theorem \ref{schmerlthm}.  The reason is that graphs are very free structures
and the increased freedom leads to more possibilities on how an
ultrahomogeneous structure can be constructed. Given the complexity of
the arguments, the resulting statement is surprisingly simple.

\begin{thm}[Lachlan and Woodrow \cite{lachlan}]
\label{lach}
Every countable ultrahomogeneous undirected graph is isomorphic to one of the
following:
\begin{enumerate}
\item Finite cases:
\begin{enumerate}
\item 5-cycle,
\item the graph $L(K_{3,3})$ depicted in Figure \ref{homgraf1},
\item finitely many disjoint copies of a complete graph $K_r$,
\item complements of graphs listed in $(c)$.
\end{enumerate}
\item The disjoint union of $m$ complete graphs of size $n$, where $m,n\leq \omega$ and at least one of $m$ or $n$ is $\omega$.
\item Complements of graphs listed in 2.
\item The generic graph for the class of all countable graphs not containing $K_n$ for a given $n\geq 3$.
\item Complements of graphs listed in 4.
\item The Rado graph $\Rado$ (generic graph for the class of all countable
graphs).
\end{enumerate}
\end{thm}

In 1987 Lachlan \cite{La1} continued this line of research with the classification of ultrahomogeneous
tournaments.

Recall that a {\em tournament} is an oriented graph obtained by assigning an orientation for each edge of an undirected complete graph.


Denote by $S(2)$ the following tournament. The vertices of $S(2)$ are all
rational numbers  $q$ with an odd denominator, $0\leq q<1$.  There is an edge
$(a,b)$ in $S(2)$ if and only if either $a<b<a+\frac 1 2$ or $a-1<b<a-\frac 1 2$.

\begin{figure}
\centerline{\includegraphics{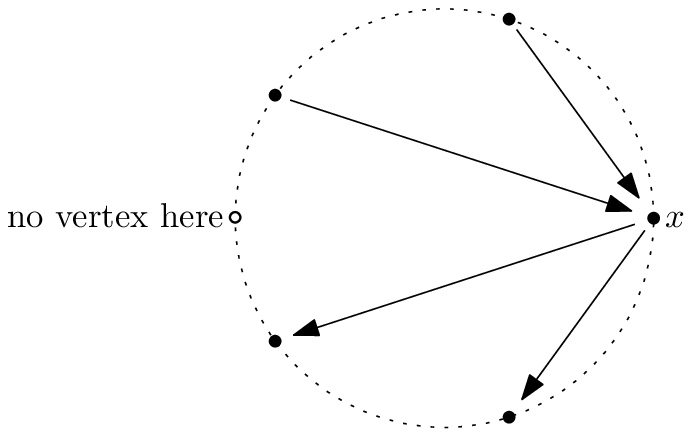}}
\caption{Relations among vertex $x$ in the dense local order $S(2)$.}
\label{s3d}
\end{figure}
Intuitively, the tournament $S(2)$ can be seen as a circle with edges forming
a dense countable set of chords.  The orientation is chosen in such a way that
shorter chords are oriented clockwise.  For this reason $S(2)$ is also called
the {\em dense local order} (see Figure~\ref{s3d}).  

\begin{thm}[Lachlan \cite{La1}]
\label{gentournaments}
Every countable ultrahomogeneous tournament is isomorphic to one of the
following:
\begin{enumerate}
\item Finite cases:
\begin{enumerate}
\item the singleton one-point tournament,
\item the oriented cycle of length 3, $\overrightarrow{C}_3$.
\end{enumerate}
\item $(\Q,\leq)$ (rationals with the usual ordering).
\item The dense local order $S(2)$.
\item The generic tournament for the class of all countable tournaments.
\end{enumerate}
\end{thm}

Finally, all ultrahomogeneous directed graphs were classified by Cherlin in 1998
\cite{Cherlin}.  This is a very complex result which we do not state in detail.
Undirected graphs, tournaments and partial orders are just special cases of directed graphs. 

It is important to notice that, unlike the previous cases, there are
uncountably many non-isomorphic ultrahomogeneous directed graphs. For every set $\F$ of finite tournaments (considered as directed graphs), the class
$\Forb(\F)$ has a generic directed graph. There are uncountably many choices of $\F$ which lead to different classes $\Forb(\F)$.

There are a number of other characterization theorems for special classes of
relational structures, but a full characterization of ultrahomogeneous structures
of more complex types than $\Delta=(2)$ seems out of reach of current
techniques.  We only make a simple generalization of the argument showing
that there are uncountably many directed graphs. 

\begin{defn}
The relational structure $\relsys{A}$ is {\em irreducible} if for every pair of vertices
$v_1,v_2$ of $\relsys{A}$ there is a tuple $\vec{v}\in \rel{A}{i}$ (for some $i$)
such that both $v_1$ and $v_2$ are in the tuple $\vec{v}$.
\end{defn}
In the other words, the structure $\relsys{A}$ is irreducible if it cannot
be constructed as a free amalgamation of two proper substructures.

It is easy to observe that for $\F$ a family of finite irreducible structures,
the age of the class $\Forbi(\F)$ is an amalgamation class (the class allows free amalgamations)
and thus there is always a generic structure for the class $\Forbi(\F)$.
The same holds for the class $\Forb(\F)$.


\subsection{Ultrahomogeneous structures and Ramsey theory}
\label{ramseysection}
Generic structures have been extensively studied in modern model theory (see
\cite{Hoges} or recent survey \cite{Macpherson}) and have applications to (the classification of) dynamical systems,
group theory and Ramsey theory.  We outline briefly the last connection
to also demonstrate the utility of the classification programme.

Given a set $X$ and a natural number $k$, denote by ${X}\choose{k}$ the set of all $k$-element
subsets of $X$.  The classical Ramsey theorem can be stated as follows:
\begin{thm}
\label{ramseythm}
For every choice $p,k,n$ of natural numbers there exists $N$ with the following property: If $X$ is a set of
size $N$ and ${X \choose p}=\mathcal A_1\cup \mathcal A_2\cup\ldots\cup \mathcal A_k$ is any partition of the set of $p$-subsets of $X$ then there exists $i,1\leq i\leq k$, and $Y\subseteq X$, $|Y|\geq n$, such that ${Y\choose p}\subset \mathcal A_i$.
\end{thm}

Variants of the Ramsey theorem exist for different kinds of structure.
Consider for example the formulation for finite vector spaces \cite{GLR}.
\begin{thm}
For every finite field $\F$ and for every choice $p,k,n$ of natural numbers there exists $N$ with the following property:
If $\relsys{A}$, $\relsys{B}$, $\relsys{C}$ are vector spaces of dimensions $p,n$ and $N$ respectively and if ${\relsys{C}\choose \relsys{A}}=\mathcal A_1\cup \mathcal A_2\cup\ldots\cup \mathcal A_k$
is any partition of the set $\relsys{C}\choose \relsys{A}$ of all $p$-dimensional vector subspaces of $\relsys{C}$ then there exists $i,1\leq i\leq k$, and a subspace $\relsys{B}'$
of $\relsys{C}$, $\mathrm{dim} \relsys{B}'=\mathrm{dim} \relsys{B}=n$, such that ${\relsys{B}'\choose \relsys{A}}\subset \mathcal A_i$.
\end{thm}
The formulation is strikingly similar to the formulation of Theorem
\ref{ramseythm}.  This is the case for the variant of the theorem for linearly
ordered relational structures too.
  We say that a structure $\relsys{A}$ is
{\em linearly ordered} if it is endowed with an additional linear order on
$A$.
\begin{thm}[Ne\v set\v ril, R\"odl \cite{NR1}]
For every choice of natural number $k$, of type $\Delta$, and of linearly ordered structures $\relsys{A},\relsys{B}\in \Rel(\Delta)$, there exists a structure $\relsys{C}\in \Rel(\Delta)$ with the following property: For every partition ${\relsys{C}\choose \relsys{A}}=\mathcal A_1\cup \mathcal A_2\cup\ldots\cup \mathcal A_k$ there exists $i,1\leq i\leq k$, and a substructure $\relsys{B}'\in {\relsys{C}\choose \relsys{B}}$ such that ${\relsys{B}'\choose \relsys{A}}\subset \mathcal A_i$.
\end{thm}
In this case we mean by ${\relsys{B}\choose \relsys{A}}$ the class of all substructures $\relsys{A}'$ of $\relsys{B}$ which are isomorphic to $\relsys{A}$.

The similarity between formulations of different variants of Ramsey's theorem
motivates the following notion of a Ramsey class (see e.g. \cite{N3,N4}).
%

Let $\cal K$ be a class of objects which is isomorphism-closed and endowed with subobjects.
Given two objects $\relsys A, \relsys B \in {\cal K}$ we denote by ${\relsys B}\choose{\relsys A}$ the set of all subobjects $\relsys A'$ of $\relsys B$
which are isomorphic to $\relsys A$. We say that the class $\cal K$ has the {\it $\relsys A$-Ramsey property} if the following statement holds:
For every positive integer $k$ and for every $\relsys B \in {\cal K}$ there exists $\relsys C \in {\cal K}$
such that $\relsys C \longrightarrow (B)_k^\relsys A$. 
Here the last symbol ({\em Erd\H os--Rado partition arrow}) has the following meaning:
For every partition ${{\relsys C}\choose{\relsys A}} = {\cal A}_1 \cup {\cal A}_2 \cup\ldots \cup{\cal A}_k$
there exists $\relsys B' \in {{\relsys C}\choose{\relsys B}}$ and an $i, 1 \leq i \leq k$ such that ${{\relsys B'}\choose{\relsys A}} \subset {\cal A}_i$.

In the extremal case that a class $\cal K$ has the $\relsys A$-Ramsey property for every one of its objects $\relsys A$ we say that
$\cal K$ is a {\it Ramsey class}.
The notion of a Ramsey class is highly structured and in a sense it is the top
of the line of the Ramsey notions (``one can partition everything in  the any number
of classes to get anything ultrahomogeneous'', see also  \cite{N3,N4}).
Consequently there are not many (essentially different) examples of Ramsey
classes known. 

The key connection for us is the following result that relates two seemingly
unrelated things: Ramsey classes and ultrahomogeneous structures. 

%


\begin{thm}[Ne\v set\v ril \cite{HN-Posets}]
\label{ramseyhomo}
Let ${\cal K}$ be a Ramsey class (with ordered embeddings as subobjects) 
which is hereditary, isomorphism-closed  and with the joint embedding property.
Then ${\cal K}$ is the age of a generic (ultrahomogeneous and universal) structure.
\end{thm}

This allows one to use known results about ultrahomogeneous structures (in the
cases when their classification programme have been completed) and to check whether
the corresponding classes (i.e., their ages) are Ramsey. This classification
programme of Ramsey classes was proposed by Ne\v set\v ril \cite{N3}.

The ultrahomogeneous structures listed in Section \ref{classection} were all
examined one by one and it was either proved or disproved that their age forms
an Ramsey class.  Proving the fact that a given age of an ultrahomogeneous structure
is a Ramsey class is often a difficult task (using ad hoc techniques) and thus
we omit the details.  See \cite{N3} for a full survey of these results
from which we give only a compact summary.

For a class $\K$, denote by $\K_c$ the class of all complements of graphs in $\K$. For an undirected
graph we get the following result.
\begin{thm}[Ne\v set\v ril \cite{N1}]
The following are all Ramsey classes of (undirected) graphs:
\begin{enumerate}
\item The class $\{K_1\}$.
\item The class of all complete graphs.
\item The class of all (linearly ordered) disjoint unions of complete graphs.
(With the complete graphs forming intervals of the linear order.)
\item The classes of all (linearly ordered) finite graphs not containing $K_n$
for a given $n\geq
3$.
\item The class $\mathcal K_c$ for each of the above classes.
\item The class of all (linearly ordered) finite graphs.
\end{enumerate}
\end{thm}

All these classes are ages of the ultrahomogeneous graphs listed in Theorem
\ref{lach}.  The converse does not hold.  Not every age of an ultrahomogeneous graph produces a Ramsey class.  In particular, the only
finite case is the singleton graph. In the infinite case disjoint unions of
finitely many complete graphs fail to be Ramsey.  See \cite{N1} for details.

For the case of ordered tournaments we get:
\begin{thm}[Ne\v set\v ril \cite{N3}]
The following are Ramsey classes of ordered tournaments:
\begin{enumerate}
\item The class $\{K_1\}$.
\item The class of all linear orders (transitive tournaments).
\item The class of all (linearly ordered) tournaments.
\end{enumerate}
\end{thm}
Comparing this with Theorem \ref{schmerlthm}, we see that the ultrahomogeneous
tournament $S(2)$ and oriented cycle $\overrightarrow{C}_3$ fail to produce a Ramsey class. 

For partial orders we have:
\begin{thm}[Ne\v set\v ril \cite{N3}]
The following are Ramsey classes of partially ordered sets:
\begin{enumerate}
\item $\{K_1\}$ ($K_1$ here means the singleton partially ordered set.)
\item The class of all finite linear tournaments.
\item The class of all chain-sums of finite antichains.
\item The class of all ordered antichains.
\item The class of all ordered finite partially ordered sets.
\end{enumerate}
\end{thm}

The classification of oriented graphs provided by \cite{Cherlin} has also been
discussed, but it has not been fully determined which classes are Ramsey and which are
not.  We present just an extension of the observation about the classes
$\Forbi(\F)$, where $\F$ is a family of finite irreducible structures. However
this extension is a deep result:
\begin{thm}[Ne\v set\v ril, R\"odl \cite{NR5}]
For a given family $\F$ of finite irreducible relational structures, the class of
all linearly ordered structures $\relsys{A}$ such that $\relsys{A}\in \Forbi(\F)$ is Ramsey.
\end{thm}

Kechris, Pestov, and Todor\v cevi\v c \cite{KPT} relate the extreme amenability (of subgroups of
$S_\omega$) to purely combinatorial problems of Ramsey classes.  
Several permutation groups have been shown to be extremely amenable using
combinatorial examples of Ramsey classes (such as the class of all finite
graphs, the class of all finite partial orders or the class of Hales-Jewett cubes)
and thus some further examples of extremely amenable groups have been found \cite{GW,KPT,Pestov98,Pestov02}.
This also provoked some combinatorial questions which led to new examples of Ramsey
classes:
\begin{enumerate}
\item
In particular, Ne\v set\v ril proved that all (ordered) finite metric spaces
form a Ramsey class \cite{N5}, see also \cite{De Prisco}. This
gives \cite{KPT} a simpler new proof that $Aut(\Urysohn)$ ($\Urysohn$ is the Urysohn space) is an extremely amenable group
(originally shown in \cite{Pestov02}).
\item
More recently Farah and Solecki \cite{Farah} isolated in the context of extreme
amenability a new ``group-valued'' Hales-Jewett Theorem in the context of L\'evy groups.
\end{enumerate}



\section{Classes with universal non-homogeneous structures}
\label{univsectioncl}
There are many classes $\K$ with the (embedding-)universal structure $\relsys{U}$
but no generic structure. Take, for example, the class $\Forbi(S_3)$ of all
countable graphs containing no vertex of degree 3. ($S_3$ stands for a graph
forming a star with a single center vertex and 3 vertices connected to the
central vertex by an edge.)  Denote by $C_l$ the graph consisting of a single
cycle of length $l$.  Because the class $\Forbi(S_3)$ consists of paths and
cycles only, we can build a universal graph $\Rado_{\Forbi(S_3)}$
for the class $\Forbi(S_3)$ as the union of infinitely many copies of each
graph $C_l$, for $l=3,4,\ldots$, and infinitely many copies of a doubly
infinite path.

The graph $\Rado_{\Forbi(S_3)}$ is not generic because it contains components
of different sizes.  It is easy to see that any universal graph for the class
$\Forbi(S_3)$ must have the same components as $\Rado_{\Forbi(S_3)}$ and thus
there is no generic graph for $\Forbi(S_3)$.

It may seem that the existence of a universal structure, requiring as it does a
much weaker condition than the existence of a generic structure, is very
often satisfied.  The study of classes containing a universal structure was
however motivated by negative results (see \cite{Hajnal
Komjatk,CherlinKomjath}). For example, the class $\Forbi(C_l)$ of all countable
graphs not containing $C_l$ does not contain a universal graph for any $l>3$.  

Cherlin and Shelah \cite{CherlinShelah} generalized the example of the class
$\Forbi(S_3)$ to the notion of a {\em near-path} --- a graph tree which is not
a path but is obtained by attaching one edge with one additional vertex to a
path.  They show that for a given finite graph tree $T$ there is a universal
graph for the class $\Forbi(T)$ if and only if $T$ is a path or a near-path.
This problem was for several years open as the Tallgren tree conjecture and also
supports the fact that classes with universal structures are 
relatively rare.

Similarly to the definition of an amalgamation class, one may ask what
properties a class $\K$ has to satisfy to contain a universal structure.  This is
an open problem.  In \cite{CherlinShelahShi} and later in \cite{restruct}
the following variant of this question is posed:
\begin{quote}
Is there an algorithm which determines for each finite set $\F$ of finite connected
``forbidden'' subgraphs whether the corresponding universal graph exists,
for the class $\Forbi(\F)$?
\end{quote}
It is suggested in \cite{restruct} that the problem may well be undecidable.
However there are a number of deep and interesting related results.

Before stating some of those results let us the observe similarities to generic structures.
In many cases we can find universal structures that are not ultrahomogeneous,
but have properties that resemble those guaranteed by ultrahomogeneity.
In particular they are $\omega$-categorical:

\begin{defn} A countably infinite structure is called {\em
$\omega$-categorical} (sometimes also {\em countably-categorical}, {\em $\aleph_0$-categorical} or {\em categorical}) if all
countable models of its first order theory are isomorphic.
\end{defn}

To see how the notion of ultrahomogeneity and $\omega$-categoricity are
related, we use the following characterization given by Engeler \cite{Engeler},
Ryll-Nardzewski \cite{Nardzewsky} and Svenonius \cite{Svenonius}.
\begin{thm}
For a countable first order structure $\relsys{A}$, the following conditions are equivalent:
\begin{enumerate}
\item $\relsys{A}$ is $\omega$-categorical.
\item The automorphism group of $\relsys{A}$ has only finitely many orbits on $n$-tuples, for every $n$.
\end{enumerate}
\end{thm}

It follows that for relational structures of finite type $\omega$-categoricity
is really a weaker variant of ultrahomogeneity.  In an automorphism group of an
ultrahomogeneous structure the number of orbits on $n$-tuples is determined by
the number of induced substructures of size $n$. In an $\omega$-categorical
structure the number of orbits on $n$-tuples can be arbitrarily large, but
finite.

In Section \ref{Genericsection} we outlined that every ultrahomogeneous
structure $\relsys{A}$ is universal for the class of all countable structures
younger than $\relsys{A}$.  (This follows from the extension property.)  In the
case of $\omega$-categorical structures the same holds, but the proof is more
complicated. 

\begin{thm}[Cameron \cite{cameron2, cameron}]
If $\relsys{A}$ is $\omega$-categorical then it is universal for the class of all countable
structures younger than $\relsys{A}$.
\end{thm}

How common are $\omega$-categorical universal objects?  Our introductory
example of the graph $\Rado_{\Forbi(S_3)}$ universal for the class $\Forbi(S_3)$ is clearly not
$\omega$-categorical. It has infinitely many orbits of $1$-tuples.  

This example is however quite a special one and a universal graph exists only
``by accident.'' The class $\Forbi(S_3)$ contains only countably many
non-isomorphic connected graphs and our universal graph is the union of all of
them.  The classes forbidding a near-path (but not a path) are similar cases
(see \cite{CherlinShelah}).  It seems, from the lack of known examples, that
classes with a universal structure but without an $\omega$-categorical universal structure are even more rare.

Consider the class $\Forb(C_5)$ of all graphs not containing a cycle of
length $3$ or $5$.
It is a non-trivial fact that there is a universal graph for
$\Forb(C_5)$ (see \cite{Mekler}).  We show an explicit construction of such
an $\omega$-categorical graph in Chapter \ref{Forbchapter}.  It
is however easy to see that there is no generic graph for the class
$\Forb(C_5)$: two vertices $u$ and $v$ not connected by an edge can be connected
either by a path of length $2$ or a path of length $3$. Connecting the vertices
$u$ and $v$ by both a path of length $2$ and of length $3$ would form a $5$-cycle or a
$3$-cycle.  It follows that there are at least two types of independent sets of
size two in every universal graph for ${\Forb(C_5)}$: those that are connected by a path of
length 2 and those that are connected by a path of length 3. This is not
possible in an ultrahomogeneous structure.

The amalgamation property condition of Theorem \ref{fraissethm} can be relaxed
to a sufficient condition for the existence of an $\omega$-categorical universal
structure for a given class as shown by Covington \cite{Covington}.
\begin{defn}
Let $\K$ be a class of countable relational structures, $\relsys{A},\relsys{B},
\relsys{C}\in \K$, $\alpha$ an embedding of $\relsys{C}$ to $\relsys{A}$ and
$\beta$ an embedding of $\relsys{C}$ to $\relsys{B}$.  A tuple $(\relsys{A},
\relsys{B}, \relsys{C}, \alpha, \beta)$ such that there is no 
$(\relsys{D}, \gamma, \delta)$, $\relsys{D}\in \K,$ that is an amalgamation of $(\relsys{A}, \relsys{B}, \relsys{C}, \alpha, \beta)$ is called an {\em (amalgamation) failure of
$\K$}.

The failure $(\relsys{A}', \relsys{B}', \relsys{C}', \alpha', \beta')$ is a {\em
subfailure} of the failure $(\relsys{A}, \relsys{B}, \relsys{C}, \alpha, \beta)$ if
there are embeddings $a:\relsys{A'}\to \relsys{A}$, $b:\relsys{B'}\to \relsys{B}$ and $c:\relsys{C'}\to
\relsys{C}$ such that for every $x\in C'$ we have
$a(\alpha'(x))=\alpha(c(x))$ and
$b(\beta'(x))=\beta(c(x))$.

The class $\K$ has a {\em local amalgamation failure} if and only if there is a finite set $S$ of
failures of $\K$ such that for every failure $(\relsys{A}, \relsys{B},
\relsys{C}, \alpha, \beta)$ of $\K$ there exists a failure $(\relsys{A}',
\relsys{B}', \relsys{C}', \alpha', \beta') \in S$ that is a subfailure of
$(\relsys{A}, \relsys{B}, \relsys{C}, \alpha, \beta)$.
\end{defn}
\begin{figure}
\centerline{\includegraphics{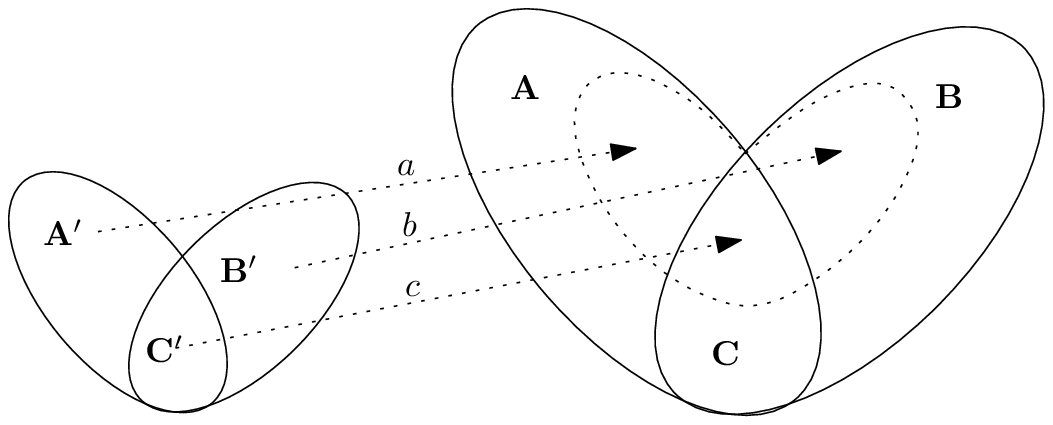}}
\caption{The failure $(\relsys{A}', \relsys{B}', \relsys{C}', \alpha', \beta')$ is subfailure of the failure $(\relsys{A}, \relsys{B}, \relsys{C}, \alpha, \beta)$.}
\label{subamalgamfig0}
\end{figure}

The notion of a subfailure is very intuitive and depicted in Figure
\ref{subamalgamfig0}.

The main result of \cite{Covington} is the following:

\begin{thm}[Covington \cite {Covington}]
\label{Convigtonthm} Let $\F$ be a family of finite structures of finite
type.  Assume that $\Age(\Forbi(\F))$ has the joint embedding property and
that it has local failure of amalgamation. Then $\Forbi(\F)$ contains an $\omega$-categorical
structure $\relsys{U}$ that is universal for $\Forbi(\F)$.
\end{thm}

Theorem \ref{Convigtonthm} can be seen as variant of Theorem \ref{fraissethm}.
Because type $\Delta$ is finite, there are only countably many mutually
non-isomorphic structures in $\Age(\Forbi(\F))$.  The class $\Age(\Forbi(\F))$ is
also obviously hereditary and isomorphism-closed.  Local failure of
amalgamation is a weaker variant of the amalgamation property. It can also be
easily seen that for a family $\F$ of finite connected structures, the class
$\Age(\Forbi(\F))$ always has the joint embedding property. 

\begin{figure}
\centerline{\includegraphics{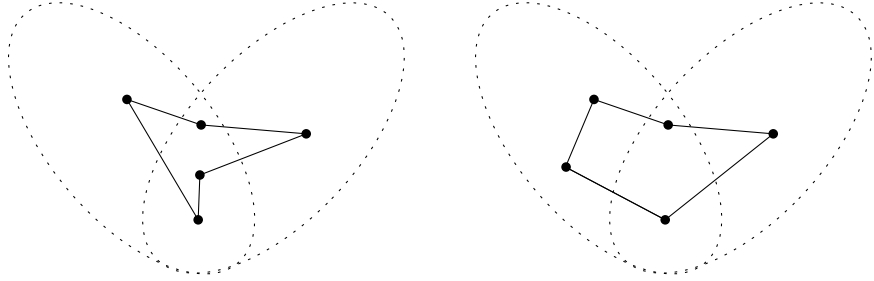}}
\caption{Amalgamation failures of the class $\Forb(C_5)$.}
\label{5cyclefig}
\end{figure}
Theorem \ref{Convigtonthm} directly applies to the class $\Forb(C_5)$: any amalgamation failure must contain a subfailure consisting of
a $3$-cycle or a $5$-cycle.  The class of graphs not containing a $3$-cycle is
an amalgamation class and thus has no amalgamation failures.  It follows that
every amalgamation failure must contain a $5$-cycle. There are just two
amalgamation failures consisting of $5$-cycle alone, depicted in Figure \ref{5cyclefig}.
The second is a subfailure of the first and thus the set $S$ of failures can consist of
a single failure.

Theorem \ref{Convigtonthm} is proved by imposing an additional structure on relational
structures in $\K$ by adding new relations and using Theorem
\ref{fraissethm} to obtain the generic structure in the extended language (so-called homogenization).  As a result, the universal structure $\relsys{U}$ is
$\omega$-categorical and inherits many other properties of the generic structure it
is created from. 

\begin{figure}
\centerline{\includegraphics{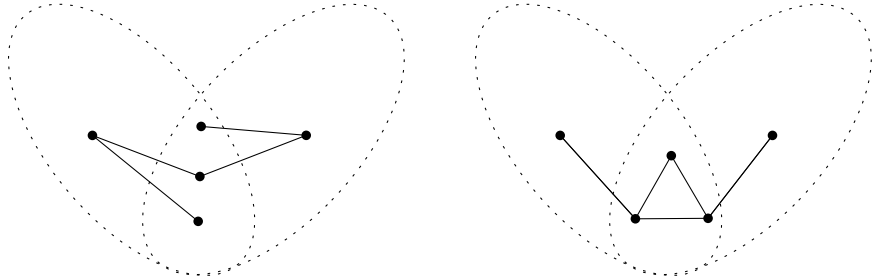}}
\caption{Amalgamation failures of the class of graphs not containing an induced path of length 3.}
\label{obrazky n free}
\end{figure}
The local amalgamation failure property, despite its natural definition, seems
difficult to apply and thus Theorem \ref{Convigtonthm} has found surprisingly little use.
Covington in \cite{Covington2,Covington} shows that the class of all graphs omitting
isomorphic copies of a path of length 3 ($N$-free graphs) has the local
amalgamation failure property.  Amalgams in the set $S$ are depicted in Figure
\ref{obrazky n free}. The class is homogenized by adding a ternary relation to
distinguish one vertex from each triple carrying a null or complete induced
subgraph.

More examples of classes with a universal graph were found in
\cite{Komjath,Mekler,CherlinShi,NRR,NTardif}, mostly by means of amalgamation
techniques developed for the particular structure.

Necessary and sufficient conditions for the existence of an $\omega$-categorical
universal graph are known for classes $\Forbi(\F)$ where $\F$ is a finite family
of finite connected graphs.  This is a deep model-theoretic result of Cherlin,
Shelah and Shi \cite{CherlinShelahShi}.  To state the result we need to
introduce the notion of algebraic closure.

\begin{defn}
Let the relational structure $\relsys{A}$ be an induced substructure of a relational
structure $\relsys{B}$. We say that the relational structure $\relsys{A}$ is {\em
existentially complete in $\relsys{B}$} if every existential statement $\psi$
which is defined in $\relsys{A}$ and true in $\relsys{B}$ is also true in
$\relsys{A}$.

For a class $\K$ of relational structures, we say that $\relsys{A}\in \K$ is {\em existentially complete in class $\K$} if $\relsys{A}$ is existentially complete for every structure $\relsys{B}\in \K$ such that $\relsys{A}$ is an induced substructure of $\relsys{B}$.

Let $\relsys{A}$ be an existentially complete relational structure in a class $\K$, set $S\subseteq A$ and vertex $a\in A$.
We say that $a$ is {\em algebraic over $S$} (in $\relsys{A}$) if there is an existential formula $\psi(x,\overline{a})$ with $\overline{a}\in S$ such that the set $\{a'\in A:\varphi(a',\overline{a})\}$ is finite and contains $a$. 

We write {\em $\mathrm{acl}_\relsys{A}(S)$ (algebraic closure)} for the set of $a\in S$ that are algebraic over $S$. We say $S$ is {\em algebraically closed in $\relsys{A}$} if $\mathrm{acl}_\relsys{A}(S)=S$.
\end{defn}

The notion of algebraic closure is a complicated concept. See
\cite{CherlinShelahShi} for further analysis of its behavior.  Informally, a
vertex is in the algebraic closure of a given set if the number of vertices of
the same type must be finite in any structure in $\K$.  For example, in the case
of the class $\Forbi(S_3)$, the algebraic closure of a set is the union of its
connected components. 

%
This use of algebraic closure was motivated by its use in earlier proofs of
the non-existence of a universal structure $\relsys{U}$ for certain special classes
$\K$ \cite{Hajnal Komjatk,CherlinKomjath}.  As was indicated in those proofs, when algebraic closure is not
bounded it is possible to find uncountably many structures in $\K$ with the
property that their isomorphic copies in any structure $\relsys{U}\in \K$ have
just a small overlap.  From this it follows that a universal graph cannot be
countable.  The main result of Cherlin, Shelah and Shi is the following theorem,
showing that this is the only obstacle to the existence of an $\omega$-categorical
universal graph for the class $\Forbi(\F)$.

Note that \cite{CherlinShelahShi} states the results in the context of 
graphs.  Most of the results can be extended to relational structures. 

\begin{thm}[Cherlin, Shelah, Shi \cite{CherlinShelahShi}]
\label{cherlinthm}
Let $\F$ be a finite set of connected graphs.
Denote by $T^{*}_{\F}$ the theory of all existentially complete graphs in $\Forbi(\F)$.
Then the following conditions are equivalent:
\begin{enumerate}
\item $T^{*}_{\F}$ is $\omega$-categorical.
\item When $A$ is a finite subset of a model $\relsys{M}$ of $T^{*}_{\F}$, $acl_\relsys{M}(A)$ is finite.
\end{enumerate}
These conditions imply:

\begin{enumerate}
\item[3.] The class $\Forbi(\F)$ contains an $\omega$-categorical universal graph.
\end{enumerate}
\end{thm}

See \cite{CherlinShelahShi} for a list of families $\F$ where the existence of
a universal graph $\Forbi(\F)$ is known as well as a proof of the finiteness of
the corresponding algebraic closure.  We only show the following example to illustrate
how rare and irregular those cases can be.

\begin{figure}
\centerline{\includegraphics{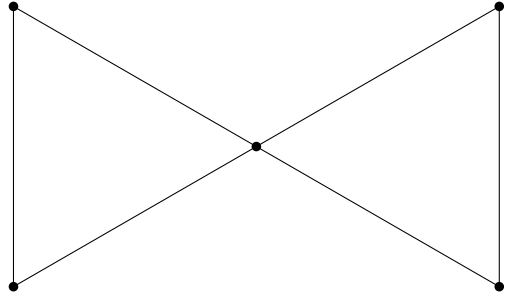}}
\caption{The ``bow tie'' graph $B_{3,3}$.}
\label{kravata}
\end{figure}
Consider the graph $B_{n,m}$ constructed by taking the union of complete graphs $K_n$ and $K_m$, where a single
vertex from $K_n$ is identified with a single vertex of $K_m$ and other vertices
in the union are disjoint. As was shown by Komj\'ath \cite{Komjath}, there is a
universal graph in the class $\Forbi(B_{3,3})$ ($B_{3,3}$ is a ``bow tie graph'', see Figure
\ref{kravata}).  The result was generalized by Cherlin and Shi in \cite{CherlinShi2}: there is
a universal graph in the class $\Forbi(B_{m,n})$ if and only if $\min(m,n)\leq 5$
and $(m,n)\neq (5,5)$. This shows how delicate the conditions for the
existence of a universal structure can be.

The algebraic closure is related to amalgamation, as shown by the following lemma.
\begin{lem}[Cherlin, Shelah, Shi \cite{CherlinShelahShi}]
\label{cherlinlema}
Let $\F$ be a finite family of finite graphs, $G$ an existentially complete graph in $\Forbi(\F)$, and $A\subseteq G$. The following statements are equivalent:
\begin{enumerate}
\item $A$ is not algebraically closed in $G$.
\item There is some graph $F\in \F$, $F'$ an induced subgraph of $G$ and a homomorphism $F\to F'$ so that $F$ embeds in the free amalgamation of $|F|$ copies of $F'$ over $A$.
\end{enumerate}
\end{lem}

We turn our attention to the following corollary that captures what we consider
to be the most interesting case.

\begin{corollary}[Cherlin, Shelah, Shi \cite{CherlinShelahShi}]
\label{1thm}
For every finite family $\F$ of finite connected graphs there is a universal
graph for the class $\Forb(\F)$.
\end{corollary}

\begin{proof}[Proof of Corollary \ref{1thm}]

Take $G$ an existentially complete graph in $\Forb(\F)$. We prove that every finite
subset of vertices of $G$ is algebraically closed in $G$. 

Fix $A$ a finite subset of vertices of $G$ and assume that $A$ is not
algebraically closed in $G$.  By Lemma \ref{cherlinlema} there is a graph $F\in
\F$ such that there is a homomorphism $F\to F'$ and $F'$ is an induced subgraph of
$G$. It follows that $G\notin \Forb(\F)$, a contradiction.

Because every finite subset of $G$ is algebraically closed in $G$, by Theorem
\ref{cherlinthm} there is a universal graph for $\Forb(\F)$.
\end{proof}

Corollary \ref{1thm} stated for relational structures follows also from Theorem
\ref{Convigtonthm}.  It is easy to see that for a finite family $\F$ of finite
connected relational structures the classes $\Forb(\F)$ do have local failure of
amalgamation. Every amalgamation failure contains a homomorphic copy of one of
the forbidden structures.

However both of these approaches (based on \cite{Covington} and \cite{CherlinShelahShi}) use model-theoretic tools.
In Chapter \ref{Forbchapter} we give a new combinatorial proof of this result.
Similarly to the proof of Theorem {\ref{Convigtonthm}}, we also extend the
language by new relations to build an amalgamation class based on the class
$\Forb(\F)$.  We use the notions of lifts and shadows:

Fix type $\Rel(\Delta), \Delta=(\delta_i: i\in I)$, $I$ finite.  Now let
$\Delta'=(\delta'_i:i\in I')$ be a type containing type $\Delta$. (By this we
mean $I\subseteq I'$ and $\delta'_i=\delta_i$ for $i\in I$.) Then every
structure $\relsys{X}\in \Rel(\Delta')$ may be viewed as a structure
$\relsys{A}=(A,(\rel{A}{i}: i\in I))\in \Rel(\Delta)$ together with some
additional relations $\rel{X}{i}$ for $i\in I'\setminus I$. 

We call $\relsys{X}$ a {\em lift} of $\relsys{A}$ and $\relsys{A}$ is called
the {\em shadow} (or {\em projection}) of $\relsys{X}$. 

In Chapter \ref{Forbchapter} we prove the following:
\begin{thm}
\label{mainthmvuvodu}
Let $\F$ be a countable set of finite connected relational structures (of finite type $\Delta$).  Then there exists a class $\Lifts$ of lifts (relational structures of type $\Delta'$) such that the shadow of any $\relsys{X}\in
\Lifts$ is in $\Forb(\F)$. Moreover $\Age(\Lifts)$ is an amalgamation class and
there is a generic structure $\relsys{U}$ for $\Lifts$. The shadow of $\relsys{U}$ is an $\omega$-categorical universal structure for the class $\Forb(\F)$.

For $\F$ finite, there is a finite class of finite connected lifts $\F'$ such that
$\Lifts=\Forbi(\F')$.
\end{thm}

This result gives an explicit construction of $\omega$-categorical universal graphs
for the classes $\Forb(\F)$. Moreover the class $\Lifts$ has a relatively easy
description allowing us to examine those structures for several special classes
$\F$. This has several combinatorial consequences. Particularly we show the connection to homomorphism dualities and Urysohn spaces.

\subsection{On-line embeddings}
In the previous section we gave characterization theorems similar to the
\Fraisse{}
theorem for the existence of universal structures with special properties.  Now we concentrate on the
opposite problem of proving the universality of a known structure $\relsys{U}$.
We introduce a notion for universal structures similar to the extension property for homogeneous structures (and to \Fraisse{}-Ehrenfeucht games).

By an {\em on-line representation} of a class $\K$ of relational structures in a
structure $\relsys{U}$, we mean that one can construct an embedding
$\varphi:\relsys{A}\to \relsys{U}$ of any structure $\relsys{A}$ in the class $\K$ under the
restriction that the elements of $\relsys{A}$ are revealed one by one.  The on-line
representation of a class of a relational structure can be considered as a game
between two players $A$ and $B$ (usually Alice and Bob).  Player $B$ chooses a
structure $\relsys{A}$ in the class $\K$, and reveals the elements of
$\relsys{A}$ one by one to player $A$ ($B$ is a bad guy).  Whenever a vertex
of $x$ of $\relsys{A}$ is revealed to $A$, the relations among $x$ and
previously revealed elements are also revealed.  Player $A$ is required to
assign a vertex $\varphi(x)$---before the next element is revealed---such that
$\varphi$ is an embedding of the substructure induced by  $\relsys{U}$ on the
already revealed vertices of $\relsys{A}$.  Player $A$ wins the game if he
succeeds in constructing an embedding $\varphi$.  The class $\K$ of relational structures
is on-line representable in the structure $\relsys{U}$ if player $A$ has
a winning strategy. 

On-line representation (describing a winning strategy for $A$) is a convenient way
of showing the universality of a relational structure for a given hereditary class
$\K$ of countable relational structures. In particular, it transforms the problem of
embedding countable structures into the finite problem of extending an existing
finite partial embedding to the next element. 

Universal structures in general may or may not allow an on-line representation.
Consider the example of a universal graph $U$ for the class $\Forb(C_5)$ of
all countable graphs not containing a homomorphic image of the cycle on 5 vertices.
It is easy to see that there is no winning strategy for player $A$.

Player $B$ may first embed two vertices $v_1$ and $v_2$ not joined by an
edge.  The images provided by player $A$ in the graph $U$ may be
connected by the path of length 3 or the path of length 2.  Since no two
vertices in $U$ are connected by both the path of length 3 and the
path of length 2, player $B$ may continue by asking $A$ to embed the missing
path and win the game.

This argument is just a variant of the argument we gave about the lack of a generic graph
for the class $\Forb(C_5)$ and thus it may seem that on-line embeddings give a
little help in showing universality of a structure for classes without a generic
one.  As shown in Chapter $\ref{Forbchapter}$, the rules of the game can be
modified to get a variant with winning strategy for $A$.  All we need is to ask
$B$ to announce also the existence of paths of length 2 as well as the
existence of paths of length 3 connecting the already revealed vertices of the
graph. Such a modified game still implies universality. 

In Part II, on-line embedding will be the key tool to show universality of
explicit partial orders.  In this case we do have a generic structure for the
class of all partial orders, yet we are interested in various universal, but
not generic, examples.  The extension property is a stronger form of on-line
representation and thus we get the following simple lemma:

\begin{lem}
Let $\K$ be a class of countable relational structures that contains
a generic structure for $\K$. For every $\relsys{U}\in \K$ the
following conditions are equivalent:
\begin{enumerate}
\item There is an on-line representation of $\K$ in the relational structure $\relsys{U}$.
\item The relational structure $\relsys{U}$ is universal for $\K$.
\end{enumerate}
\end{lem}

\section{Homomorphism-universal structures}
\label{homounivsection}
Homomorphism-universality is a weaker notion than embedding-universality: if a
class $\K$ of countable relational structures contains an embedding-universal
structure $\relsys{U}$, the same structure is also homomorphism-universal.

Often also the following notion of universality is considered: for a given
class $\K$ of relational structures we say that the structure $\relsys{U}$ is
an {\em monomorphism-universal} (sometimes also {\em weakly-universal})
structure for $\K$ if $\relsys{U}\in \K$ and every structure $\relsys A\in \K$
can be found as (possibly non-induced) substructure of $\relsys U$.

For countable structures the problems of the existence of monomorphism{}- and em\-bed\-ding-universal
structures coincide.  This has been proved in \cite{CherlinShelahShi}.
 On the
other hand, the notions homomorphism-universal and embedding-universal are
clearly different.  Consider as an example the class of all planar graphs.  In
this case the finite homomorphism-universal graph exists (the graph $K_4$ is
homomorphism-universal by virtue of 4-color theorem) while neither an embedding- nor
a monomorphism-universal graph exists (see \cite{Hajnal Komjatk}). However in
many cases we can prove that not only does an embedding-universal graph not exist, but also that there is no homomorphism-universal graph.
 This is the case for example with forbidding
$C_4$ -- the cycle of length 4. 

New interesting questions arise when we focus on homomorphism-universality
alone.  In particular, it is interesting to ask whether there exists a finite
homomorphism-universal structure $\relsys{D}$.

A {\em finite duality} (for structures of given type) is any equation
$$\Forb(\F)=\{\relsys{A}: \relsys{A}\to \relsys{D}\}$$ where $\relsys{D}$ is
a finite relational structure and $\F$ is a finite set of finite relational
structures. $\relsys{D}$ is called the {\em dual of $\F$}, the pair $(\F,\relsys{D})$ the
{\em dual pair}.  For this case we also say that the class $\Forb(\F)$ has
finite duality.

\begin{figure}
\centerline{\includegraphics{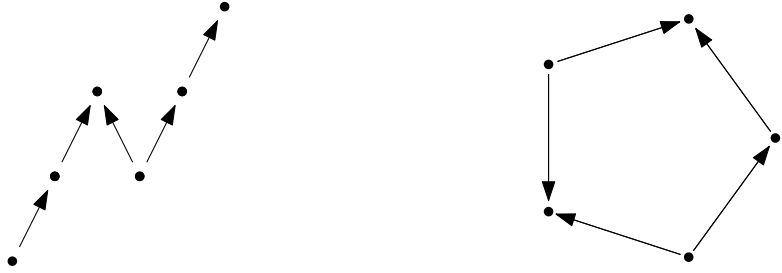}}
\caption{A dual pair.}
\label{dualpair}
\end{figure}
An example of a dual pair is depicted in Figure \ref{dualpair}.

Given that homomorphism-universality is related to embedding-universality it
may be surprising that there is very simple characterization of such families
$\F$.

\begin{thm}[Ne\v set\v ril, Tardif \cite{NTardif}]
For every type $\Delta$ and for every finite set $\F$ of finite relational trees there exists
a dual $\Delta$-structure $\relsys{D}_\F$. Up to homomorphism-equivalence there
are no other dual pairs.
\end{thm}

A (relational) tree can be defined as follows:
\begin{defn}
\label{reltree}
The {\em incidence graph} $ig(\relsys{A})$ of a relational structure $\relsys{A}$ is the                      
bipartite graph with parts $A$ and                                  
$$                                                                              
Block(A) = \{ (i,(a_1, \ldots, a_{\delta_i})) : i \in I,                       
(a_1, \ldots, a_{\delta_i}) \in \rel{A}{i} \},                                      
$$                                                                              
and edges $\{a, (i,(a_1, \ldots, a_{\delta_i}))\}$                       whenever
$a \in (a_1, \ldots, a_{\delta_i})$. (Here we write                             
$x \in (x_1, \ldots, x_n)$ when there exists an index $k$                       
such that $x = x_k$; $Block(A)$ is a multigraph.)

A relational structure $\relsys{A}$ is called a {\em tree} when                 $ig(\relsys{A})$ is a graph tree (see e.g. \cite{MatN}).
\end{defn}

Finite dualities also correspond to the only first-order-definable CSP (Constraint Satisfaction Problems, Atserias \cite{Atserias}, Rossman \cite{Rossman},
see e.g. \cite{HellNese}).

A number of constructions of duals are known \cite{NTardif2}.  We give a new
construction in Chapter \ref{Forbchapter}.
Strengthening Theorem \ref{mainthmvuvodu} for the special case of relational
trees, we show that the lifted class $\Lifts$ can be constructed in a way
extending the original type by unary relations only.  Such a lift is called
called a {\em monadic lift}.  New unary relations can be seen as colors of
vertices and the generic structure for the class $\Lifts$ can be retracted by
identifying  vertices of the same color thereby giving a finite dual.  We obtain the following theorem (proved in Chapter \ref{Forbchapter}):

\begin{thm}
For $\F$ a finite family of finite relational trees, there exists a structure
$\relsys{U}\in \Forb(\F)$ that is embedding-universal for the class
$\Forb(\F)$. Moreover, there is a finite structure $\relsys{D}\in \Forb(\F)$ that
is a retract of $\relsys{U}$ and is homomorphism-universal for the class
$\Forb(\F)$.
\end{thm}

In Chapter \ref{Forbchapter} we give the necessary arity of the lift for a given
family $\F$ and show the existence of families $\F$ of structures that are not
relational trees and still there is a lifted class $\Lifts$ with monadic lifts.
It follows that the existence of a finite dual is stronger than the existence
of an embedding-universal structure that is the shadow of its monadic lift.

We have given a characterization of finite families $\F$ having a dual $\relsys{D}$.
In the opposite direction, we can ask when a given finite $\relsys D$ is
homomorphism-universal for some $\Forb(\F)$, $\F$ finite. Or, equivalently, whether $\relsys D$ is the
dual of some finite set $\F$.  An explicit characterization of all structures that are
duals was given by Larose, Loten and Tardif in \cite{charakterizacedualu}.  Note also that Feder
and Vardi \cite{FederVardi} provided a characterization of all structures
$\relsys D$ that are homomorphism-universal for $\Forb(\F)$, where $\F$ is an infinite family
of trees. 

As finite dualities are characterized by these results we can look at the
notion of restricted dualities.  Here we want the duality to hold only for
structures from a given class $\K$. 

For a finite family $\F$ of finite structures and a structure $\relsys D$, we say that $\F$ and
$\relsys D$ establish a {\em $\K$-restricted duality} if the following
statement holds for every $\relsys{A}\in \K$: 
$$(\relsys{F}\not \to \relsys{A}) \hbox {, for every } \relsys{F}\in \F \hbox{, if and only if} (\relsys{A}\to \relsys{D}).$$
%
In the other words, $\relsys {D}$ is an upper bound of the set $\Forb(\F)\cap \cal
\K$ in the homomorphism order.

As the extremal case we make the following definition.
\begin{defn}
We say that the class of relational structures {\em $\K$ admits all restricted dualities} if,
for any finite set of connected structures $\F=\{\relsys{F_1},\relsys{F_2},\ldots, \relsys{F_t}\}$,
there exists a finite structure $\relsys{D}^{\K}_\F$ such that $\relsys{F_i}\not\to \relsys{D}^{\K}_\F$ for $i=1,2,\ldots, t$, and, for all $\relsys{A}\in \K$,
$$(\relsys{F}_i\not\to \relsys{A}), i=1,2,\ldots,t \hbox{, if and only if} (\relsys{A}\to \relsys{D}^{\K}_\F).$$
\end{defn}
The definition can be motivated by the following example (cf.
\cite{NesetrilDualities}).  Gr\"otzsch's theorem (see for example \cite{Grotzsch})
says that every triangle-free planar graph is 3-colorable.  In the language of
homomorphisms this says that for every triangle-free planar graph $G$ there is
a homomorphism $G\to K_3$.  Or in the other words, $K_3$ is an upper bound in
the homomorphism order for the class $\mathcal P_3$ of all planar triangle-free graphs.  The fact that $K_3\notin \mathcal P_3$ motivates a question
(first formulated in \cite{NT}): Is there a yet smaller bound? The
answer, which may be viewed as a strengthening of  Gr\"otzsch's  theorem, is
positive: there exists a triangle-free 3-colorable graph $H$ such that $G\to H$
for every $G\in \mathcal P_3$. 

Examples of classes with all restricted dualities include: planar graphs,
proper minor-closed classes, bounded expansions.  Such classes were recently
characterized by Ne\v set\v ril and Ossona de Mendez \cite{NesetrilPatris}
using limit objects.


\section{Explicit models of universal structures}
\label{explicitmodels}
Let us return to the example of the generic graph $\Rado$ for the class of all countable
graphs.

The existence of the graph $\Rado$ was  proved in Section
\ref{Genericsection} by applying \Fraisse{}'s theorem and also by showing its
isomorphism to the random graph $\Rado'$. However this construction gives little
insight into the structure of the graph itself. 

An explicit representation (or model) $\Rado_\N$ of the generic graph $\R$ was
first given by Rado \cite{Rado} (and this is the reason why $\Rado$ is known as {\em the Rado graph}):

\begin{enumerate}
\item
the vertices of $\Rado_\N$ are all finite 0--1 sequences $(a_1,a_2,\ldots,a_t), t\in \N$,
\item a pair $\{(a_1,a_2,\ldots, a_t),(b_1,b_2,\ldots,b_s)\}$ form an edge of $\Rado_\N$ if and only if $b_a=1$ where $a=\sum^t_{i=1} a_i2^i$ (or vice versa).
\end{enumerate}

It is not difficult to show that $\Rado_\N$ has the extension property and thus
is isomorphic to $\Rado$.  This remarkably simple explicit description of
$\Rado$ has motivated further ones, such that the following:
\begin{enumerate}
  \item $\Rado$ is isomorphic to the following graph $\Rado_\in$:
	the vertices of $\Rado_\in$ are all finite sets (in some countable model of set
        theory) with edges of the form $\{A,B\}$ where either $A\in B$ or $B\in A$.
  \item $\Rado$ is isomorphic to the following graph $\Rado_{QR}$: the vertices of $\Rado_{QR}$ are all prime natural numbers $x \equiv 1 \mod 4$ with $xy$ forming an edge if and only if $({{x}\over{y}})=+1$.
\end{enumerate}

There are other explicit constructions (see the excellent survey by Cameron \cite{cameron}, see also \cite{Cherlin2,La2}). It is remarkable that all these
seemingly unrelated constructions  define the same graph $\Rado$ and moreover the equivalence can be shown as a trivial application of the extension property.  We give a proof of this fact only for the case of $\Rado_\in$. The other
constructions are entirely analogous.

\begin{thm}
\label{runiv}
  The graph $\Rado_\in$ has the extension property for the class of finite undirected graphs.
  Thus $\Rado_\in$ is isomorphic to the generic undirected graph $\R$.
\end{thm}
\begin{proof}
  Let $J$ and $D$ be two disjoint finite sets of vertices of $\Rado_\in$.
  
 To satisfy the extension property (in the formulation given in Fact \ref{grafext}) for $J$ and $D$ we are looking for a vertex
$X$ of $\Rado_\in$ such that
\begin{enumerate}
\item $Y\in X$ for every $Y\in J$,
\item $Y\notin X$ for every $Y\in D$.
\end{enumerate}
 It suffices to put $X=J\cup\{v\}$ with $v$ chosen in a way so that $v\notin D$ and
$J\cup\{v\} \notin Y$ for all $Y\in D$.
  Thus $\Rado_\in$ has the extension property and thus it is generic for the class of
  all countable undirected graphs.
\end{proof}

In this work we study explicit representations of universal structures.  We
call those representations {\em finite presentations}.  Here we broadly
interpret the notion of finite presentation as a succinct representation of an infinite set.  By succinct we mean that the elements
 are finite models with
relations induced by ``compatible mappings'' (such as homomorphisms) between
the corresponding models. This intuitive definition suffices as we are
interested in the (positive) examples of such representations.

\eject
``Concise representations'' of finite structures have been studied from the complexity
point of view for graphs \cite{NarorNan, Turan} and partially ordered sets
\cite{lposet,NP}.

The notion of finite presentation is also related to the concepts of constructive mathematics. 
Our constructions are essentially
constructive (see \cite{constructiveurysohn} for a reformulation of our
construction of the rational Urysohn space in the context of constructive mathematics).
The notion of a finite presentation is stronger than the notion of constructivity.
We want the elements of a structure to be defined in a simple way
that is independent of other elements. Similarly, the relations among the elements are to be defined purely based on a knowledge of those elements participating in the relation.
In particular, a construction given via repeated amalgamation and joint embedding
by the proof of \Fraisse{}'s theorem per se is not a finite presentation.

This work is divided into two parts.

In Part I we consider known ultrahomogeneous structures as provided by
the classification programme outlined in Section \ref{Genericsection} and look for
their finite presentations.  We represent all ultrahomogeneous undirected
graphs (Chapter~\ref{graphschapter}), all partial orders (Chapter~\ref{homposetchapter}) and all ultrahomogeneous tournaments (Chapter~\ref{generictournaments}). The main contribution of Part I is a finite
presentation of an ultrahomogeneous partial order related to Conway's surreal
numbers (Chapter \ref{homposetchapter}) and a finite presentation of the rational
Urysohn metric space (Chapter \ref{homprostorchapter}).

As a result it may seem that ultrahomogeneous structures are very likely to have a finite
presentation.  Even our informal definition of a finite presentation makes it
possible to show that this is not always the case: as discussed in Section
\ref{classection} there are uncountably many different ultrahomogeneous oriented graphs,
but there are only countably many structures with a finite presentation (in
a proper axiomatization of the term). Thus it is not possible to find a finite
presentation for every ultrahomogeneous oriented graph.

In Part II we take the opposite approach.  We look for well-known finitely
presented structures and try to prove their (embedding-)universality. Motivated by
the difficulties in finding a finite presentation of the generic partial order and the lack of
many examples of universal partial orders, we develop a technique of embedding
a universal partial order into new structures, which leads to a number of new finite
presentations, given in Chapter~\ref{posetschapter}.

In the main result of Chapter \ref{cestickychapter} we focus on the homomorphism order
of relational structures and show that even very restricted classes of
relational structures (rooted oriented paths) produce a universal partial
order.

\section{Summary}
Several results in this thesis have been published or accepted for publication.
Chapters \ref{graphschapter}, \ref{homposetchapter} and \ref{generictournaments}
on finite presentations of generic structures are based on the paper
\cite{HN-Posets}. Some of the constructions, including the finite presentation
of the generic partial order, were first given in the author's diploma thesis
\cite{diplomka}.

Chapter \ref{homprostorchapter} giving a finite presentation of the rational Urysohn space is based on the paper \cite{metric}.

Chapter \ref{posetschapter} is based on the paper \cite{HN-Posets2} (accepted for publication).

\eject
Chapter \ref{cestickychapter} presents a new proof of the universality of rooted
oriented paths ordered by homomorphism.  Universality of oriented trees ordered
by homomorphism was the main result of \cite{HN-trees,diplomka}.
The papers \cite{HN-paths,diplomka} contain an earlier and more complex proof of
universality of oriented paths ordered by homomorphism.  The new proof was
accepted for publication as part of \cite{HN-Posets2}.

Finally, Chapter \ref{Forbchapter} combines results of the as yet unpublished papers
\cite{paper2} and \cite{paper3}.

\part{Finite presentations of ultrahomogeneous structures}

\chapter{Ultrahomogeneous graphs}
\label{graphschapter}
It is our aim to show that ultrahomogeneous structures are likely to be
finitely presented.  Intuitively it is plausible that a high degree of symmetry (ultrahomogeneity)
leads to a ``low entropy'' and thus in turn perhaps to a concise
representation.

We begin by developing representations similar to the Rado's representation of the graph $\Rado$ and gradually
progress to more complicated cases.  First we show easy
examples---representations of all ultrahomogeneous undirected graphs.  This
will serve also as a warm-up for more involved representations of partial
orders, the rational Urysohn space and ultrahomogeneous tournaments presented
in subsequent chapters. 

\section{Ultrahomogeneous undirected graphs}

We follow the classification programme outlined in Section \ref{Genericsection}
(in particular the classification of ultrahomogeneous graphs is given by
Theorem \ref{lach}).  All finite structures are obviously finitely presented.
By Theorem \ref{lach} a countably infinite ultrahomogeneous undirected graph is
isomorphic to one of the following graphs:
\begin{enumerate}
\item The disjoint union of $m$ complete graphs of size $n$, where $m,n\leq \omega$ and at least one of $m$ or $n$ is $\omega$ (or the complement of such a graph).
\item The \Fraisse{} limit of the class of countable graphs not containing $K_n$ for given $n\geq 3$ (or the complement of such a graph).
\item The Rado graph $\Rado$.
\end{enumerate}
A finite presentation for 1.~is an easy exercise.  For the Rado graph we gave
several representations in Section \ref{explicitmodels}. When a structure is
finitely presented, its complement is also finitely presented and thus it
remains to look for a finite presentation of generic graphs not containing $K_n$.  We use
similar tools as for the Rado graph and its representation $\R_\in$ (see Section
\ref{explicitmodels}), just in a more general setting.

%
  Throughout this chapter we shall use the following notation for universal graphs.
  By $\R_\K$ we denote the ultrahomogeneous universal (i.~e. generic) graph for the class $\K$ of undirected graphs (if it exists).
  By $\RO_{\overrightarrow \K}$ we denote the ultrahomogeneous universal graph for a class $\overrightarrow{\K}$ of directed graphs (if it exists).

  Recall that by $\Forbi(G)$ we denote the class of all countable graphs not containing graph $G$ as an induced subgraph.

  We now construct graphs $\R_{\Forbi(\Kk),\in}$, $k\geq 3$ which are isomorphic to the generic graph $\R_{\Forbi(\Kk)}$.
  The construction of graph $\R_{\Forbi(\Kk),\in}$, $k\geq 3$, is an
  extension of the construction of $\R_\in$.
  (Recall that a finite set $S$ is called {\em complete} if for any $X,Y\in S$, $X\neq Y$ either $X\in Y$ or $Y\in X$.)

\begin{defn}
 The undirected graph $\R_{\Forbi(\Kk),\in}$, $k\geq 3$, is constructed as
follows:
\begin{enumerate}
\item  The vertices of
  $\R_{\Forbi(\Kk),\in}$ are all (finite) sets which do not contain a complete subset with
  $k-1$ elements.
\item  Two vertices of $S$ and $S'$ form an edge of
  $\R_{\Forbi(\Kk),\in}$ if and only if either $S\in S'$ or $S'\in S$.
\end{enumerate}
\end{defn}

Thus $\R_{\Forbi(\Kk),\in}$ is the restriction of the graph $R_\in$ to the class of all sets without a complete subset of size $k-1$.

\begin{thm}
  \label{rforbkk}
  $\R_{\Forbi(\Kk),\in}$ does not contain an isomorphic copy of $\Kk$ and has the
extension property for the class $\Forbi(\Kk)$. Consequently   $\R_{\Forbi(\Kk),\in}$ is the 
generic graph for the class $\Forbi(\Kk)$.
\end{thm}
\begin{proof}
  $\R_{\Forbi(\Kk),\in}$ does not contain $\Kk$: For a contradiction, let us suppose that
  $V_1$, $V_2,\dots,V_k$ are vertices of a complete graph.  Without loss of generality we
  may assume that $V_i\in V_{i+1}$ for each $i=1,2,\ldots,k-1$.  Since $\Kk$ is
  a complete graph, $V_i\in V_k$ for each $i=1,2,\ldots,k-1$.
  It follows that $\{V_1,\ldots,V_{k-1}\}$
  is a prohibited complete subset of $k-1$ elements. Thus $V_k$ is not a vertex of
  $\R_{\Forbi(\Kk)}$.

  To show the extension property of $\R_{\Forbi(\Kk),\in}$ we use the following
reformulation of the extension property for class $\Forbi(\Kk)$, similar to Fact \ref{grafext}:
\begin{quote}
For every $J,D$ finite disjoint subsets of vertices of $\R_{\Forbi(\Kk),\in}$, there exists either a vertex $X\in V$
joined by an edge to every vertex in $J$ and no vertex in $D$ or there is
$J'\subseteq J$ such that the graph induced on $J'$ by $\R_{\Forbi(\Kk),\in}$ is isomorphic to $K_{k-1}$.
\end{quote}

Fix $J$ and $D$ finite disjoint subsets of $V$ and assume that there is no
$J'\subseteq J$ such that the graph induced on $J'$ by $\R_{\Forbi(\Kk),\in}$
is isomorphic to $K_{k-1}$.  Similarly as in proof of Theorem \ref{runiv}
we put $X=J\cup\{v\}$ with $v$ chosen in such a way so that $v\notin D$, 
$J\cup\{v\} \notin Y$ for all $Y\in D$, and (additionally) so that $v\cup J$ is empty.

It is easy to verify that $X$ is a vertex of $\R_{\Forbi(\Kk),\in}$.
\end{proof}
To summarize the representations in this section, we have the following corollary.
\begin{corollary}
\label{ografp}
All ultrahomogeneous undirected graphs are finitely presented.
\end{corollary}

\section{Ultrahomogeneous directed graphs}
Directed graphs are among the most complicated structures for which the classification
programme had been solved.  As already discussed in Section \ref{explicitmodels}, it
is too ambitious to ask for finite representations of all directed graphs:
there are only countably many finite representations, but uncountably many
ultrahomogeneous directed graphs.  

In this section we restrict ourselves to simple cases that help us to develop
the background for representations of partial orders.  Additional examples of
finite presentations of directed graphs will be given in subsequent chapters;
both partially ordered sets and tournaments are special cases of directed
graphs.

First we construct the directed graph $\RO$ generic for
the class of all countable directed graphs.  
In the rest of this chapter we will use a fixed standard countable model of set theory $\Model$ containing a
single atomic element $\apple$. This allows us to use the following definition of the ordered pair.

\begin{defn}
  For every set $M$ we put
  $$M_L=\{A:A\in M,\apple\notin A\},$$
  $$M_R=\{A:A\cup{\{\apple\}}\in M,\apple\notin A\}.$$

  For any two sets $A$ and $B$ we shall denote by $\HE 0,A,B$
  the set $$A\cup \{M\cup\{\apple\}:M\in B\}.
  $$
\end{defn}

For any set $M$ not containing $\apple$ the following holds: $\HE 0,M_L, {M_R} = M$.
Thus for the model $\Model$, the class of sets not containing $\apple$ represents the
universe of  recursively nested ordered pairs.

\begin{defn}
\label{rodef}
  The directed graph $\RO_\in$ is constructed as follows:
\begin{enumerate}
 \item The vertices of $\RO_\in$ are all finite sets in $\Model$ not containing $\apple$.
 \item $(M,N)$ is an edge of $\RO_\in$ if and only if either $M\in N_L$ or $N\in M_R$.  
\end{enumerate}
\end{defn}

\begin{thm}
  \label{ROuniv}
The directed graph $\RO_\in$ is isomorphic to $\RO$ (the generic directed graph for the class of all countable directed graphs).
\end{thm}
\begin{proof}
  We proceed analogously to the proof of Theorem \ref{runiv}.
  To show that $\RO_\in$ has the extension property let 
  $M_-$, $M_+$ and $M_0$ be three disjoint sets of vertices, where $M_0\cap (M_-\cup M_+)$ is empty.
  We need to find vertex $M$ with following properties:
  \begin{enumerate}
     \item[I.] For each $X\in M_-$ there is an edge from $X$ to $M$.
     \item[II.] For each $X\in M_+$ there is an edge from $M$ to $X$.
     \item[III.] For each $X\in (M_-\cup M_+\cup M_0)$ there are no other edges from $X$ to $M$ or $M$ to $X$ than the ones given by I. and II.
  \end{enumerate}
  Fix any $$x\notin \bigcup_{m\in M_-\cup M_+\cup M_0}m.$$
  Obviously, the vertex $M=\HE 0,M_-\cup\{x\},{M_+}$ has the required properties I.,II.,III..
\end{proof}

Consequently, generic graphs (for both the class of all undirected and the class of all directed graphs) are finitely presented.  We can extend these presentations to other ultrahomogeneous structures.  
We illustrate this by 
the construction of the generic directed graphs $\RO_{\Forbi(T),\in}$ not containing a given finite tournament $T$.
This is slightly more technical (although it parallels the undirected case).

Put $T=(V,E)$ and for each $v\in V$ put 
$$L(v)=\{v'\in V:(v',v)\in E\},$$
$$R(v)=\{v'\in V:(v,v')\in E\}.$$
(Observe that $L(v)\cup R(v)=V-\{v\}$.)

The vertices of $\RO_{\Forbi(T),\in}$ are sets $M$ which satisfy the following condition $C_v(M)$ (for each $v\in V$).
\medskip

  \noindent
  $C_v(M)$:
\begin{quote}
  There are no sets $X_{v'}$, $v'\in L(v)\cup R(v)$, satisfying the following
  \begin{enumerate}
    \item[I.] $X_{v'}\in M_L$ for $v'\in L(v)$,
    \item[II.] $X_{v'}\in M_R$ for $v'\in R(v)$,
    \item[III.] for every edge $(v',v'')\in E$, $v',v''\in L(v)\cup R(v)$,
	        either $X_{v'}\in(X_{v''})_L$ or $X_{v''}\in (X_{v'})_R$.
  \end{enumerate}
\end{quote}
In the other words, $C_v(M)$ holds if the sets $X_{v'}$, $v'\in L(v) \cup R(v)$,
do not represent the tournament $T-\{v\}$ in $\RO_\in$.
\begin{defn}
  Denote by $\RO_{\Forbi(T),\in}$ the directed graph $\RO_\in$ restricted to the
  class of all sets $M$ which satisfy the condition $C_v(M)$ for every $v\in V$.
\end{defn}

\begin{thm}
\label{ROFORB}
  $\RO_{\Forbi(T),\in}$ is isomorphic to $\RO_{\Forbi(T)}$.\\
 Explicitly, $\RO_{\Forbi(T),\in}$ is the generic graph for the class of all directed
  graphs not containing $T$.
\end{thm}
\begin{proof}
  We can follow analogously  the proof of Theorem \ref{rforbkk}.
  First show that $\RO_{\Forbi(T),\in}$ does not contain an isomorphic copy of $T$
  and then show the extension property of $\RO_{\Forbi(T),\in}$.  We omit the details.
\end{proof}

This can be extended to classes $\Forbi(\T)$ for any finite set of finite tournaments
(but clearly not to all classes $\Forbi(\T)$ where $\T$ is an infinite set of finite tournaments).  In Chapter~\ref{generictournaments} we shall 
prove that all ultrahomogeneous tournaments are also finitely presented.

\chapter{Ultrahomogeneous partial orders}
\label{homposetchapter}

Ultrahomogeneous partial orders an pose interesting problem for finite
presentations.  The condition of transitivity can be easily axiomatized by forbidding a special
configuration of three vertices.  In a finite presentation, however, any binary
relation needs to be derived from two vertices alone.  In the case of the universal
graph for the class $\Forbi(\Kk)$ we solved the problem by explicitly representing all simpler
vertices connected by an edge within the representation of every vertex. This does
not suffice for partial orders.  The forbidden configuration is not irreducible
(contains two vertices not connected by an edge) and thus the representation of
vertex must encode more than just neighboring vertices. This is the main problem
we need to solve in giving a finite presentation of the generic partial order.

Several examples of (not necessarily ultrahomogeneous) finitely presented
linear orders and partially ordered sets are easy to find:
\begin{itemize}
 \item the set of all natural numbers $(\N,\leq)$ (according to von Neumann one can define an ordinal as a well founded complete set and the order $\leq$ is identified with $\in$),
 \item the set $(\Q,\leq)$ 
 (see \cite{ehrlich} where a variant of surreal numbers \cite{conway} is presented which implies a finite representation of $\Q$, also see Section \ref{conwaysur}),
 \item $(P,\leq_P)\times (P',\leq_{P'})$ for finitely presented structures $(P,\leq_P)$ and $(P',\leq_{P'})$,
 \item the lexicographic product of $(P,\leq_P)$ and $(P',\leq_{P'})$ for finitely presented $(P,\leq_P)$ and $(P',\leq_{P'})$  (In fact any ``product'' defined ``coordinate-wise'' is finitely presented).
\end{itemize}
\label{genericpo}

In this chapter we show finite presentations of all ultrahomogeneous partial
orders.  Recall the classification given by Theorem \ref{schmerlthm}.  A finite
presentation of an antichain is trivial.  Using a finite presentation of $(\Q,\leq)$, it
is easy to construct a finite presentation of an antichain of chains as well as
a chain of antichains.  The only remaining ultrahomogeneous partially ordered
set is the generic one.  This is an interesting and not obvious case.

\section{The generic partial order}
The main result of this chapter is a finite presentation of the generic partial order
$(\He,\leq_\in)$.  We shall proceed in two steps.  In this section we first
define a partially ordered set $(\He,\leq_\in)$ which extends the definition of $\RO_\in$.
The definition of $\He$ is recursive and thus it may not be considered to be a finite presentation (depending on precise axiomatization of the term).   However
it is possible to modify the construction of $\He$ to a finite presentation
$\HF$. This is done in the last part of this section (see Definition
\ref{hfdef} and Theorem \ref{hfthm}).

  We use the same notation as in Chapter \ref{graphschapter}. In particular we work in a fixed
countable model $\Model$ of the theory of finite sets extended by a single atomic
set $\apple$.  Also recall the following notations:
  $$M_L=\{A:A\in M,\apple\notin A\},$$
  $$M_R=\{A:(A\cup{\{\apple\}})\in M,\apple\notin A\}.$$

Here is the recursive definition of $(\He,\leq_\He)$.

\begin{defn}
  \label{Hdef}
  The elements of $\He$ are all sets $M$
  with the following properties:
   \begin{enumerate}
     \item(correctness)
	\begin{enumerate}
	\item $\apple\notin M$,
	\item $M_L\cup M_R\subset \He$,
	\item $M_L\cap M_R = \emptyset$.
	\end{enumerate}
     \item(ordering property) $(\{A\}\cup A_R)\cap(\{B\}\cup B_L)\neq \emptyset$ for each $A\in M_L, B\in M_R$,
     \item(left completeness) $A_L\subseteq M_L$ for each $A\in M_L$,
     \item(right completeness) $B_R\subseteq M_R$ for each $B\in M_R$.
   \end{enumerate}
  The relation of $\He$ is denoted by $\leq_\in$ and it is defined as follows:  We put $M<_\in N$ if
  $$(\{M\}\cup M_R)\cap (\{N\}\cup N_L)\neq \emptyset.$$
  We write $M\leq_\in  N$ if either $M<_\in N$ or $M=N$.
\end{defn}

The class $\He$ is non-empty  (as  $M=\emptyset=
\HE 0,\emptyset,\emptyset \in \He$). (Obviously the correctness property
holds. Since $M_L=\emptyset$, $M_R=\emptyset$, the ordering property and
completeness properties follow trivially.)

Here are a few examples of non-empty elements of the structure $\He$:
$$
  \begin{array}{c}
    \TA,\\
    \TB,\\
    \TC.\\
  \end{array}
$$

It is a non-trivial fact that $(\He,\leq_\in)$ is a partially ordered set.  This will be proved after
introducing some auxiliary notions:

\begin{defn}
  Any element $W\in (A\cup A_R)\cap(B\cup B_L)$ is called a {\em witness}
of the inequality $A<_\in B$.
\end{defn}
\begin{defn}
  The {\em level of $A\in \He$} is defined as follows:
\begin{eqnarray}
l(\emptyset)&=&0,\nonumber \\
l(A)&=&max(l(B):B\in A_L\cup A_R) + 1 \hbox{ for } A\neq \emptyset.\nonumber
\end{eqnarray}
\end{defn}

We observe the following facts (which follow directly from the
definition of $\He$):
\begin{fact}
 \label{lpmnoziny}
 $X <_\in A <_\in Y$ for every $A\in \He$, $X\in A_L$ and $Y\in A_R$.
\end{fact}
\begin{fact}
 \label{fact2}
$A\leq_\in W^{AB}\leq_\in B$ for any $A<_\in B$ and witness $W^{AB}$ of $A<_\in B$.
\end{fact}
\begin{fact}
\label{level}
Let $A<_\in B$ and let $W^{AB}$ be a witness of $A<_\in B$.\\ Then
$l(W^{AB})\leq \min(l(A), l(B))$, and either $l(W^{AB})<l(A)$ or $l(W^{AB})<l(B)$.
\end{fact}
First we prove transitivity.
  \begin{lem}
    \label{ltr}
    The relation $\leq_\in$ is transitive on the class $\He$.
  \end{lem}
  \begin{proof}
   Assume that three elements $A,B,C$ of $\He$ satisfy $A<_\in B<_\in C$. We prove that $A<_\in C$
   holds.  Let $W^{AB}$ and $W^{BC}$ be witnesses of the inequalities $A<_\in B$
   and $B<_\in C$ respectively. First we prove that $W^{AB}\leq_\in W^{BC}$.  We distinguish four cases (depending on the definition of the witness):
   \begin{enumerate}
      \item $W^{AB}\in B_{L}$ and $W^{BC}\in B_{R}$.
      
      In this case it follows from Fact \ref{lpmnoziny} that $W^{AB}<_\in W^{BC}$.
      \item $W^{AB}=B$ and $W^{BC}\in B_{R}$. 
      
      Then $W^{BC}$ is a witness of the inequality $B<_\in W^{BC}$ and thus $W^{AB}<_\in W^{BC}$.
      \item $W^{AB}\in B_{L}$ and $W^{BC} = B$.
      
      The inequality $W^{AB}\leq_\in W^{BC}$ follows analogously to the previous case.
      \item $W^{AB}=W^{BC}=B$ (and thus $W^{AB}\leq_\in W^{BC}$).
   \end{enumerate}
   In the last case  $B$ is a witness of the inequality $A<_\in C$.
   Thus we may assume that $W^{AB}\neq_\in W^{BC}$.  Let $W^{AC}$ be a witness of
   the inequality $W^{AB}<_\in W^{BC}$.  Finally we prove that $W^{AC}$ is a witness
   of the inequality $A<_\in C$. We distinguish three possibilities:
   \begin{enumerate}
      \item $W^{AC}=W^{AB}=A$. 
      \item $W^{AC}=W^{AB}$ and $W^{AC}\in A_R$.  
      \item $W^{AC}\in W^{AB}_R$, then also $W^{AC}\in A_R$ from the completeness property.
   \end{enumerate}
   It follows that either $W^{AC}=A$ or $W^{AC}\in A_R$.  Analogously either
   $W^{AC}=C$ or $W^{AC}\in C_L$ and thus $W^{AC}$ is the witness of inequality
   $A<_\in C$.
  \end{proof}
  \begin{lem}
    \label{lasy}
    The relation $<_\in$ is strongly antisymmetric on the class $\He$.
  \end{lem}
  \begin{proof}
    Assume that $A<_\in B<_\in A$ is a counterexample with minimal $l(A)+l(B)$.
    Let $W^{AB}$ be a witness of the inequality $A<_\in B$ and $W^{BA}$ a witness of
    the reverse inequality.  From Fact \ref{fact2} it follows that $A\leq_\in W^{AB}\leq_\in B\leq_\in W^{BA}\leq_\in A \leq_\in W^{AB}$.
    From the transitivity we know that $W^{AB}\leq_\in W^{BA}$ and $W^{BA}\leq_\in W^{AB}$.

    Again  we  consider 4 possible cases:
    \begin{enumerate}
      \item $W^{AB}=W^{BA}$.
      
      From the disjointness of the sets $A_L$ and $A_R$ it follows that $W^{AB}=W^{BA}=A$.  Analogously we obtain
      $W^{AB}=W^{BA}=B$, which is a contradiction.
      \item Either $W^{AB}=A$ and $W^{BA}=B$ or $W^{AB}=B$ and $W^{BA}=A$. 
      
      Then
      a contradiction follows in both cases from the fact that $l(A)<l(B)$ and $l(B)<l(A)$ (by Fact \ref{level}).
      \item $W^{AB}\neq A$, $W^{AB}\neq B$, $W^{AB}\neq W^{BA}$.  
      
      Then $l(W^{AB})<l(A)$ and
      $l(W^{AB})< l(B)$. Additionally  we have $l(W^{BA})\leq l(A)$ and $l(W^{BA})\leq
      l(B)$ and thus $A$ and $B$ is not a minimal counter example.
      \item $W^{BA}\neq A$, $W^{BA}\neq B$, $W^{AB}\neq W^{BA}$.
      
      The contradiction follows symmetrically to the previous case from the minimality of $l(A)+l(B)$.
    \end{enumerate}
  \end{proof}
\begin{thm}\label{order}
  $(\He,\leq_\in)$ is a partially ordered set.
\end{thm}
  \begin{proof}
    Reflexivity of the relation follow directly from the definition, transitivity and antisymmetry follow from  Lemmas \ref{ltr} and \ref{lasy}.
  \end{proof}

Now we are ready to prove the main result of this section:

\begin{thm}
  \label{univ2}
  $(\He,\leq_\in)$ is the generic partially ordered set for the class of all countable partial orders.
\end{thm}
First we show the following lemma:
\begin{lem}
  \label{ext}
  $(\He,\leq_\in)$ has the extension property.
\end{lem}
\begin{proof}
  Let $M$ be a finite subset of the elements of $\He$.  We want to extend the
  partially ordered set induced by $M$ by the new element
  $X$.  This extension can be described by three subsets of $M$: $M_-$ containing
  elements smaller than $X$, $M_+$ containing elements greater than $X$, and $M_0$
  containing elements incomparable with $X$.  Since the extended relation is
  a partial order we have the following properties of these sets:
  \begin{enumerate}
    \item[I.] Any element of $M_-$ is strictly smaller than any element of $M_+$,
    \item[II.] $B\leq_\in A$ for no $A\in M_-$, $B\in M_0$,
    \item[III.] $A\leq_\in B$ for no $A\in M_+$, $B\in M_0$,
    \item[IV.] $M_-$, $M_+$ and $M_0$ form a partition of $M$.
  \end{enumerate}

  Put
  $$\overline{M_-}=\bigcup_{B\in M_-}B_L\cup M_-,$$
  $$\overline{M_+}=\bigcup_{B\in M_+}B_R\cup M_+.$$
  We verify that the properties I., II., III., IV. still hold for sets $\overline{M_-}$, $\overline{M_+}$, $M_0$.

  \begin{enumerate}
    \item [ad I.]
    We prove that any element of $\overline{M_-}$ is strictly smaller than any element of $\overline{M_+}$:
    
    Let $A\in\overline {M_-}, A'\in\overline {M_+}$.  We prove $A<_\in A'$.
    By the definition of $\overline {M_-}$ there exists $B\in M_-$ such that either $A=B$ or $A\in B_L$.
    By the definition of $\overline {M_+}$ there exists $B'\in M_+$ such that either $A'=B'$ or $A'\in B'_R$.
    By the definition of $<_\in $ we have $A\leq_\in B$, $B<_\in B'$ (by I.) and $B'\leq_\in A'$ again by the definition of $<_\in $.  It follows $A<_\in A'$.
    \item [ad II.]
    We prove that $B\leq_\in A$ for no $A\in \overline{M_-}$, $B\in M_0$:

    Let $A\in\overline {M_-}, B\in M_0$ and let
    $A'\in M_-$ satisfy either $A=A'$ or $A\in A'_L$.
    We know that $B\nleq_\in A'$ and as $A\leq_\in A'$ we have also $B\nleq_\in A$.
    \item [ad III.] To prove that $A\leq_\in B$ for no $A\in \overline{M_+}$, $B\in M_0$ we can proceed similarly to ad II.
    \item [ad IV.]
    We prove that $\overline{M_-}$, $\overline{M_+}$ and $M_0$ are pairwise disjoint:
    
    $\overline {M_-}\cap \overline {M_+}=\emptyset$ follows from I.
    $\overline {M_-}\cap M_0=\emptyset$ follows from II.
    $\overline {M_+}\cap M_0=\emptyset$ follows from III.
  \end{enumerate}

  It follows that $A=\HE N,{\overline {M_-}},{\overline {M_+}}$ is an element
  of $\He$ with the desired inequalities for the elements in the sets $M_-$ and $M_+$.

  Obviously each element of $M_-$ is smaller than $A$ and each element of $M_+$ is greater than $A$.

  It remains to be shown that each $N\in M_0$ is incomparable with $A$. However
  we run into a problem here: it is possible that $A=N$.
  We can avoid this problem by first considering the set:
 $$M'=\bigcup_{B\in M}B_R\cup M.$$
 It is then easy to show that $B=\HE 0,\emptyset,{M'}$ is an element of $\He$
 strictly smaller than all elements of $M$.

 Finally we construct the set $A'=\HE 0,{A_L\cup\{B\}},{A_R}$.  The set $A'$ has the same
 properties with respect to the elements of the sets $M_-$ and $M_+$ and differs from any set in
 $M_0$.  It remains to be shown that $A'$ is incomparable with $N$.

 For contrary, assume for example, that $N<_\in A'$ and $W^{NA'}$ is the witness of the
 inequality.  Then $W^{NA'}\in  \overline {M_-}$ and $N\leq_\in W^{NA'}$.
 Recall that $N\in M_0$.
 From IV. above and the definition of $A'$ it follows that $N<_\in W^{NA'}$.
 From $ad$ III. above it follows that there is no
 choice of elements with $N<_\in W^{NA'}$, a contradiction.

 The case $N>_\in A'$ is analogous.
 The case $N>_\in A'$ is analogous.
\end{proof}
\begin{proof}
  Proof of Theorem \ref{univ2} follows by combining Lemma \ref{ext} and fact that extension property imply both universality and ultrahomogeneity of the partial order (Lemma \ref{extuniv}).
\end{proof}
\eject

  \begin{figure}[t]
    \begin{center}
      {\def\IPEfile{poset.ipe}\begingroup
  \catcode`\%=9\catcode`\!=0\catcode`\-=11\input{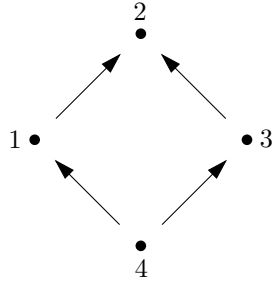}}
    \end{center}
    \caption{Partially ordered set $(P,\leq_P)$.}
    \label{posetr}
  \end{figure}
\begin{example}
Consider partial order $(P,\leq_P)$ depicted in Figure \ref{poset}. The function
$c$ embedding  $(P,\leq_P)$ to $(\He,\leq_\He)$ can be defined as:
  \begin{eqnarray}
c(1)&=&\TA\nonumber \\
c(2)&=&\TB\nonumber \\
c(3)&=&\TC\nonumber \\
c(4)&=&\TD\nonumber
  \end{eqnarray}
\end{example}

\subsection{Finite presentation of the generic partial order}
Definition \ref{Hdef} of $\P_\in$ is recursive and thus may not be considered a finite presentation.  However it can be modified to give a finite presentation of $\P$ which we denote by $\HF$.  After defining carefully the elements of $\HF$ the relation $\leq_\HF$ follows easily.

\begin{defn}
\label{hfdef}
Elements of $\HF$ are all pairs $(P,\leq_P)$ which satisfy the following:

\begin{enumerate}
\item [I.] Axioms for $P$: 

   \begin{enumerate}
     \item[1.](correctness) 
	\begin{enumerate}
	\item[(a)]$\apple\notin M$,
	\item[(b)] $M_L\cup M_R\subset P$,
	\item[(c)] $M_L\cap M_R = \emptyset$.
	\end{enumerate}
     \item[2.](ordering property) $(\{A\}\cup A_R)\cap(\{B\}\cup B_L)\neq \emptyset$ for each $A\in M_L, B\in M_R$,
     \item[3.](left completeness) $A_L\subseteq M_L$ for each $A\in M_L$,
     \item[4.](right completeness) $B_R\subseteq M_R$ for each $B\in M_R$.
   \end{enumerate}
\item [II.] Axioms for $\leq_P$:

\begin{enumerate}
\item[1.] $\leq_P$ is a partial order.
\item[2.] $\leq_P$ is the transitive closure of the set $\{(A,B):A\in B_L\cup B_R,B\in P\}\cup\{(A,A):A\in P\}$.
\item[3.] $(P,\leq_P)$ has a maximum denoted by $m(P,\leq_P)$.
\end{enumerate}
\end{enumerate}

The relation $\leq_\HF$ of $\HF$ is defined by comparison (in $\P_\in$) of the greatest elements:
$$(P,<_P)\leq_\HF(P',<_{P'})\hbox{ if and only if }m(P,<_P)\leq m(P',<_{P'})\hbox{ in }\He.$$
\end{defn}

This definition is a finite presentation.  Note that the maximum, the completeness and the transitive closure are axiomatized by first order formulas.  We next turn to the presentation of $\leq_\HF$.  First we show that $\He$ and $\HF$ are compatible:
\begin{lem}
\label{pkor}
$P\subset \He$ for each $(P,<_P)\in \HF$.
\end{lem}
\begin{proof}
Suppose on the contrary that
there is $A\in P\in \HF$ such that $A\notin \He$.  Without loss of generality we may
assume that there is no $B\in P$, $B\notin \He$ such that $B<_P A$.  From the
definition of $<_P$ it follows that $C\in \He$ for each $C\in A_L\cup A_R$.  Thus
for $A$ we have 1.(b) in Definition \ref{hfdef} equivalent to the I.1.(b)
from Definition \ref{Hdef}.  The rest of the definition is equivalent too,
so we have $A\in \He$.
\end{proof}

\begin{thm}
\label{hfthm}
$(\HF,\leq_\HF)$ is finitely presented and isomorphic to $(\He,\leq_\HF)$ (as well as to $(\P,\leq_\P)$).
\end{thm}
\begin{proof}
For the correctness of the definition of $\HF$ note that $m(P,<_P)$ are elements of $\He$ and $\leq_\in$ in $\P_\in$ is described by  a first order formula. 

We already noted that Definition \ref{hfdef} is a finite presentation of $\He$.  We claim that the correspondence $$\varphi:(P,<_P)\mapsto m(P,<_P)$$
is isomorphism of $\HF$ and $\He$.

Clearly it suffices to prove that $\varphi$ is bijective.  This follows from the following two facts:
\begin{enumerate}
\item For each $(P,<_P)$ the set $P$ contains all the elements of $\He$ which appear in
      the construction of $m(P,<_P)\in \He$.  (This is the consequence of 1.(b)) and both Definition \ref{Hdef} and Definition \ref{hfdef} I.)
\item For each $(P,<_P)$ the set $P$ consists only of elements of $\He$ which appear in the construction of $m(P,<_P)$.

      Let $A^1<_P m(P,<_P)$.  By definition of $<_P$ we have $A^1,A^2,\ldots,A^t=m(P,<_P)$ such that $A^i\in A^{i+1}_L\cup A^{i+1}_R$.
      But as $m(P,<_P)\in \He$ we get also $A\in \He$ by Definition \ref{Hdef} 2.
\end{enumerate}
So for different sets, the maximum
elements are different and each $M\in \He$ can be used as maximum element to
construct an element of $\HF$.
\end{proof}
From the discussion in the introduction of this chapter it follows:
\begin{thm}
All ultrahomogeneous partial orders are finitely presented.
\end{thm}

\subsection{Remark on Conway's surreal numbers}
\label{conwaysur}
Recall the definition of surreal numbers, see \cite{conway}. (For a
recent generalization see \cite{ehrlich}). Surreal numbers are defined
recursively together with their linear order.  We briefly indicate how
the partial order $(\He,\leq_\He)$ fits into this scheme.
\begin{defn}
  A surreal number is a pair $x=\{x^L|x^R\}$, where every member of the sets $x^L$ and
  $x^R$ is a surreal number and every member of $x^L$ is strictly smaller than
  every member of $x^R$.
  
  We say that a surreal number $x$ is less than or equal to the surreal number $y$ if and
  only if $y$ is not less than or equal to any member of $x^L$ and any member of $y^R$ is not
  less than or equal to $x$.
  
  We will denote the class of surreal numbers by $\Sur$.
\end{defn}

$\He$ may be thought of as a subset of $\Sur$ (we recursively add $\apple$ to express pairs $x^L$, $x^R$).  The recursive definition of
$A\in\He$ leads to the following order which we define explicitly:

\begin{defn}
  For elements $A,B\in \He$ we write $A\leq_\Sur B$, when there is no $l\in A_L$ such that $B\leq_\Sur l$
  and no $r\in B_R$ such that $r\leq_\Sur A$.
\end{defn}

$\leq_\Sur$ is a linear order of $\He$ and it is the restriction of Conway's order.
It is in fact a linear extension of the partial order $(\He,\leq_\in)$:

\begin{thm}
  For any $A,B\in \He$, $A<_\in B$ implies $A<_\Sur B$.
\end{thm}
\begin{proof}
  We proceed by induction on $l(A)+l(B)$.

  For empty $A$ and $B$ the theorem holds as they are not comparable by $<_\in$.

  Let $A<_\in B$ with $W^{AB}$ as a witness.  If $W^{AB}\neq A,B$, then $A<_\Sur W^{AB} <_\Sur B$ by induction.  In the case $A \in B_L$, then $A<_\Sur B$ from the definition of $<_\Sur$.
\end{proof}

\chapter{Ultrahomogeneous tournaments}
\label{generictournaments}
Let us examine generic tournaments characterized by Theorem
\ref{gentournaments}.  Again finite cases are always finitely presented and
thus we focus on infinite ones:
\begin{enumerate}
\item The tournament $(\Q,\leq)$ formed by rationals with usual ordering.
\item The dense local order $S(2)$.
\item The generic tournament $\T$ for the class of all countable tournaments.
\end{enumerate}

The purpose of this short chapter is to show the following perhaps surprising result which parallels the result on partial orders.

\begin{thm}
\label{Tuniv}
All ultrahomogeneous tournaments are finitely presented.
\end{thm}
\begin{proof}

We already outlined the representation of $(\Q,\leq_\Q)$ by Conway's surreal numbers.

To build a representation of the generic tournament $\T$ lets briefly consider
oriented graphs (i.~e. antisymmetric relations).  Let $\RD$ denote the generic
oriented graph.  $\RD$ has finite presentation $\RD_\in$ which we obtain as a
variant of $\RO_\in$: we say that $M$ a is vertex of $\RD_\in$ if and only if $M\in\RO_\in$ and satisfies $M_L\cap M_P=\emptyset$.  (see Definition \ref{rodef}).

The finite presentation of the generic oriented graph $\RD$ may be used to construct a finite
presentation of the generic tournament $\T$.

Denote by $\RD_\N$ the arithmetic presentation of $\RD_\in$.  Explicitly, an
integer $n$ is a vertex of $\RD_\N$ if and only if there exists an element $M$ of
$\RD_\in$ such that $n=c(M)$.  Let $n$ and $n'$ be vertices of $\RD_\in$. There is an edge from $n$ to $n'$ if and
only if there are sets $M$ and $M'$ such as $c(M)=n$ and $c(M')=n'$ and there is
edge from $M$ to $N'$ in $\RD_\in$.  Alternatively there is an edge from $n$ to $n'$
if there is 1 on $2n'$-th place of binary representation of $n$ or on
$(2n+1)$-th place of binary representation of $n'$.

We use the finite presentation $\RD_\N$ of generic oriented graph $\RD$ for the
construction of a finite presentation $\T_\N$ of the generic tournament $\T$:
An integer $n$ is vertex of $\T_\N$ if and only if $n$ is a vertex of $\RD_\N$.  The edges
of $\T_\N$ will be all edges of $\RD_\N$ together with pairs $(n,n')$, $n\leq n'$
for which $(n',n)$ is not an edge of $\RD_\N$.

$\T_\N$ is obviously a tournament.  $\T_\N$ has the extension property by the analogous argument as in the proof of Theorem \ref{ROuniv}.

Finally one can check that the description of $S(2)$ given prior statement
of Theorem \ref{gentournaments} is a finite presentation based on the finite presentation of $(\Q,\leq_\Q)$.
%


\end{proof}

\chapter{Finite presentation of the rational Urysohn space}
\label{homprostorchapter}
\section{Introduction (a bit of history)}
\label{uryintro}
In this chapter we focus on metric spaces.  Unlike all our other examples, metric
spaces are relational structures with a function symbol.  The basic notions of
universality and ultrahomogeneity however translate directly.  There is a unique (up
to isometry) separable Polish space $\Urysohn$ which is both universal (for all
separable metric spaces) and ultrahomogeneous.  (Space is ultrahomogeneous if
every isometry between finite subspaces extends to a total isometry.)

This remarkable result is due to Urysohn \cite{urysohn} and it is quoted as
his last paper (written in 1925). The paper was almost neglected until 1986
when Kat\v etov wrote (one of the last papers in his distinguished career) a
paper \cite{katetov} where he gave a new construction of the Urysohn space.  


The recent activity and importance of the Urysohn space, besides being a
beautiful result in topology (see \cite{urysohn,katetov,Uspenskij,V}), stems from several sources:

\subsection{Early limit argument}
The proof of Urysohn uses a construction of a countable metric space with
rational distances $\Urysohn_\Q$ of which $\Urysohn$ is then the Cauchy
completion. This $\Urysohn_\Q$ is a direct limit of the set of all finite
rational metric spaces.  This limit is a special case of \Fraisse{} limit
introduced several years later.  This is a key result of modern model theory.
It appears that Urysohn anticipated this construction in a quite general (and
complicated) case. (It also appears that Kat\v etov was unaware of
Fra\"\i{}ss\'e's work.) 

\subsection{Topological dynamics}
The Urysohn space is not only an important (and generic) space in the context
of topological dynamics. The automorphism group $Aut(\Urysohn)$ is extremely
amenable which in turn is related to triviality of minimal flows.  This
important connections were discovered in \cite{Pestov02,Pestov98} and then on a
very abstract level by \cite{KPT}, see the recent book \cite{Pestov03}.

\subsection {Combinatorial connection}
The Urysohn space is among the most interesting generic structures with applications
already outlined in Section \ref{Genericsection}. 
Other combinatorial aspects of the Urysohn space are related to the concept of divisibility (see e.g. \cite{d9,NRam,Hjorth,d74}).
Sauer \cite{Sauer} summarize known results
about the the age and weak indivisibility with variants of the Urysohn metric space serving
as the most striking examples demonstrating the existence of generic structures with particular indivisibility properties.

All those examples illustrate the broad context of the Urysohn space.
\section{Finite presentation of $\Urysohn_\Q$}
\label{urysohnsection2}

Given the difficulties to represent even a universal and later the generic partial
order, it seemed that the generic rational metric space was out of reach of
finite presentations. This was also the conclusion of discussions held with
Cameron, Vershik and others in St.~Petersburg meeting in 2005. We have been also informed that Urysohn indicated this as a problem \cite{Soleckypriv}.
In this
section we give such representation that in fact builds upon ideas used in
the construction of the generic partial order.

Now we prove the following which may be viewed
as a contribution to Problem 12 of \cite{Pestov03} (about a model of the Urysohn Space $\Urysohn$).
\begin{thm}
The rational Urysohn space $\Urysohn_\Q$ has a finite presentation.
\end{thm}

\label{Urysohnsection}


We start to develop the theory for vertices as follows:

\label{defU}
\medskip
\noindent
{\bf 1.} A {\em triplet} $\mathbf A$ is a triple $(A,\LeqUp_\mathbf A,\d_\mathbf A)$ where
\begin{enumerate}
\item[$i.$] $A$ is a finite set,
\item[$ii.$] $(A,\LeqUp_\mathbf A)$ is a partial order on $A$,
\item[$iii.$] $(A,\d_\mathbf A)$ is a rational metric space (i.e. $\d_\mathbf A:A\times A\to \Q$ is a metric).
\end{enumerate}
$\LeqUp_\mathbf A$ is called the {\em standard} order of $\mathbf A$.

Triplets $\mathbf A=(A,\LeqUp_\mathbf A,\d_\mathbf A)$ and
$\mathbf B=(B,\LeqUp_\mathbf B,\d_\mathbf B)$ are said to be {\em isomorphic}
if there exists a bijection $\varphi:A\to B$ which is both isomorphism of partial orders $(A,\LeqUp_\mathbf A)$ and $(B,\LeqUp_\mathbf B)$ and isometry of spaces $(A,\d_\mathbf A)$ and $(B,\d_\mathbf B)$.

Concerning partial orders we use the standard terminology. Particularly any
element $a\in A$ determines a {\em down set} $\dr a=\{b:b\LeqUp_\mathbf A a\}$,
which induces by the restriction of $\LeqUp_\mathbf A$ and $\d_\mathbf A$ the
triplet $\dr a$.  By abuse of the notation this triplet will be also denoted by
$\dr a$.  Let also $\h(\mathbf A)$ ({\em height of $\mathbf A$}) be the maximal
size of a chain in $(A,\LeqUp_\mathbf A)$.

\medskip
\noindent
{\bf 2.} A triplet $\mathbf A$ is said to be {\em proper} if all its down sets (as triplets) are non-isomorphic and if $(A,\LeqUp_\mathbf A)$ has both a greatest element and a smallest element (denoted by $\max_\mathbf A$ and $\min_\mathbf A$).

\medskip
\noindent
{\bf 3.}  A proper triplet $\mathbf A$ is said to be {\em path metric PM} if for every $a, a'\in A$ which are incomparable in $\LeqUp_\mathbf A$ there exist $a''\in A,a''\LeqUp_\mathbf A a, a''\LeqUp_\mathbf A a'$ such that $\d_\mathbf A(a,a')=\d_\mathbf A(a,a'')+\d_\mathbf A(a'',a')$.
Such an $a''$ will be called the {\em witness of $\d_\mathbf A(a,a')$}.

Proper path-metric triplet will be abbreviated as PPM-triplet.  An example of PPM-triplet is in Figure~\ref{PPM}.
\begin{figure}[t]
\includegraphics{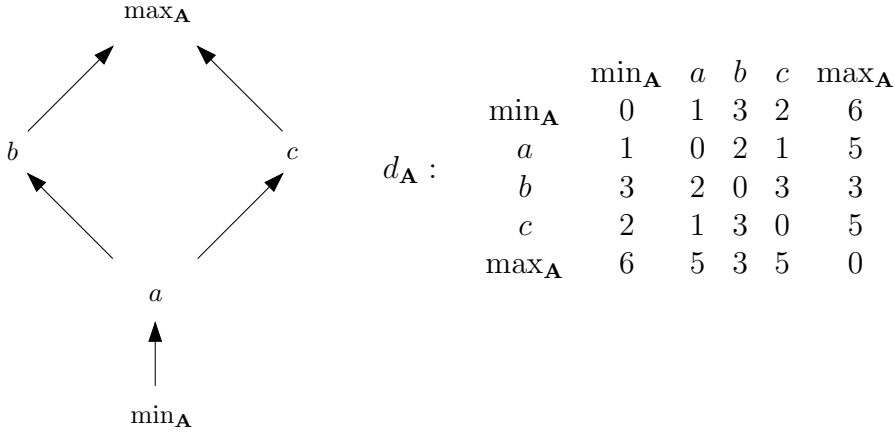}
\vskip -5cm
\hskip 5cm
$
d_\mathbf A:
\begin{array}{lcccccc}
&&\min_\mathbf A&   a& b& c&\max_\mathbf A\\
&\min_\mathbf A& 0& 1& 3& 2& 6\\
&             a& 1& 0& 2& 1& 5\\
&             b& 3& 2& 0& 3& 3\\
&             c& 2& 1& 3& 0& 5\\
&\max_\mathbf A& 6& 5& 3& 5& 0\\
\end{array}
$
\vskip 2cm
\caption{A PPM-triplet $\mathbf A$.}
\label{PPM}
\end{figure}
\begin{figure}[h!t]
\centerline{\includegraphics{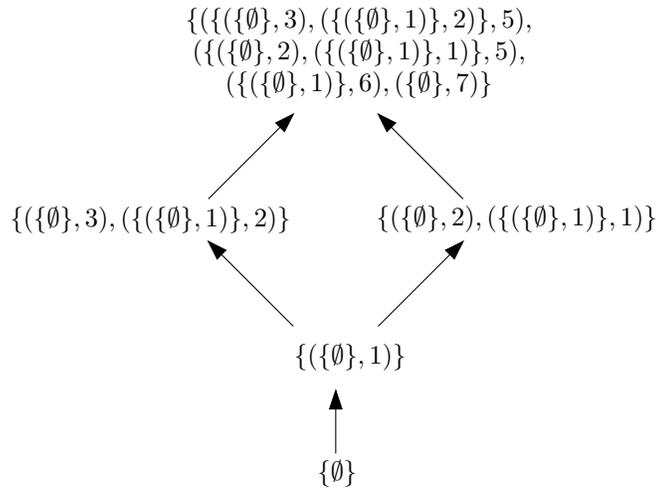}}
\caption{A complete triplet $\mathbf A$.}
\label{complete}
\end{figure}

\medskip
\noindent
{\bf 4.} A PPM-triplet is said to be {\em complete} if the following holds for every $a\in A$:
$$a=\{(b,\d_\mathbf A(a,b)):b\in\dr a,a\neq b\}.$$
Note that $\min_\mathbf A=\emptyset$.

An example of a complete triplet isomorphic to the PPM-triplet of Figure~\ref{PPM} is shown in Figure~\ref{complete}.

Thus the structure of $max_\mathbf A$ encodes the whole complete triplet $\mathbf A$.

Observe also that every downset $\dr a$ is itself a complete triplet. This triplet will also be denoted shortly by $\dr a$. If $b\in A$ then we also say that {\em $\dr b$ is mentioned in $\mathbf A$}. By induction on the $\h(\mathbf A)$ we easily see the following fact (which is the reason why we introduced the notion of complete triplets):

\begin{fact}
Let $\mathbf A$, $\mathbf B$ be isomorphic complete triplets. Then $\mathbf A=\mathbf B$.
\end{fact}
Now  we can state the basic construction of this chapter, a finite presentation of $\Urysohn_\Q$ (which should be compared with inductive constructions of Urysohn and Kat\v etov, see also Section \ref{alternative}):
\begin{defn}[a finite presentation of the Urysohn space $\Urysohn_\Q$]
\label{urysohndef}
~\\
Denote by $\U$ the set of all complete triplets.
The metric $\d_\U$ on $\U$ is defined as follows:
Let $\mathbf A=(A,\LeqUp_\mathbf A,\d_\mathbf A)$,
$\mathbf B=(B,\LeqUp_\mathbf B,\d_\mathbf B)$ 
be complete triplets. We put
$\d_\U(\mathbf A,\mathbf B)=\min (\d_\mathbf A(\max_\mathbf A,a)+\d_\mathbf B(\max_\mathbf B,b))$ where the minimum is taken over all $a\in \mathbf A$, $b\in \mathbf B$ such that $a = b$. 
\end{defn}

If $\max_\mathbf B\in A$ (and thus also $\d_\U(\mathbf A,\mathbf B)=\d_\mathbf A(\max_ \mathbf A,\max_ \mathbf B)$)
we say that {\em $\mathbf B$ is mentioned in $\mathbf A$}.

If neither $\mathbf A$ is mentioned in $\mathbf B$ nor is $\mathbf B$ mentioned in $\mathbf A$ then for $a, b$ reaching the minimum, we call the triplet $\dr a=\dr b$ a
{\em witness of $\d_\U(\mathbf A,\mathbf B)$}.

We will show that this construction yields a finite presentation of $\Urysohn_\Q$.  This will be done in a sequence of
statements formulated as Proposition \ref{metrik}, Proposition \ref{univ}, and Theorem \ref{generic} which is the main result of this chapter.

\begin{prop}
\label{metrik}
$(\U,d_\U)$ is a metric space.
\end{prop}
\begin{proof}
Clearly $\d_\U\geq 0$ and $d_\U(\mathbf A,\mathbf B)=0$ if and only if $\mathbf A=\mathbf B$

Assume that the triangle inequality does not hold. Take the triangle $\mathbf A,\mathbf B,\mathbf C\in \U$ such that $\h{}(\mathbf A)+\h{}(\mathbf B)+\h{}(\mathbf C)$ is minimal and the triangle
inequality does not hold for $\mathbf A,\mathbf B,\mathbf C$. Without loss of generality, assume that $$\d_\U(\mathbf A,\mathbf B)>\d_\U(\mathbf B,\mathbf C)+\d_\U(\mathbf C,\mathbf A).$$

We distinguish several cases according to the existence of witness elements:

\medskip
\noindent
{\bf Case 1:} The distances $\d_\U(\mathbf A,\mathbf B)$, $\d_\U(\mathbf B,\mathbf C)$ and $\d_\U(\mathbf C,\mathbf A)$ do not have any witness:
\begin{enumerate}
\item If $\mathbf A$ and $\mathbf B$ are both mentioned in $\mathbf C$, then
  there exist $a\in C, b\in C$ such that $\d_\U(\mathbf B,\mathbf C)=\d_\mathbf
  C(b,\max_ \mathbf C)$, $\d_\U(\mathbf A,\mathbf C)=\d_\mathbf C(a,\max_ \mathbf
  C)$ and thus the triangle $a,b,\max_\mathbf C$ violates
  the triangle inequality in $\d_\mathbf C$.
  Similarly we can proceed for any other vertex of the triangle and thus no
  vertex defines the distances to both remaining vertices.  

\item If $\mathbf A$ is mentioned in $\mathbf B$ and $\mathbf C$ mentioned in
  $\mathbf A$, then there will be some $a \in \mathbf B$ such that $\dr a
= \mathbf A$ and also there will be some $c\in \mathbf A$ such that $c\LeqUp_\mathbf B a\in \mathbf B$ such that $\dr c = \mathbf
C$. Then the triangle $a,c,\max_\mathbf B$ would violate triangle inequality
 of $\d_\mathbf B$.
\end{enumerate}

\medskip
\noindent
{\bf Case 2:}
Assume that $\d_\U(\mathbf C,\mathbf A)$ has witness $\mathbf X$.

\centerline{\includegraphics{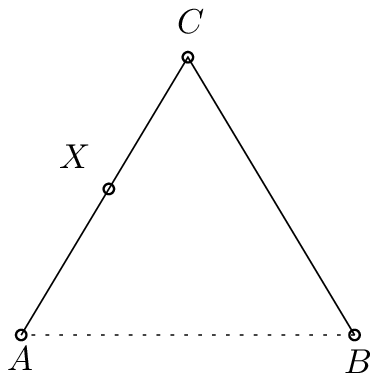}}

Since $\mathbf X$ is a witness:
   $$\d_\U(\mathbf A,\mathbf X)=\d_\U(\mathbf A,\mathbf C)-\d_\U(\mathbf X,\mathbf C).$$ 
The triangles $\mathbf B,\mathbf C,\mathbf X$ and $\mathbf A,\mathbf B,\mathbf X$ do not violate the triangle inequality (since $\h{}(\mathbf A)+\h{}(\mathbf B)+\h{}(\mathbf C)$ would not be minimal):
  $$\d_\U(\mathbf X,\mathbf B)\leq\d_\U(\mathbf X,\mathbf C)+\d_\U(\mathbf C,\mathbf B),$$
$$\d_\U(\mathbf A,\mathbf B)\leq \d_\U(\mathbf A,\mathbf X)+\d_\U(\mathbf X,\mathbf B).$$ 
It follows that:
$$\d_\U(\mathbf A,\mathbf B)\leq \d_\U(\mathbf A,\mathbf C)-\d_\U(\mathbf X,\mathbf C)+\d_\U(\mathbf X,\mathbf C)+\d_\U(\mathbf C,\mathbf B),$$
$$\d_\U(\mathbf A,\mathbf B)\leq \d_\U(\mathbf A,\mathbf C)+\d_\U(\mathbf C,\mathbf B)$$ which is a contradiction.

\medskip
\noindent
{\bf Case 3:} If $\d_\U(\mathbf C,\mathbf B)$ has a witness $\mathbf X$ then
we proceed in a complete analogy with case $2.$
(i.e. exchanging the roles of $\mathbf A$ and $\mathbf B$).

\medskip
\noindent
{\bf Case 4:} Assume that $\mathbf X$ is a witness of $\d_\U(\mathbf A,\mathbf B)$ and that $\d_\U(\mathbf A,\mathbf C)$ and $\d_\U(\mathbf B,\mathbf C)$ have no witness. Thus $\mathbf A$ mentions $\mathbf C$ (resp. $\mathbf B$ mentions $\mathbf C$) or the other way around.

\centerline{\includegraphics{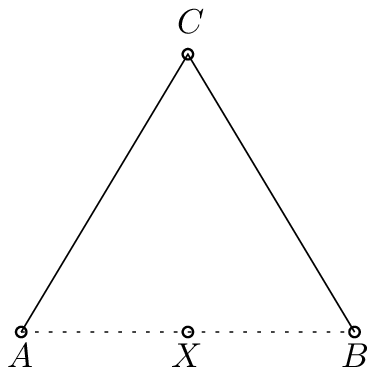}}

Since $\mathbf C$ cannot mention both $\mathbf A$ and $\mathbf B$, we can assume that $\mathbf A$ mentions $\mathbf C$.

If $\mathbf B$ mentioned $\mathbf C$ as well then $\mathbf C$ would be a witness for $\d_\U(\mathbf A,\mathbb B)$. It would follow that $$\d_\U(\mathbf A,\mathbf B)=\d_\U(\mathbf A,\mathbf C)+\d_\U(\mathbf C,\mathbf B).$$ This is a contradiction. 

Assume that $\mathbf C$ mentions $\mathbf B$.  Again from the transitivity property
we have that $\mathbf A$ defines the distances to both $\mathbf B$ and $\mathbf C$ and thus for the triangle $\mathbf A,\mathbf B,\mathbf C$ the triangle inequality holds, a contradiction.
\end{proof}

\begin{prop}
\label{univ}
$(\U,d)$ is a metric space which contains all finite metric spaces.
\end{prop}
\begin{proof}
We describe an algorithm for an isometric embedding of a given metric space $(X,d')$ into $\U$.

We fix a linear order of the vertices $x\in X$ by assigning to each vertex a unique natural number
$n(x)\in\{0,1,\ldots,|X|-1\}$.

For a vertex $x\in X$, the triplet $\mathbf U(x)=(U(x),\LeqUp_{\mathbf U(x)},d_{\mathbf U(x)})$ representing $x$ is defined recursively as follows:
\begin{enumerate}
\item
Put:

$\max(x)=\emptyset$ for $n(x)=0$,

$\max(x)=\{(U(y),d'(y,x)): y\in X, n(y)<n(x)\}$ for $n(x)>0$,

$U(x)=\{\max(y): y\in X, n(y)\leq n(x)\}$.

\item The order $\LeqUp_{\mathbf U(x)}$ is the linear order defined by:  $$U(y)\LeqUp_{\mathbf U(x)} U(y')\hbox{ if and only if }n(y)\leq n(y').$$
\item The distance is defined by $\d_{\mathbf U(x)}(U(y),U(y')) = d'(y,y')$.
\end{enumerate}

We verify that $\mathbf U(x)$ is a complete triplet:

Clearly the finite linear order $\LeqUp_{\mathbf U(x)}$ has the smallest element $0$ and the greatest element $\max_{\mathbf U(x)}=\max(x)$ and no two downsets are isomorphic. Thus $\mathbf U(x)$ is a proper triplet.

In the linear order, each pair of elements are comparable, so trivially $\d_{\mathbf U(x)}$ has the path metric property.  From the construction of $U(x)$ it follows that $\mathbf U(x)$ is a complete triplet and thus  $\mathbf U(x)\in \U$.

Consider $x,y\in X$, $n(x)\leq n(y)$. As $\mathbf U(y)$ mentions $\mathbf U(x)$:
$$d'(x,y)=d_{\mathbf U(\mathbf y)}(\mathbf U(x),\mathbf U(y)).$$
\end{proof}

\begin{thm}\label{generic}
$(\U,d)$ is the generic metric space.
\end{thm}
\begin{proof}
The set $\U$ is obviously countable, since all elements are finite.
By Proposition \ref{metrik} $(\U,d)$ is a metric space.
By a construction similar to the construction performed in the proof of Proposition \ref{univ}, we verify that $(\U,d)$ has the extension property.
Clearly it suffices to verify the extension property in the following form:

Fix $\X$ any finite subset of $\U$ together with a distance function $D:\X\to \Q$
defining a single vertex extension of the metric subspace induced by $\X$ (i.e. the desired distances to the new vertex such that $D$ does not violate the triangle inequality property of $d_\U$ restricted to $\X$). (Remark that Kat\v etov axiomatized all
possible functions $D$. Such functions are now called Kat\v etov functions \cite{Uspen2}, see also \cite{Nguyen}. The Kat\v etov's description is similar to                               
the definition (3.) of a triplet.)
We find a finite triplet $\mathbf M(\X,D)=(M(\X,D),\LeqUp_{\mathbf M(\X,D)},d_{\mathbf M(\X,D)})\in \U$, such that $\d_\U(\mathbf M(\X,D),\mathbf A)=D(\mathbf A)$ for each $\mathbf A\in \X$.

$\mathbf M(\X,D)$ is defined according to the following algorithm:

\begin{enumerate}
\item[(1)]
The vertex set of $M(\X,D)$ is the union of all sets $A$ such that there exists $\mathbf A=(A,\LeqUp_\mathbf A${}$,d_\mathbf A)\in \X$, together
with the single new vertex $m$ which we describe later (in (4)).

\item[(2)]
For $a,b$ in $\mathbf M(\X,D)$ we set $a\LeqUp_{\mathbf M(\X,D)} b$ if and only if $b=m$ or there exists $\mathbf A=(A,\LeqUp_\mathbf A,d_\mathbf A)\in \X$ such that $a,b\in A$ and $a\LeqUp_\mathbf A b$.  

Observe that $m=\max_{\mathbf M(\X,D)}$.

\item[(3)]
For $a,b\in M(\X,D)$ we set:
\begin{enumerate}
 \item[$i.$] $d_{\mathbf M(\X,D)}(a,b) = 0$ when $a=b$.
 \item[$ii.$] $d_{\mathbf M(\X,D)}(a,b) = \d_\U(\dr a,\dr b)$, when $a,b\neq m$.
 \item[$iii.$] $d_{\mathbf M(\X,D)}(m,b) = \min_{\mathbf C \in \X} D(\mathbf C) + d_\U(\mathbf C,\dr b)$. 

 We call $\mathbf C\in\X$ such that $d_{\mathbf M(\X,D)}(m,b) = D(\mathbf C) + d_\U(\mathbf C,\dr b)$ with $\mathbf C\neq\dr b$ a {\em witness} of $d_{\mathbf M(\X,D)}(m,b)$.  
 Observe that $d_{\mathbf M(\X,D)}(m,b)$ has no witness if and only if $\dr b\in \X$ and in that case, $d_{\mathbf M(\X, D)}(m,b)=D(\dr b)$.
 \item[$iv.$] $d_{\mathbf M(\X,D)}(a,m) = d_{\mathbf M(\X,D)}(m,a)$ defined in $iii$.
\end{enumerate}
\item[(4)] $m=\{(a,d_{\mathbf M(\X,D)}(m,a)): a\in \bigcup_{\mathbf B\in\X} B\}$.
\end{enumerate}

We verify that $\mathbf M(\X,D)$ is a complete triplet by verifying conditions {\bf 1.}--{\bf 4.} of the definition.

We first verify {\bf 1.} $ii.$:

$(M(\X,D),\LeqUp_{\mathbf M(\X,D)})$ is a partial order: for $a\LeqUp_{\mathbf M(\X,D)} b\LeqUp_{\mathbf M(\X,D)} c$ either $c=m$ and thus $a\LeqUp_{\mathbf M(\X,D)} c$ holds trivially from the definition or there exists $\mathbf A\in \X$ such that $a,b,c\in A$ and the fact that $a\LeqUp_{\mathbf M(\X,D)} c$ follows from $a\LeqUp_\mathbf A c$.

There is a single maximal element $m$ and a single minimal element $\emptyset$.

$\mathbf M(\X,D)$ is a proper triplet:

The downsets of every $a\in A$, $\mathbf A\in \X$ are preserved (i.e. downset
of $a$ in $\mathbf M(\X,D)$ is equivalent to the downset of $a$ in $\mathbf
A$).  This follows from the fact that $\LeqUp_{\mathbf M(\X,D)}$ is inherited from
$\LeqUp_\mathbf A$ and that the downset of $a\in A$ is identical to the donwset
of $a$ in any $\mathbf B\in \U$ such that $a\in B$.  Since all $\mathbf
A\in \X$ are complete triplets, all the downsets are non-isomorphic.  This verifies {\bf 2.}

Next, we prove that $d_{\mathbf M(\X,D)}$ is a rational metric (condition {\bf 1.} $iii.$ of the definition):

$d_{\mathbf M(\X,D)}(a,b)$ is a positive rational
number for each $a\LeqUp_{\mathbf M(\X,D)} b$.  Observe that for the last part
of the construction of $d_{\mathbf M(\X,D)}$, the shortest path always exists: there
is always a path from any element to the minimal element.
The fact that that $d_{\mathbf M(\X,D)}$ is symmetric directly follows from the construction.

We verify the triangle inequality property for $d_{\mathbf M(\X,D)}$:

Rule $ii.$
merely translates metric $d_\U$ to $d_{\mathbf M(\X,D)}$. Any triplet violating
triangle inequality property must have two distances defined by $iii.$ or
$iv$.
Let $a,b\in {\mathbf M(\X, D)}$ and consider the triangle $a,b,m$.

First assume that $a,b\in \X$ and that $\dr a, \dr b\notin \X$. Thus
let $\mathbf A$ be a witness of $d_{\mathbf M(\X,D)}(a,m)$ and let $\mathbf B$ be a witness of $d_{\mathbf M(\X,D)}(b,m)$.


By expanding the definition of $d_{\mathbf M(\X, D)}$ and using the triangle inequality we have:
\begin{eqnarray}
d_{\mathbf M(\X,D)}(a,m) + d_{\mathbf M(\X,D)}(b,m) &= & D(\mathbf A)+d_\U(\mathbf A,\dr a)+ D(\mathbf B)+d_\U(\mathbf B,\dr b)\nonumber\\
&\geq&d_\U(\mathbf A,\mathbf B)+d_\U(\mathbf A,\dr a)+d_\U(\mathbf B,\dr b)\nonumber\\
&\geq&d_\U(\dr a,\dr b)\nonumber\\
&=&d_{\mathbf M(\X, D)}(a,b).\nonumber
\end{eqnarray}

%
%
Because witness $\mathbf A$ is minimal, we have:
\begin{eqnarray}
d_{M(\X,D)}(a,m)&=&D(\mathbf A)+d_\U(\mathbf A,\dr a)\nonumber\\
&\leq& D(\mathbf B)+d_\U(\mathbf B,\dr a)\nonumber\\
&\leq& D(\mathbf B)+d_\U(\mathbf B,\dr b) + d_U(\dr a,\dr b)\nonumber\\
&=&d_{\mathbf M(\X,D)}(a,b) + d_{\mathbf M(\X,D)}(b,m).\nonumber
\end{eqnarray}
In a complete analogy we have $d_{\mathbf M(\X,D)}(b,m)\leq d_{\mathbf M(\X,D)}(a,b) + d_{\mathbf M(\X,D)}(a,m)$.

The case when $\dr a$ belongs to $\X$ can be handled similarly if we put $\mathbf A=\dr a$.

Now we show that $d_{\mathbf M(\X,D)}$ has the PM property.

Recall that we have to prove that for each $a,b$ incomparable by $\leq_{\mathbf
M(\X,D)}$ there exists $c$ such that:
\begin{enumerate}
\item $c\LeqUp_{\mathbf M(\X,D)} a$, 
\item $c\LeqUp_{\mathbf M(\X,D)} b$, 
\item $\d_\U(a,b)=\d_\U(a,c)+\d_\U(c,b)$. 
\end{enumerate}
The case $a,b\neq m$ follows directly from the definition of $d_{M(\X,D)}$.  For $a=m$ we can put $c=\max_\mathbf A$ (where $\mathbf A$ is the witness of $d_{M(\X,D)}(a,m)$).

The triplet $\mathbf M(\X,D)$ is complete ({\bf 4.}) and thus
$\mathbf M(\X,D)\in \U$.  By construction it directly follows that $\mathbf M(\X,D)$ mentions every $\mathbf A\in \X$ with the desired distance. By $ii.$ we have $d_\U(\mathbf A,\mathbf M(\X,D))=d_{\mathbf M(\X,D)}(\max_\mathbf A,m)=D(\mathbf A)$.

This proves Theorem \ref{generic} of the finite presentation of $\Urysohn_\Q$.

\end{proof}
\section{The generic partial order revisited}
While it is not immediately obvious, the presentation of the Urysohn space really
builds upon ideas of the earlier finite presentation of the generic partial order shown
in Chapter \ref{homposetchapter}. We can restate a finite presentation of
Definition \ref{Hdef} as follows:

\medskip
\noindent
{\bf 1.} Triple $\mathbf A=(A,\LeqUp_\mathbf A,\leq_\mathbf A)$ is a {\em $\HP$-triplet}  if and only if
\begin{itemize}
\item $A$ is a finite set,
\item Relation $\LeqUp_\mathbf A$ is a partial order on A,
\item Relation $\leq_\mathbf A$ is a partial order on A.
\end{itemize}
As in Section \ref{urysohnsection2}, we say that
$\HP$-triplets $\mathbf A=(A,\LeqUp_\mathbf A,\leq_\mathbf A)$ and
$\mathbf B=(B,\LeqUp_\mathbf B,\leq_\mathbf B)$ are said to be {\em isomorphic}
if there exists a bijection $\varphi:A\to B$ which is both isomorphism of partial orders $(A,\LeqUp_\mathbf A)$ and $(B,\LeqUp_\mathbf B)$ and partial orders $(A,\leq_\mathbf A)$ and $(B,\leq_\mathbf B)$. The downset $\{x:x\LeqUp_\mathbf A a\}$ will be denoted by $\dr a$. (There will be no downsets with respect to $\leq_\mathbf A$.)

\medskip
\noindent
{\bf 2.} $\HP$-triplets are {\em proper} if no two downsets considered as a $\HP$-triplets are isomorphic and if $(A,\LeqUp_\mathbf A)$ has both a greatest and a smallest element (denoted by $\max_\mathbf A$ and $\min_\mathbf A$ respectively).

\medskip
\noindent
{\bf 3.} $\leq_A$ is said to be {\em induced by edges of $\LeqUp_A$} if for every
$a, a'\in A$, $a\leq_\mathbf A a'$ which are incomparable in $\LeqUp_\mathbf A$ there exists $a''\in A,a''\LeqUp_\mathbf A a, a''\LeqUp_\mathbf A a'$ such that $a\leq_\mathbf A a''\leq_\mathbf A a'$.

\medskip
\noindent
{\bf 4.} A proper $\HP$-triplet $\mathbf A=(A,\LeqUp_\mathbf A,\leq_\mathbf A)$ where $\leq_\mathbf A$ is induced by edges of $\LeqUp_\mathbf A$ is said to be complete
if the following holds:
\begin{enumerate}
\item $\min_\mathbf A=\emptyset$.
\item For every $a\in A$ holds:
$$a=\{(b,-1):b\in\dr a,a\neq b,b\leq_\mathbf A a\}\cup
\{(b,1):b\in\dr a,a\neq b,a\leq_\mathbf A b\}\cup$$
$$\cup\{(b,0):b\in\dr a,a\neq b,a\nleq_\mathbf A b,b\nleq_\mathbf A a\}.$$
\end{enumerate}

Denote by $\HP$ the set of all complete $\HP$-triplets.
Complete triplets induce a partial order denoted by $\HPleq$:
\begin{defn}
\label{defposet}
For $\mathbf A=(A,\LeqUp_\mathbf A,\leq_\mathbf A)$, $\mathbf B=(B,\LeqUp_\mathbf B,\leq_\mathbf B)\in \HP$
we write $\mathbf A\HPleq \mathbf B$ if and only if there exists $a\in A$ such that $a\in B$ and $\max_\mathbf A\leq_\mathbf A a$, $a\leq_\mathbf B \max_\mathbf B$.
\end{defn}

As in Section \ref{Urysohnsection} we can then prove:
\begin{thm}\label{thm31}~
\begin{enumerate}
\item $(\HP,\HPleq)$ is a partially ordered set.
\item $(\HP,\HPleq)$ is isomorphic to the generic partial order $(\Poset,\leq_\Poset)$.
\end{enumerate}
\end{thm}
\section{Alternative representations}
\label{alternative}
The rational Urysohn space and the generic partial order are uniquely determined (up to isometry or isomorphism). Thus also our finite presentations (Theorems \ref{generic} and \ref{thm31}) describe the same
objects. Of course our finite presentation is not unique (as these representations may use different languages). Here is another      
variant motivated by the above presentation and Kat\v etov's construction already mentioned in \ref{uryintro}.
This construction, which we denote by $({\U'}, d_{\U'})$, is perhaps even more ``concise'':

\begin{defn}
The vertices of $\U'$ are functions $f$ such that:
\begin{enumerate}
\item The domain $D_f$ of $f$ is a finite (possibly empty) set of functions.
\item The range of $f$ is a subset of the positive rationals.
\item For every $f'\in D_f$ and $f''\in D_{f'}$, we have $f''\in D_f$.
\item $D_f$ using metric $d_{\U'}$ defined below forms a metric space.
\item The function $f$ defines an extension of metric space on vertices $D_f$ by adding a new vertex as in the proof of Theorem \ref{generic}.
\end{enumerate}
The metric $d_{\U'}(f,g)$ is defined by:
\begin{enumerate}
  \item if $f=g$ then
$d_{\U'}(f,g)=0,$
  \item if $f\in D_g$ then
$d_{\U'}(f,g)=g(f),$
  \item if $g\in D_f$ then 
$d_{\U'}(f,g)=f(g),$
  \item if none of the above hold then
$d_{\U'}(f,g)=min_{h\in D_f\cap D_g} f(h)+g(h).$
\end{enumerate}
\end{defn}

\begin{thm}
$({\U'}, d_{\U'})$ is the generic metric space.
\end{thm}

\begin{proof}[Proof (sketch)]
This follows from our Definition \ref{urysohndef} by encoding PPM-triplets as functions. 
\end{proof}
 Note that also our description of the generic partial orders (Definition \ref{defposet}) leads to a similar reformulation.

Urysohn space was also studied in the context of constructive mathematics and
effective constructibility.  \cite{constructiveurysohn} show that techniques
presented here (in language of classical mathematics) are essentially
constructive.

\section{Other metrics, other structures}

It is obvious that the finite presentation given in Section \ref{urysohnsection2} for $\Urysohn_\Q$ can be easily modified for Urysohn spaces with the rational metrics
restricted to some interval (say $\Urysohn \cap [0,1]$ or $\Urysohn \cap [0,a]$).  Such variants of the Urysohn spaces has been thoroughly investigated in \cite{Nguyen, KPT}
where the Urysohn space with $S$-valued metric was denoted by $\Urysohn_S$. $\Urysohn_S$ need not exist as is demonstrated by the failure of the amalgamation property. However this is characterized in \cite{d9} by {\em 4-values condition}.

\begin{defn}
\label{4values}
 Let $S\subseteq[0,+\infty]$. $S$ satisfies the {\em 4-values condition}
when for every $s_0, s_1, s'_0, s'_1 \in S$, if there is $t \in S$ such that:
$$|s_0 - s_1| \leq t \leq s_0 + s_1
\mathrm{~and~}|s'_0-s'_1|\leq t\leq s'_0+s'_1$$
then there is $u \in S$ such that:
$$|s_0 - s'_0| \leq u \leq s_0 + s'_0
\mathrm{~and~}|s_1 - s'_1| \leq u \leq s_1 + s'_1.$$
\end{defn}
\begin{thm}[Delhomm\'e, Laflamme, Pouzet, Sauer \cite{d9}]
Let $S\subseteq[0,+\infty]$.\\The following conditions are equivalent:
\begin{enumerate}
\item There is a countable ultrahomogeneous metric space $\Urysohn_S$ with
distances in $S$ into which every countable metric space with distances in $S$
embeds isometrically.
\item $S$ satisfies the 4-values condition.
\end{enumerate}
\end{thm}

One can prove (by a cardinality argument) that the 4-value property does not suffice for the existence of a finite presentation of $\Urysohn_S$. However we have the following:

A class $\cal K$ of rational finite metric spaces is said to be {\em triangle
axiomatized} if $A\in\cal K$ if and only if every $3$-point subspace of $A$
belongs to $\cal K$. An ultrahomogeneous metric space $\X$ is said to be
triangle axiomatized if the class of all finite subspaces is triangle
axiomatized. We can prove \cite{novyclanek}:

\begin{thm}
Every ultrahomogeneous space $\cal X$ which is triangle axiomatized
and where there is formula $\varphi$ deciding whether given metric space
on with 3 vertices is subspace of $\cal X$ has a finite presentation.
\end{thm}

Triangle axiomatized classes include classes of ultrametric spaces thoroughly investigated recently in \cite{Nguyen}.

\part{Embedding-universal structures}

\chapter{Some examples of universal partial orders}
\label{posetschapter}
In this chapter we present several simple constructions which yield (countably)
universal partial orders. Such objects are interesting on their own and were
intensively studied in the context of universal algebra and categories. For
example, it is a classical result of Pultr and Trnkov\'a \cite{Pultr} that
finite graphs with the homomorphism order forms a quasi-order that embeds a
countably universal partially ordered set. Extending and completing
\cite{HN-Posets} we give here several constructions which yield universal
partial orders.  These constructions include: 
\begin{enumerate}
\item The order $(\Words,\leq_\Words)$ on sets of words in the alphabet $\{0,1\}$.
\item The dominance order on the binary tree $(\Bintree,\leq_\Bintree)$.
\item The inclusion order of finite sets of finite intervals $(\Intervals,\leq_\Intervals)$.
\item The inclusion order of convex hulls of finite sets of points in the plane $(\Convex,\leq_\Convex)$.
\item The order of piecewise linear functions on rationals $(\Functions,\leq_\Functions)$.
\item The inclusion order of periodic sets $(\Periodic,\subseteq)$.
\item The order of sets of truncated vectors (generalization of orders of vectors
of finite dimension) $(\TV,\leq_\TV)$.
\item The orders implied by grammars on words $(\Grammar,\leq_\Grammar)$.
\item The homomorphism order of oriented paths $(\Paths,\leq_\Paths)$.
\end{enumerate}

Note that with universal partial orders we have more freedom (than with the generic
partial order) and as a consequence we give a perhaps surprising variety of finite
presentations.

We start with a simple representation by means of finite sets of binary words. This
representation seems to capture properties of such a universal partial order very
well and it will serve as our ``master'' example.  In most other cases we prove the
universality of some particular partial order by finding a mapping from the
words representation into the structure in question.  This technique will be
shown in several applications in the next sections. While some of these
structures are known be universal, see e.g. \cite{hedrlin,nesu,HN-paths}, in several
cases we can prove the universality in a new, we believe, much easier way.
The embeddings of structures are presented as follows (ones denoted by dotted
lines are not presented in this thesis, but references are given). 

\medskip
~

\centerline{\includegraphics[width=16cm]{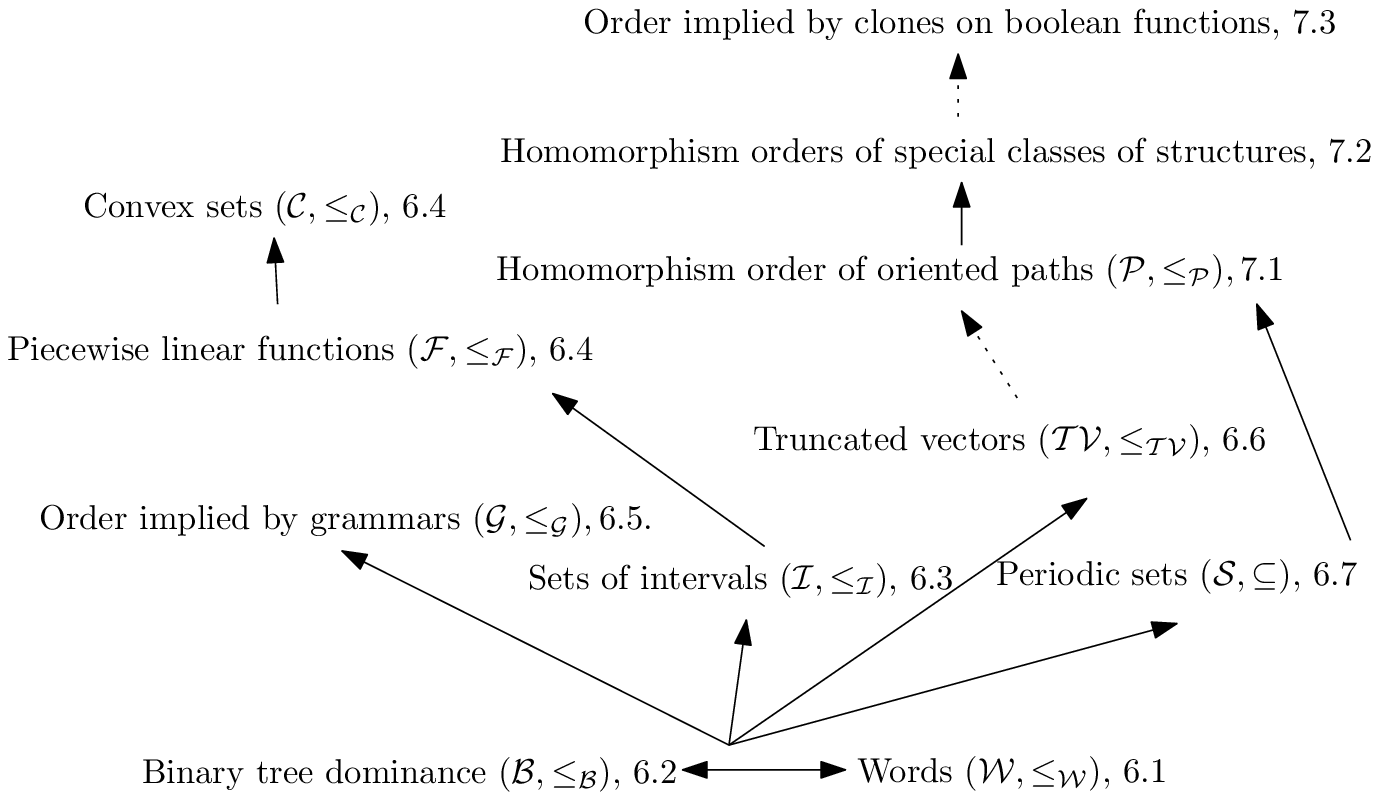}}

\medskip
~

The finite presentation of the generic partial order was given in Chapter
\ref{homposetchapter}.  A bit surprisingly this is the only known one.  The
constructions of universal partial orders are easier, but they are often not
generic. We discuss reasons why other structures fail to be ultrahomogeneous.  In
particular we will look for gaps in the partial order.  Recall that the {\em
gap} in a partial order $(P,\leq_P)$ is a pair of elements $v, v'\in P$ such that
$v<_\Bintree v'$.  A partial order having no gaps is called {\em dense}. We will
show examples of universal partial orders both with gaps  and without gaps but
still failing to be generic.

To prove the universality of a given partially ordered set 
is often a difficult task \cite{hedrlin,Pultr,HN-trees,nesu}.
However, the individual proofs, even if developed independently, use similar tools.
We demonstrate this by isolating a ``master'' construction (in Section \ref{wordsection}). This construction is then embedded into partial orders defined
by other structures (as listed above). We shall see that the representation of
this particular order is flexible enough to simplify further embeddings.

\section{Word representation}
\label{wordsection}

The set of all words over the alphabet $\Sigma=\{0,1\}$ is denoted by $\{0,1\}^*$.
For words $W,W'$ we write $W\leq_w W'$ if and only if $W'$ is an initial segment (left
factor) of $W$. Thus we have, for example, $\{011000\}\leq_w \{011\}$ and
$\{010111\}\nleq_w \{011\}$.

\begin{defn}
\label{Wordsdef}

Denote by $\Words$ the class of all finite subsets $A$ of $\{0,1\}^*$ such that no distinct words $W,W'$ in $A$ satisfy $W\leq_w W'$. For
$A,B\in \Words$ we put $A\leq_\Words B$ when for each $W\in A$ there exists
$W'\in B$ such that $W\leq_w W'$.
\end{defn}

Obviously $(\Words,\leq_\Words)$ is a partial order (antisymmetry follows from the fact that $A$ is an antichain in the order $\leq_w$).

\begin{defn}
For a set $A$ of finite words denote by $\min A$ the set of all minimal words
in $A$ (i.e. all $W\in A$ such that there is no $W'\in A$ satisfying $W'<_wW$).
\end{defn}

Now we show that there is an on-line embedding of any finite partial order to $(\Words,\leq_\Words)$.
Let $[n]$ be the set $\{1,2,\ldots, n\}$.  The partial orders will be restricted to those whose vertex sets are sets $[n]$ (for some $n>1$) and
the vertices will always be embedded in the natural order.  Given a partial order
$([n],\leq_P)$ let $([i],\leq_{P_i})$ denote the partial order induced by
$([n],\leq_P)$ on the set of vertices $[i]$. 

Our main construction is the function $\Psi$ mapping partial orders $([n],\leq_P)$ to elements of $(\Words,\leq_\Words)$ defined as follows:
\begin{defn}
\label{algu}
Let $L([n],\leq_P)$ be the union of all $\Psi([m],\leq_{P_m})$, $m<n$, $m\leq_P n$.

Let $U([n],\leq_P)$ be the set
of all words $W$ such that $W$ has length $n$, the last letter is $0$ and
for each $m<n,n\leq_P m$ there is a $W'\in \Psi([m],\leq_{P_m})$ such that $W$ is an initial
segment of $W'$.

Finally, let  $\Psi([n],\leq_P)$ be $\min (L([n],\leq_P) \cup U([n],\leq_P))$.  

In particular, $L([1],\leq_P)=\emptyset, U([1],\leq_P)=\{0\},\Psi ([1],\leq_P)=\{0\}$.
\end{defn}

The main result of this section is the following:

\begin{thm}
\label{Pembed}
Given a partial order $([n],\leq_P)$ we have:
\begin{enumerate}
\item For every $i,j\in[n]$, $$i\leq_P j \hbox{ if and only if } \Psi([i],\leq_{P_i})\leq_\Words \Psi([j],\leq_{P_j})$$ and $$\Psi([i],\leq_{P_i}) = \Psi([j],\leq_{P_j}) \hbox{ if and only if } i=j.$$ (This says that the mapping $\Phi(i) = \Psi([i],\leq_{P_i})$ is an embedding of $([n],\leq_P)$ into $(\Words,\leq_\Words)$),
\item for every $S\subseteq [n]$ there is a word $W$ of length $n$ such that for each $k\leq n$, $\{W\}\leq_\Words\Psi([k],\leq_{P_k})$ if and only if either $k\in S$ or there is a $k'\in S$ such that $k'\leq_P k$.
\end{enumerate}
\end{thm}
The on-line embedding $\Phi$ is illustrated by the following example:
\eject
  \begin{figure}[t]
    \begin{center}
      {\def\IPEfile{poset.ipe}\begingroup
  \catcode`\%=9\catcode`\!=0\catcode`\-=11\input{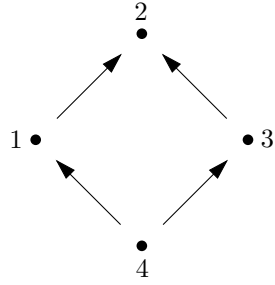}}
    \end{center}
    \caption{The partial order $([4],\leq_P)$.}
    \label{poset}
  \end{figure}
\begin{example}
The partial order $([4],\leq_P)$ depicted in Figure~\ref{poset} has the following values of $\Psi([k],\leq_{P_k}), k=1,2,3,4$:

$$
\begin{array}{lll}
L([1],\leq_{P_1})=\emptyset& U([1],\leq_{P_1})=\{0\}& \Psi([1],\leq_{P_1})=\{0\},\\
L([2],\leq_{P_2})=\{0\}& U([2],\leq_{P_2})=\{00,10\}& \Psi([2],\leq_{P_2})=\{0,10\},\\
L([3],\leq_{P_3})=\emptyset& U([3],\leq_{P_3})=\{000,100\}& \Psi([3],\leq_{P_3})=\{000,100\},\\
L([4],\leq_{P_4})=\emptyset& U([4],\leq_{P_4})=\{0000\}& \Psi([4],\leq_{P_4})=\{0000\}.\\
\end{array}
$$
\end{example}
\begin{proof}[Proof (of Theorem \ref{Pembed})]
We proceed by induction on $n$.

The theorem obviously holds for $n=1$. 

Now assume that the theorem holds for every partial order $([i],\leq_{P_i})$, $i=1,\ldots, n-1$.

We first show that $2.$ holds for $([n],\leq_P)$.  Fix $S\subseteq\{1,2,\ldots, n\}$.
Without loss of generality assume that for each $m\leq n$ such that there is an
$m'\in S$ with $m'\leq_P m$, we also have $m\in S$ (i.e. $S$ is closed upwards).  By the induction hypothesis, there is a word $W$ of length $n-1$ such that for each $n'<n$, $\{W\}\leq_\Words
\Psi([n'],\leq_{P_{n'}})$ if and only if $n'\in S$. Given the word $W$ we can construct a word $W'$ of length $n$ such that $\{W'\}\leq_\Words \Psi([n'],\leq_{P_{n'}})$ if and only if $n'\in S$. To see this, consider the following cases: 
\begin{enumerate}
\item $n\in S$
\begin{enumerate}
\item $\{W\}\leq_\Words \Psi([n],\leq_P)$. Put $W'=W0$. Since $\{W'\}\leq_\Words\{W\}$, $W'$ obviously has the property.

\item $\{W\}\nleq_\Words \Psi([n],\leq_P)$.  In this case we have $m\in S$ for each $m<n,n\leq_P m$, and thus $\{W\}\leq_\Words \Psi([m],\leq_{P_m})$.  By the definition of $\leq_\Words$, for each such $m$ we have $W''\in \Psi([m],\leq_{P_m})$ such that $W''$ is an initial segment of $W$. This implies that $W0$ is in $U([n],\leq_P)$ and thus $\{W\}\leq_{\Words}\Psi([n],\leq_P)$, a contradiction.
\end{enumerate}
\item $n\notin S$
\begin{enumerate}
\item $\{W\}\nleq_\Words \Psi([n],\leq_P)$. In this case we can put either $W'=W0$ or $W'=W1$.
\item $\{W\}\leq_\Words \Psi([n],\leq_P)$.  We have $\{W\}\nleq_{\Words} L([n],\leq_P)$---otherwise we
would have $\{W\}\leq_{\Words} \Psi([m],\leq_{P_m})\leq_{\Words} \Psi([n],\leq_P)$ for some $m<n$ and thus $n\in
S$. Since $U([n],\leq_P)$ contains words of length $n$ whose last digit is 0 putting $W'=W1$ gives $\{W'\}\nleq_\Words U([n],\leq_P)$ and thus also $\{W'\}\nleq_\Words \Psi([m],\leq_{P_m})$.
\end{enumerate}
\end{enumerate}

\bigskip
\noindent
This finishes the proof of property $2.$

Now we prove $1.$  We only need to verify that for $m=1,2,\ldots, n-1$ we have $\Psi([n],\leq_P${}$)\leq_\Words \Psi([m],\leq_{P_m})$ if and only if $n\leq_P m$ and $\Psi([m],\leq_{P_m})\leq_\Words \Psi([n],\leq_P)$ if and only if $m\leq_P n$.  The rest follows by induction. Fix $m$ and consider the following cases:

\begin{enumerate}
\item
{\em $m\leq_P n$ implies $\Psi([m],\leq_{P_m})\leq_\Words \Psi([n],\leq_P)$}: This follows easily from the fact that every word in $\Psi([m],\leq_{P_m})$ is in $L([n],\leq_P)$ and the initial segment of each word in $L([n],\leq_P)$ is in $\Psi([n],\leq_P)$).

\item
{\em $n\leq_P m$ implies $\Psi([n],\leq_P)\leq_\Words \Psi([m],\leq_{P_m})$}: $U([n],\leq_P)$ is a maximal set of words of length $n$ with last digit $0$ such that $U([n],\leq_P)\leq_\Words \Psi([m'],\leq_{P_{m'}})$ for each $m'<n,n\leq_P m'$, in particular for $m'=m$. It suffices to show that $L([n],\leq_P)\leq_\Words \Psi([m],\leq_{P_m})$. For $W\in L([n],\leq_P)$, we have an $m''$, $m''\leq_P n\leq_P m$, such that $W\in \Psi([m''],\leq_{P_{m''}})$.  From the induction hypothesis $ \Psi([m''],\leq_{P_{m''}})\leq_\Words \Psi([m],\leq_{P_m})$---in particular the initial segment of $W$ is in $\Psi([m],\leq_{P_m})$.

\item
{\em $\Psi([m],\leq_{P_m})\leq_\Words \Psi([n],\leq_P)$ implies $m\leq_P n$}:
Since $U([n],\leq_P)$ contains words longer than any word of $m$, we have $\Psi([m],\leq_{P_m})\leq_\Words L([n],\leq_P)$.  By $2.$ for $S=\{m\}$ there is a word $W$ such that $\{W\}\leq_\Words  \Psi([m'],\leq_{P_{m'}})$ if and only if $m\leq_P m'$.  Since $\{W\}\leq_\Words L([n],\leq_P)$, we have an $m'$ such that $m\leq_P m'\leq_P n$.

\item
{\em $\Psi([n],\leq_P)\leq_\Words \Psi([m],\leq_{P_m})$ implies $n\leq_P m$}: We have $\Psi([n],\leq_P)\leq_\Words \Psi([m],\leq_{P_m})$.   By $2.$ for $S=\{n\}$ there is a word $W$ such that $\{W\}\leq_\Words  \Psi([m'],\leq_{P_{m'}})$ if and only if $n\leq_P m'$. Since $\{W\}\leq_\Words \Psi([m],\leq_{P_m})$ we also have $n\leq_P m$.
\end{enumerate}

\end{proof}
\begin{corollary}
The partial order $(\Words,\leq_\Words)$ is universal.
\end{corollary}

Note that $\Words$ fails to be a ultrahomogeneous partial order. For example the empty set is the minimal element. $\Words$ is also not dense as shown by the following example: $$A=\{0\}, B=\{00,01\}.$$

This is not unique gap---we shall characterize all gaps in $(\Words,\leq_\Words)$ after
reformulating it in a more combinatorial setting in Section \ref{dominance}.

\section{Dominance in the countable binary tree}
\label{dominance}
As is well known, the Hasse diagram of the partial order $(\{0,1\}^*, \leq_w)$ (defined in Section \ref{wordsection})
forms a complete binary tree $T_u$ of infinite depth. Let $r$ be its root
vertex (corresponding to the empty word).  Using $T_u$ we can reformulate our
universal partial order as:

\begin{defn}
The vertices of $(\Bintree,\leq_\Bintree)$ are finite sets $S$ of vertices of $T_u$ such that there is no
vertex $v\in S$ on any path from $r$ to $v'\in S$ except for $v'$. (Thus $S$ is a finite antichain in the order of the tree $T$.)

We say that $S'\leq_\Bintree S$ if and only if for each path from $r$ to $v\in S$ there is a vertex
$v'\in S'$.
\end{defn}

\begin{corollary}
The partially ordered set $(\Bintree,\leq_\Bintree)$ is universal.
\end{corollary}
\begin{proof}
$(\Bintree,\leq_\Bintree)$ is just a reformulation of $(\Words,\leq_\Words)$ and thus both partial orders are isomorphic.
\end{proof}
  \begin{figure}[t]
    \begin{center}
      \includegraphics[width=10cm]{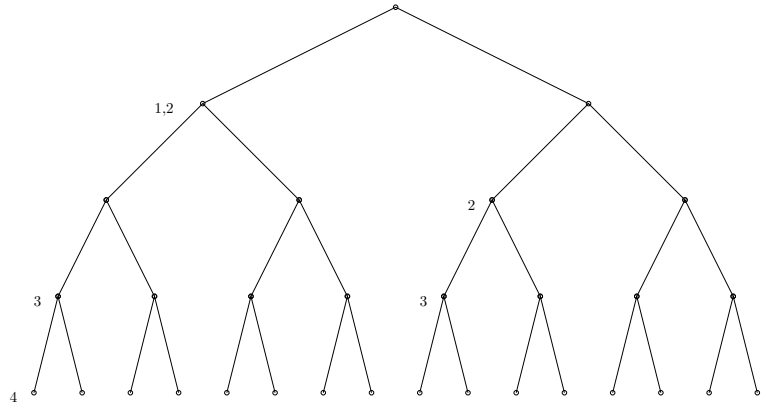}
    \end{center}
    \caption{Tree representation of $([4],\leq_P)$ (Figure~\ref{poset}).}
    \label{strom-plny}
  \end{figure}
Figure~\ref{strom-plny} shows a portion of the tree $T$ representing
the same partial order as in Figure~\ref{poset}.

The partial order $(\Bintree,\leq_\Bintree)$ offers perhaps a better intuitive
understanding as to how the universal partial order is built from the very simple partial
order $(\{0,1\}^*, \leq_w)$ by using sets of elements instead of single
elements.  Understanding this makes it easy to find an embedding
of $(\Words,\leq_\Words)$ (or equivalently $(\Bintree,\leq_\Bintree)$) into a new structure by first looking for a way to represent the partial order $(\{0,1\}^*, \leq_w)$ within the new structure and then a way to represent subsets of $\{0,1\}^*$. This idea will be applied several times in the following sections.

Now we characterize gaps.

\begin{prop}
$S<S'$ is a gap in $(\Bintree,\leq_\Bintree)$ if and only if there exists an $s'\in S'$ such that
\begin{enumerate}
\item there is a vertex $s\in S$ such that both sons $s_0,s_1$ of $s$ in the tree $T$ are in $S'$,
\item $S\setminus\{s_0,s_1\}=S'\setminus\{s\}$.
\end{enumerate}
\end{prop}
This means that all gaps in $\Bintree$ result from replacing a member by its two sons.
\begin{proof}
Clearly any pair $S<S'$ satisfying $1.$, $2.$ is a gap (as any $S\leq_\Bintree S''\leq_\Bintree S'$ has to contain $S'\setminus\{s\}$, and either $s$ or the two vertices $s_0$, $s_1$).  

Let $S\leq_\Bintree S'$ be a gap.  If there are distinct vertices $s'_1$ and $s'_2$ in $S'$ and $s_1,s_2\in S$ are such that $s_i\leq s'_i$, i=1,2, then $S''$ defined as $min(S\setminus\{s_1\})\cup\{S'_1\}$ satisfies $S<_\Bintree S''<_\Bintree S'$.

Thus there is only one $S'\in S'\setminus S$ such that $s'>s$ for an $s\in S$. However then there is only one such $s'$ (so if $s_1,s_2$ are distinct then $S<S\setminus\{s_2\}<S'$). Moreover it is either $s=s'0$ or $s=s'1$. Otherwise $S<S'$ would not be a gap.
\end{proof}

The abundance of gaps indicates that $(\Bintree,\leq_\Bintree)$ (or
$(\Words,\leq_\Words)$) are redundant universal partial orders. This makes
them, in a way, far from being generic, since the generic partial order has no
gaps.  The next section has a variant of this partial order avoiding this
problem.  On the other hand gaps in partial orders are interesting and are
related to dualities, see \cite{Trotter, NZhu}.

\section{Intervals}
We show that the vertices of $(\Words,\leq_\Words)$ can be coded by geometric
objects ordered by inclusion.  Since we consider only countable structures we restrict ourselves to objects formed from rational numbers.

While the interval on rationals ordered by inclusion can represent infinite increasing chains,
decreasing chains or antichains, obviously this interval order has 
dimension 2 and thus fails to be universal. However
considering multiple intervals overcomes this limitation:
\begin{defn}
The vertices of $(\Intervals,\leq_\Intervals)$ are finite sets $S$ of closed
disjoint intervals  $[a,b]$ where $a$, $b$ are rational numbers and
$0\leq a<b\leq 1$.

We put $A\leq_\Intervals B$ when every interval in $A$ is covered by some interval of $B$. 
\end{defn}

In the other words elements of $(\Intervals,\leq_\Intervals)$ are finite sets of pairs of rational
numbers.  $A\leq_\Intervals B$ holds if for every $[a,b]\in A$, there is
an $[a',b']\in B$ such that $a'\leq a$ and $b\leq b'$.

\begin{defn} 
 A word $W=w_1w_2\ldots w_t$ on the alphabet $\{0,1\}$ can be considered as a number $0\leq n_W\leq 1$ 
with ternary expansion: $$n_W=\sum_{i=1}^{t} w_i{1\over 3^i}.$$  

For $A\in \Words$, the representation of $A$ in $\Intervals$ is then the following set of
intervals:

$$\Embed{\Words}{\Intervals}(A)=\{[n_W,n_W+{2\over 3^{|W|+1}}]: W\in A\}.$$
\end{defn}

The use of the ternary base might seem unnatural---indeed the binary base would
suffice. The main obstacle to using the later is that the embedding of $\{00,01\}$ would be two
intervals adjacent to each other overlapping in single point.  One would need
to take special care when taking the union of such intervals---we avoid this by
using ternary numbers.

\begin{lem}
$\Embed{\Words}{\Intervals}$ is a embedding of $(\Words,\leq_{\Words})$ into $(\Intervals,\leq_{\Intervals})$.

\end{lem}
\begin{proof}
It is sufficient to prove that for $W$, $W'$ there is an interval $[n_W,n_W+{1\over
3^{|W|}}]$ covered by an interval $[n_{W'},n_{W'}+{1\over 3^{|W'|}}]$ if and only if $W'$ is
initial segment of $W$. This follows easily from the fact that
intervals represent precisely all numbers whose ternary expansion starts with
$W$ with the exception of the upper bound itself.
\end{proof}

\begin{example}
The representation of $([4],\leq_P)$ as defined by Figure~\ref{poset} in $(\Intervals,\leq_\Intervals)$ is:
$$
\begin{array}{lllll}
\Embed{\Words}{\Intervals}(\Psi([1],\leq_{P_1}))&=&\Embed{\Words}{\Intervals}(\{0\})&=&\{(0,{2\over 3^2})\},\\
\Embed{\Words}{\Intervals}(\Psi([2],\leq_{P_2}))&=&\Embed{\Words}{\Intervals}(\{0,10\})&=&\{(0,{2\over 3^2}),({1\over 3}, {1\over 3}+{2\over 3^3})\},\\
\Embed{\Words}{\Intervals}(\Psi([3],\leq_{P_3}))&=&\Embed{\Words}{\Intervals}(\{000,100\})&=&\{(0,{2\over 3^4}),({1\over 3},{1\over 3}+{2\over 3^4})\},\\
\Embed{\Words}{\Intervals}(\Psi([4],\leq_{P_4}))&=&\Embed{\Words}{\Intervals}(\{0000\})&=&\{(0,{2\over 3^5})\}.
\end{array}
$$
\end{example}

\begin{corollary}
The partial order $(\Intervals,\leq_\Intervals)$ is universal.
\end{corollary}

The partial order $(\Intervals,\leq_\Intervals)$ differs significantly from
$(\Words,\leq_\Words)$ by the following:

\begin{prop}
The partial order $(\Intervals,\leq_\Intervals)$ has no gaps (is dense).
\end{prop}
\begin{proof}
Take $A,B\in \Intervals$, $A<_\Intervals B$. Because all the intervals in both $A$ and $B$ are closed and disjoint, there must be at least one interval $I$ in $B$ that is not fully covered by intervals of $A$ (otherwise we would have $B\leq_\Intervals A$). We may construct an element $C$ from $B$ by shortening the interval $I$  or splitting it into two disjoint intervals in a way such that $A<_\Intervals C<_\Intervals B$ holds.
\end{proof}
Consequently the presence (and abundance) of gaps in most of the universal
partial orders studied is not the main obstacle when looking for
representations of partial orders.  It is easy to see that
$(\Intervals,\leq_\Intervals)$ is not generic.

By considering a variant of $(\Intervals,\leq_\Intervals)$ with open 
(instead of closed) intervals we obtain a universal partial order
$(\Intervals',\leq_{\Intervals'})$ with gaps. The gaps are similar to the ones in
$(\Bintree,\leq_\Bintree)$ created by replacing interval $(a,b)$ by two
intervals $(a,c)$ and $(c,d)$.  Half open intervals give a
quasi-order containing a universal partial order.

\section{Geometric representations}
The representation as a set of intervals might be considered an
artificially constructed structure.  Partial orders represented by geometric
objects are studied in \cite{Alon}. It is shown that objects with $n$ ``degrees
of freedom'' cannot represent all partial orders of dimension $n+1$. It follows that
convex hulls used in the representation of the generic partial order cannot be defined by a constant
number of vertices.  We will show that even the simplest geometric objects with
unlimited ``degrees of freedom'' can represent a universal partial order.

\begin{defn}
Denote by $(\Convex,\leq_\Convex)$ the partial order whose vertices are all convex hulls
of finite sets of points in $\mathbb Q^2$, ordered by inclusion.
\end{defn}
This time we will embed $(\Intervals,\leq_\Intervals)$ into $(\Convex,\leq_\Convex)$.
\begin{defn}
For every $A\in \Intervals$ denote by $\Embed{\Intervals}{\Convex}(A)$ the convex hull generated by the points: $$(a,a^2),
({{a+b}\over 2}, ab), (b,b^2)\hbox{, for every } (a,b)\in A.$$
\end{defn}
  \begin{figure}[t!h]
    \begin{center}
1:
      \includegraphics[width=5cm]{t.1}
2:
      \includegraphics[width=5cm]{t.2}\\
3:
      \includegraphics[width=5cm]{t.3}
4:
      \includegraphics[width=5cm]{t.4}
    \end{center}
    \caption{Representation of the partial order $([4],\leq_P)$ in $(\Convex,\leq_\Convex)$.}
    \label{posetconvex}
  \end{figure}
See Figure~\ref{posetconvex} for the representation of the partial order in Figure~\ref{poset}.
\begin{thm}
$\Embed{\Intervals}{\Convex}$ is an embedding of $(\Intervals,\leq_\Intervals)$ to $(\Convex,\leq_\Convex)$.
\end{thm}
\begin{proof}
All points of the form $(x,x^2)$ lie on a convex parabola $y=x^2$. The points $
({{a+b}\over 2}, ab)$ are the intersection of two tangents of this parabola at the
points $(a,a^2)$ and $(b,b^2)$. Consequently all points in the construction of
$\Embed{\Intervals}{\Convex}(A)$ lie in a convex configuration.

We have $(x,x^2)$ in the convex hull $\Embed{\Intervals}{\Convex}(A)$ if and only if there is $[a,b]\in A$ such that $a\leq x\leq b$.  Thus for $A,B\in \Intervals$ we have $\Embed{\Intervals}{\Convex}(A)\leq_\Convex\Embed{\Intervals}{\Convex}(B)$ implies $A\leq_\Intervals B$.

To see the other implication, observe that the convex hull of $(a,a^2)$, $({{a+b}\over 2}, ab)$, $(b,b^2)$ is a subset of the convex hull of $(a',a'^2)$, $({{a'+b'}\over 2}, a'b'),$ $(b',b'^2)$ for every $[a,b]$ that is a  subinterval of $[a',b']$.
\end{proof}

We have:
\begin{corollary}
The partial order $(\Convex,\leq_\Convex)$ is universal.
\end{corollary}

\begin{remark}
Our construction is related to Venn diagrams.  Consider the partial order $([n],\leq_P)$.  For the empty relation $\leq_P$ the representation constructed by $\Embed{\Intervals}{\Convex}(\Embed{\Words}{\Intervals}(\Psi([n],\emptyset)))$ is a Venn diagram, by Theorem \ref{Pembed} $(2.)$. Statement 2 of Theorem \ref{Pembed} can be seen as a Venn diagram condition under the constraints imposed by $\leq_P$.
\end{remark}

The same construction can be applied to functions, and stated in a perhaps more precise manner.
\begin{corollary}
Consider the class $\Functions$ of all convex piecewise linear functions  on the interval $(0,1)$ consisting of a finite set of segments, each with rational boundaries.  Put $f\leq_\Functions g$ if and only if $f(x)\leq g(x)$ for every $0\leq x\leq 1$. Then the partial order $(\Functions,\leq_\Functions$) is universal.
\end{corollary}

Similarly the following holds:
\begin{thm}
Denote by $\Polynoms$ the class of all finite polynomials with rational coefficients.  For $p,q\in \Polynoms$, put $p\leq_\Polynoms q$ if and only if $p(x)\leq q(x)$ for $x\in (0,1)$. The partial order
$(\Polynoms,\leq_\Polynoms)$ is universal.
\end{thm}
The proof of this theorem needs tools of mathematical analysis
and it will appear in \cite{JNS} (jointly with Robert \v S\'amal).

\section{Grammars}

The rewriting rules used in a context-free grammar can be also used to define
a universal partially ordered set.  
\begin{defn}

The vertices of $(\Grammar,\leq_\Grammar)$ are all words over the alphabet $\{\downarrow{},\uparrow{},0,1\}$ created from the word $1$ by the following rules:
$$
\begin{array}{rcl}
1&\to& \downarrow{}11\uparrow{},\\
1&\to& 0.
\end{array}
$$

$W\leq_\Grammar W'$ if and only if $W$ can be constructed from $W'$ by:
$$
\begin{array}{rcl}
1&\to& \downarrow{}11\uparrow{},\\
1&\to& 0,\\
\downarrow{}00\uparrow{}&\to& 0.
\end{array}$$
\end{defn}

$(\Grammar,\leq_\Grammar)$ is a quasi-order: the transitivity of $\leq_\Grammar$ follows from the composition of lexical transformations.

\begin{defn}
Given $A\in \Words$ construct $\Embed{\Words}{\Grammar}$ as follows:

\begin{enumerate}
\item $\Embed{\Words}{\Grammar}(\emptyset)=0$.

\item $\Embed{\Words}{\Grammar}(\{\hbox{empty word}\})=1$.

\item $\Embed{\Words}{\Grammar}(A)$ is defined as the concatenation $\downarrow\Embed{\Words}{\Grammar}(A_0)\Embed{\Words}{\Grammar}(A_1)\uparrow$, where $A_0$ is created from all words of $A$ starting with $0$ with the first digit  removed and $A_1$ is created from all words of $A$ starting with $1$ with the first digit removed.
\end{enumerate}
\end{defn}

\begin{example}
The representation of $([4],\leq_P)$ as defined by Figure~\ref{poset} in $(\Grammar,\leq_\Grammar)$ is as follows (see also the correspondence with the $\Bintree$ representation in Figure~\ref{strom-plny}):
$$
\begin{array}{lllll}
\Embed{\Words}{\Grammar}(\Psi([1],\leq_{P_1}))&=&\Embed{\Words}{\Grammar}(\{0\})&=&\downarrow{}10\uparrow{},\\
\Embed{\Words}{\Grammar}(\Psi([2],\leq_{P_2}))&=&\Embed{\Words}{\Grammar}(\{0,10\})&=&\downarrow{}1\downarrow{}10\uparrow{}\uparrow{},\\
\Embed{\Words}{\Grammar}(\Psi([3],\leq_{P_3}))&=&\Embed{\Words}{\Grammar}(\{000,100\})&=&\downarrow{}\downarrow{}\downarrow{}10\uparrow{}0\uparrow{}\downarrow{}\downarrow{}10\uparrow{}0\uparrow{}\uparrow{},\\
\Embed{\Words}{\Grammar}(\Psi([4],\leq_{P_4}))&=&\Embed{\Words}{\Grammar}(\{0000\})&=&\downarrow{}\downarrow{}\downarrow{}\downarrow{}10\uparrow{}0\uparrow{}0\uparrow{}0\uparrow{}.\\
\end{array}
$$
\end{example}
We state the following without proof as it follows straightforwardly from the definitions.
\begin{prop}
For $A,B\in \Words$ the inequality $A\leq_\Words B$ holds if and only if $\Embed{\Words}{\Grammar}(A)\leq_\Grammar \Embed{\Words}{\Grammar} (B)$.
\end{prop}
$(\Grammar,\leq_\Grammar)$ is a quasi-order. We have:
\begin{corollary}
The quasi-order $(\Grammar,\leq_\Grammar)$ contains a universal partial order.
\end{corollary}

\section{Multicuts and truncated vectors}
\label{TVsection}
A universal partially ordered structure similar to $(\Words,\leq_\Words)$, but
less suitable for further embeddings, was studied in
\cite{hedrlin,nesu,HN-trees}. While the structures defined in these papers are easily shown to be
equivalent, their definition and motivations were different.  \cite{hedrlin} contains the first finite presentation of universal partial order. \cite{nesu} first
used the notion of on-line embeddings to (1) prove the universality of the structure and
(2) as intermediate structure to prove the universality of the homomorphism
order of multigraphs. The motivation for this structure came from the  analogy with
Dedekind cuts and thus its members were called {\em multicuts}.  In
\cite{HN-trees} an essentially equivalent structure with the inequality reversed was
used as an intermediate structure for the stronger result showing the universality of
oriented paths. This time the structure arises in the context of orders of vectors
(as the simple extension of the orders of finite dimension represented by finite
vectors of rationals) resulting in name {\em truncated vectors}.

We follow the presentation in \cite{HN-trees}.

\begin{defn}
\label{TVdef}
Let $\vec{v}=(v_1,\ldots,v_t)$, $\vec{v}'=(v'_1,\ldots,v'_{t'})$ be $0$--$1$ vectors.
We put: $$\vec{v}\leq_{\vec{v}}\vec{v}'\hbox{ if and only if }t\geq t'\hbox{ and }v_i\geq v'_i\hbox{ for }i=1,\ldots,t'.$$  
\end{defn}
Thus we have e.g. $(1,0,1,1,1)<_{\vec{v}}(1,0,0,1)$ and $(1,0,0,1)>_{\vec{v}}(1,1,1,1)$.
An example of an infinite descending chain is e.g.
$$(1)>_{\vec{v}}(1,1)>_{\vec{v}}(1,1,1)>_{\vec{v}}\ldots.$$
Any finite partially ordered set is representable by vectors with this
ordering: for vectors of a fixed length we have just the reverse ordering of that used in
the (Dushnik-Miller) dimension of partially ordered sets, see e.g.
\cite{Trotter}.

\begin{defn}
We denote by $\TV$ the class of all finite vector-sets.
Let $\vec{V}$ and $\vec{V}'$ be two finite sets of $0$--$1$ vectors.  
We put $\vec{V}\leq_\TV\vec{V}'$ if and only if for every $\vec{v}\in \vec{V}$ there exists a $\vec{v}'\in\vec{V}'$ such that $\vec{v}\leq_{\vec{v}}\vec{v}'$.
\end{defn}

For a word $W$ on the alphabet $\{0,1\}$ we construct a vector $\vec{v}(W)$
of length $2|W|$ such that $2n$-th element of vector $\vec{v}(W)$ is $0$ if and only if the
$n$-th character of $W$ is 0, and the $(2n+1)$-th element of the vector $\vec{v}(W)$ is $1$
if and only if the $n$-th character of $W$ is 0.

It is easy to see that $W\leq_\Words W'$ if and only if
$\vec{v}(W)\leq_{\vec{v}} \vec{v}(W')$.
The embedding $\Embed{\Words}{\TV}:(\Words,\leq_\Words)\to (\TV,\leq_\TV)$ is
constructed as follows:
$$\Embed{\Words}{\TV}= \{\vec{v}(W), W\in A\}.$$

For our example $([4],\leq_P)$ in Figure~\ref{poset} we have embedding:

$$
\begin{array}{lllll}
\Embed{\Words}{\TV}(\Psi([1],\leq_{P_1}))&=&\Embed{\Words}{\TV}(\{0\})&=&\{(0,1)\},\\
\Embed{\Words}{\TV}(\Psi([2],\leq_{P_2}))&=&\Embed{\Words}{\TV}(\{0,10\})&=&\{(0,1),(1,0,0,1)\},\\
\Embed{\Words}{\TV}(\Psi([3],\leq_{P_3}))&=&\Embed{\Words}{\TV}(\{000,100\})&=&\{(0,1,0,1,0,1),(1,0,0,1,0,1)\},\\
\Embed{\Words}{\TV}(\Psi([4],\leq_{P_4}))&=&\Embed{\Words}{\TV}(\{0000\})&=&\{(0,1,0,1,0,1,0,1)\}.\\
\end{array}
$$

\begin{corollary}
The quasi-order $(\TV,\leq_\TV)$ contains a universal partial order.
\end{corollary}

The structure $(\TV,\leq_\TV)$ as compared to $(\Words,\leq_\Words)$ is more complicated to use for further
embeddings: the partial order of vectors is already a complex finite-universal
partial order. The reason why the structure $(\TV,\leq_\TV)$ was discovered first is is that it
allows a remarkably simple on-line embedding that we outline now.

Again we restrict ourselves to the partial orders whose vertex sets are the sets
$[n]$ (for some $n>1$) and we will always embed the vertices in the natural order.
The function $\Psi'$ mapping partial orders $([n],\leq_P)$ to elements of
$(\TV,\leq_\TV)$ is defined as follows:

\begin{defn}
Let $\vec{v}({[n],\leq_{P}})=(v_1$, $v_2$, \ldots, $v_{n})$ where $v_m=1$ if and only if $n\leq_\Poset m$, $m\leq n$, otherwise $v_m=0$. 

Let $$\Psi'([n],\leq_P)=\{\vec{v}([m],\leq_{P_m}): m\in P, m\leq n, m\leq_\Poset n\}.$$
\end{defn}

For our example in Figure~\ref{poset} we get a different (and more compact) embedding:
$$
\begin{array}{llllll}
\vec{v}(1)&=&(1), &\Psi'([1],\leq_{P_1})&=&\{(1)\},\\
\vec{v}(2)&=&(0,1), &\Psi'([2],\leq_{P_2})&=&\{(1),(0,1)\},\\
\vec{v}(3)&=&(1,0,1), &\Psi'([3],\leq_{P_3})&=&\{(1,0,1)\},\\
\vec{v}(4)&=&(1,1,1,1), &\Psi'([4],\leq_{P})&=&\{(1,1,1,1)\}.
\end{array}
$$

\begin{thm}
Fix the partial order $([n],\leq_P)$.
 For every $i,j\in[n]$, $$i\leq_P j \hbox{ if and only if } \Psi'([i],\leq_{P_i})\leq_\TV \Psi'([j],\leq_{P_j})$$ and $$\Psi'([i],\leq_{P_i}) = \Psi'([j],\leq_{P_j}) \hbox{ if and only if } i=j.$$ (Or in the other words, the mapping $\Phi'(i) = \Psi'([i],\leq_{P_i})$ is the embedding of $([n],\leq_P)$ into $(\TV,\leq_\TV)$).
\end{thm}

The proof can  be done via induction analogously as in the second part of the proof of
Theorem \ref{Pembed}. See our paper \cite{HN-trees}.  The main advantage of this embedding is
that the size of the answer is $O(n^2)$ instead of $O(2^n)$. 

\section{Periodic sets}
\label{preiodickesection}
As the last finite presentation we mention the following what we believe to
be very elegant description.
Consider the partial order defined by inclusion on sets of integers. This
partial order is uncountable and contains every countable partial order. We
can however show the perhaps surprising fact that the subset of all periodic subsets (which has a very simple and finite description) is countably universal.
\begin{defn}
$S\subseteq \Z$ is {\em $p$-periodic} if for every $x\in S$ we have also $x+p\in S$ and $x-p\in S$.

For a periodic set $S$ with period $p$ denote by the {\em signature $s(p,S)$} a word over the alphabet $\{0,1\}$ of length $p$ such that $n$-th letter is $1$ if and only if $n\in S$.

By $\Periodic$ we denote the class of all sets $S\subseteq \Z$ such that $S$ is
$2^n$-periodic for some $n$.
\end{defn}

Clearly every periodic set is determined by its signature and thus
$(\Periodic,\subseteq)$ is a finite presentation. We consider the ordering
of periodic sets by inclusion and prove:

\begin{thm}
The partial order $(\Periodic,\subseteq)$ is universal.
\end{thm}

\begin{proof}
We embed $(\Words,\leq_\Words)$ into $(\Periodic,\subseteq)$ as follows:
For $A\in \Words$ denote by $\Embed\Words{\Periodic}(A)$ the set of integers such that $n\in \Embed\Words{\Periodic}(A)$ if and only if there is $W\in A$ and the $|A|$ least significant digits of the binary expansion of $n$ forms a reversed word $W$ (when the binary expansion has fewer than $|W|$ digits, add 0 as needed).

It is easy to see that $\Embed\Words{\Periodic}(A)$ is $2^n$-periodic, where $n$ is the length of
longest word in $W$, and $\Embed\Words{\Periodic}(A)\subseteq \Embed\Words{\Periodic}(A')$ if and only if $A\leq_\Words A'$.
\end{proof}

$(\Periodic,\subseteq)$ is dense, but it fails to have the $3$-extension
property: there is no set strictly smaller than the set with signature $01$ and
greater than both sets with signatures $0100$ and $0010$.

\chapter{Universality of graph homomorphisms}
\label{cestickychapter}

Perhaps the most natural order between finite models is induced by homomorphisms.
The universality of the homomorphism order for the class of all finite
graphs was first shown by \cite{Pultr}.  

Numerous other classes followed (see e. g. \cite{Pultr}) but planar graphs
(and other topologically restricted classes) presented a problem.

The homomorphism order on the class of finite paths was studied in \cite{NZhu}.
It has been proved it is a dense partial order (with the exception of a few
gaps which were characterized; these gaps are formed by all core-path of height
$\leq 4$).  \cite{NZhu} also rises (seemingly too ambitious) question whether
it is a universal partial order. 
  This has
been resolved in \cite{HN-paths,HN-trees} by showing that finite oriented paths
with homomorphism order are universal. In this section we give a new proof of
this result (see also \cite{HN-Posets}).  The proof is simpler and yields a stronger result (see Theorem
\ref{klacky}).

Recall that
an {\em oriented path $P$} of length $n$ is any oriented graph $(V,E)$ where
$V=\{v_0,v_1,\ldots,v_n\}$ and for every $i=1,2,\ldots,n$ either
$(v_{i-1},v_i)\in E$ or $(v_i,v_{i-1})\in E$ (but not both), and there are no
other edges.  Thus an oriented path is any orientation of an undirected path.

Denote by $(\Paths,\leq_\Paths)$ the class of all finite paths ordered by homomorphism
order.

To show the universality of oriented paths, we will construct an embedding of
$(\Periodic,\subseteq)$ to $(\Paths,\leq_\Paths)$.  Recall that the class $\Periodic$ denotes the class of all periodic subsets of $\Z$ (see Section \ref{preiodickesection}).
This is a new feature, which gives
a new, more streamlined and shorter proof of the \cite{HN-paths}.
  The main difference of the proof in \cite{HN-paths,HN-trees} and the one presented here is the use of $(\Periodic,\subseteq)$ as the
base of the representation instead of $(\TV,\leq_\TV)$.  The linear nature of graph
homomorphisms among oriented paths make it very difficult to adapt many-to-one
mapping involved in $\leq_\TV$. The cyclic mappings of $(\Periodic,\subseteq)$ are easier
to use.

Let us introduce terms and notations that are useful when speaking of homomorphisms
between paths. (We follow standard notations as e.g. in \cite{Hell, NZhu}.)

While oriented paths do not make a difference between initial and terminal
vertices, we will always consider paths in a specific order of vertices from
the initial to the terminal vertex.  We denote the initial vertex $v_0$ and the
terminal vertex $v_n$ of $P$ by $in(P)$ and $term(P)$ respectively.  For a path $P$ we will
denote by $\overleftarrow{P}$ the flipped path $P$ with order of vertices
$v_n,v_{n-1},\ldots,v_0$.  For paths $P$ and $P'$ we denote by $PP'$ the path
created by the concatenation of $P$ and $P'$ (i.e. the disjoint union of $P$ and $P'$
with $term(P)$ identified with $in(P')$).

The {\em length} of a path $P$ is the number of edges in $P$.  The {\em
algebraic length} of a path $P$ is the number of forwarding minus the
number of backwarding edges in $P$.  Thus the algebraic length of a
path could be negative.  The {\em level $l_P(v_i)$} of $v_i$ is the algebraic
length of the subpath $(p_0,p_1,\ldots,p_i)$ of $P$.  The {\em distance} between
vertices $p_i$ and $p_j$, $d_P(p_i,p_j)$, is given by $|j-i|$. The {\em
algebraic distance, $a_P(p_i,p_j)$}, is $l_P(v_j)-l_P(v_i)$.

Denote by $\varphi:P\to P'$ a homomorphism from path $P$ to $P'$.  Observe that
we always have $d_P(p_i,p_j)\leq d_{P'}(\varphi(p_i),\varphi(p_j))$ and
$a_P(p_i,p_j)=a_{P'}(\varphi(p_i),\varphi(p_j))$.  We will construct paths in
such a way that  every homomorphism $\varphi$ between path $P$ and $P'$ must
map the initial vertex of $P$ to the initial vertex of $P'$ and thus preserve levels of
vertices (see Lemma \ref{zacatel} below).

\section{Main construction}
  \begin{figure}[t]
    \begin{center}
      \includegraphics[width=12cm]{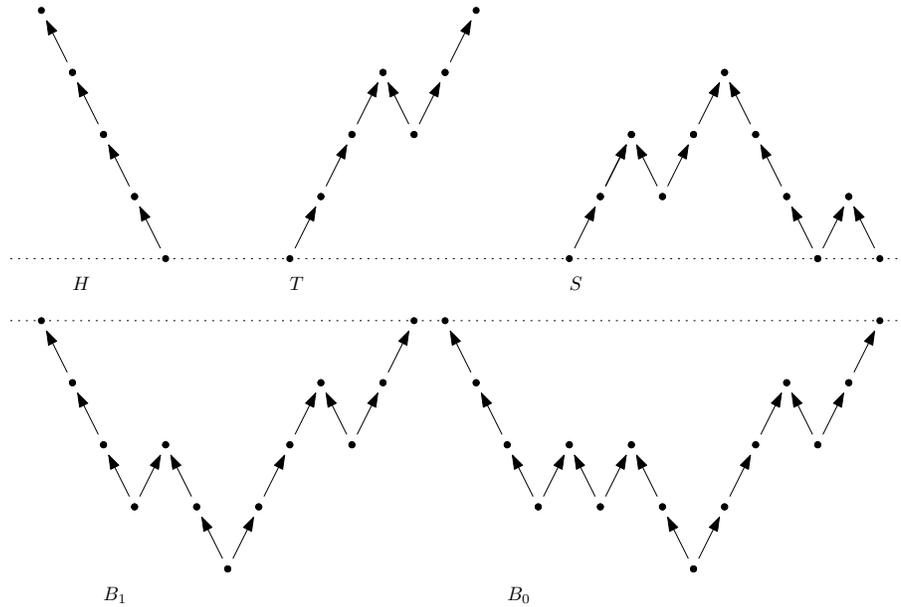}
    \end{center}
    \caption{Building blocks of $p(W)$.}
    \label{cesty}
  \end{figure}
The basic building blocks if our construction are the paths shown in Figure
\ref{cesty} ($H$ stands for {\em head}, $T$ for {\em tail}, $B$ for {\em body} and $S$ for {\em \v
sipka}---arrow in Czech language). Their initial vertices appear on the left,
terminal vertices on the right.  Except for $H$ and $T$ the paths are balanced
(i.e. their algebraic length is $0$). We will construct paths by concatenating
copies of these blocks. $H$ will always be the first path, $T$ always the last.  (The
dotted line in Figure \ref{cesty} and Figure \ref{p01110} determines vertices with level $-4$.)

\begin{defn}
Given a word $W$ on the alphabet $\{0,1\}$ of length $2^n$, we assign path $p(W)$ recursively as follows:

\begin{enumerate}
 \item $p(0)=B_0$.
 \item $p(1)=B_1$.
 \item $p(W)=p(W_1)S\overleftarrow {p(W_2)}$ where $W_1$ and $W_2$ are words of length $2^{n-1}$ such that $W=W_1W_2$. 
\end{enumerate}
Put $\overline{p}(W)=Hp(W)T$.
\end{defn}

\begin{example}
\label{exp1}
For a periodic set $S$, $s(4,S)=0110$, we construct $\overline{p}(s(4,S))$ in the following way:

$$p(0)=B_0,$$
$$p(1)=B_1,$$
$$p(01)=B_0S\overleftarrow{B_1},$$
$$p(10)=B_1S\overleftarrow{B_0},$$
$$p(0110)=B_0S\overleftarrow {B_1}SB_0\overleftarrow S\overleftarrow {B_1},$$
$$\overline{p}(0110)=HB_0S\overleftarrow {B_1}SB_0\overleftarrow S\overleftarrow {B_1}T.$$

  \begin{figure}[t]
    \begin{center}
      \includegraphics[width=13cm]{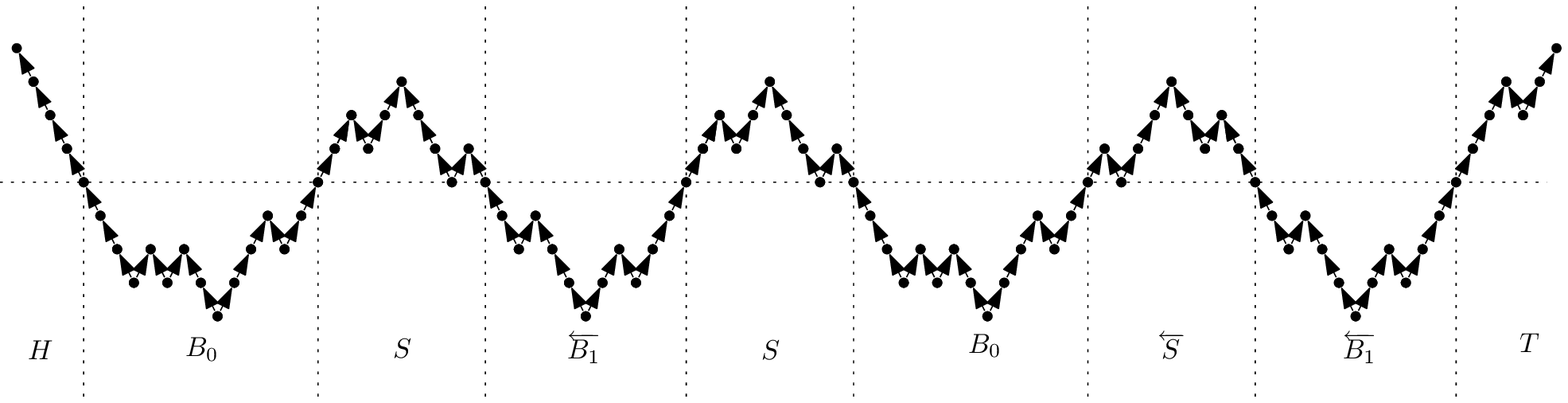}
    \end{center}
    \caption{$\overline{p}(0110)$.}
    \label{p01110}
  \end{figure}
See Figure \ref{p01110}.

\end{example}

The key result of our construction is given by the following:
\begin{prop}
\label{pscarkou}
Fix a periodic set $S$ of period $2^k$ and a periodic set $S'$ of period $2^{k'}$.
There is a homomorphism $$\varphi:\overline{p}(s(2^k,S))\to \overline{p}(s(2^{k'},S'))$$ if and only if $S\subseteq
S'$ and $k'\leq k$.

If a homomorphism $\varphi$ exists, then $\varphi$ maps the initial vertex of $\overline{p}(s(2^k,S))$ to the initial vertex of $\overline{p}(s(2^{k'},S'))$.
If $k'=k$ then $\varphi$ maps the terminal vertex of $\overline{p}(s(2^k,S))$ to the terminal vertex of $\overline{p}(s(2^{k'},S'))$. If $k'<k$
 then $\varphi$ maps the terminal vertex of $\overline{p}(s(2^k,S))$ to the initial vertex of $\overline{p}(s(2^{k'},S'))$.
\end{prop}

Prior to the proof of Proposition $\ref{pscarkou}$ we start with observations about homomorphisms between our special paths.
\begin{lem}
\label{zacatel}
Any homomorphism $\varphi:\overline{p}(W)\to \overline{p}(W')$ must map the initial vertex of $\overline{p}(W)$ to the initial vertex of $\overline{p}(W')$.
\end{lem}
\begin{proof}
$\overline{p}(W)$ starts with the monotone path of 7 edges. The homomorphism $\varphi$ must map this
path to a monotone path in $\overline{p}(W')$. The only such
subpath of $\overline{p}(W')$ is formed by first 8 vertices of $\overline{p}(W')$.

It is easy to see that $\varphi$ cannot flip the path:
If $\varphi$ maps the initial vertex of $\overline{p}(W)$ to the 8th vertex of $\overline{p}(W')$ then $\overline{p}(W)$
has vertices at level $-8$ and because homomorphisms must preserve algebraic
distances, they must map to the vertex of level $1$ in $\overline{p}(W')$ and there
is no such vertex in $\overline{p}(W')$.
\end{proof}
\begin{lem}
\label{phomo}
Fix words $W,W'$ of the same length $2^k$.  Let $\varphi$ be a homomorphism $\varphi:p(W)\to p(W')$.
 Then $\varphi$ maps the initial vertex of $p(W)$ to the initial vertex of $p(W')$
 if and only if $\varphi$ maps the terminal vertex of $p(W)$ to the terminal vertex of $p(W')$.
\end{lem}
\begin{proof}

We proceed by induction on length of $W$:

For $W=i$ and $W'=j$, $i,j\in\{0,1\}$ we have $p(W)=B_i$ and $p(W')=B_j$.
There is no homomorphism $B_1\to B_0$. The unique
homomorphism $B_0\to B_{1}$ has the desired properties. The only homomorphism $B_0\to B_0$ is the isomorphism $B_0\to B_0$.

In the induction step put $W=W_0W_1$ and $W'=W'_0W'_1$ where $W_0$, $W_1$, $W'_0$, $W'_1$ are words of length $2^{k-1}$. We have $p(W)=p(W_0)S\overleftarrow {p(W_1)}$ and $p(W')=p(W'_0)S\overleftarrow {p(W'_1)}$. 

First assume that $\varphi$ maps $in(p(W))$ to $in(p(W'))$.
Then $\varphi$ clearly maps $p(W_0)$ to $p(W'_0)$ and thus by the induction
hypothesis $\varphi$ maps $term(p(W_0))$ to $term(p(W'_0))$. Because the vertices of
$S$ are at different levels than the vertices of the final blocks $B_0$ or $B_1$ of
$p(W_0')$, a copy of $S$ that follows in $p(W)$ after $p(W_0)$ must map to a copy 
of $S$ that follows in $p(W')$ after $p(W'_0)$.
Further $\varphi$ cannot flip $S$ and thus $\varphi$ maps $term(S)$ to $term(S)$.
By same argument $\varphi$ maps $p(W_1)$ to $p(W'_1)$. The initial vertex of $p(W_1)$ is the terminal vertex of $p(W)$ and it must map
to the initial vertex of $p(W'_1)$ and thus also the terminal vertex of $p(W')$.

The second possibility is that $\varphi$ maps  $term(p(W))$ to $term(p(W'))$. This can be handled similarly (starting from the terminal vertex of paths
in the reverse order).
\end{proof}
\begin{lem}
\label{pscarkou2}
Fix periodic sets $S,S'$ of the same period $2^k$.
There is a homomorphism $$\varphi:p(s(2^k,S))\to p(s(2^k,S'))$$ mapping $in(p(s(2^k,S)))$ to $in(p(s(2^k,S')))$ if and only if $S\subseteq S'$. 
\end{lem}
\begin{proof}
If $S\subseteq S'$ then the Lemma follows from the construction of $p(s(2^k,S))$. Every digit 1 of $s(2^k,S)$ has a corresponding copy of $B_1$ in $p(s(2^k,S))$ and every
digit 0 has a corresponding copy of $B_0$ in $p(s(2^k,S))$.  It is easy to build a homomorphism $\varphi$ by concatenating a homomorphism $B_0\to B_1$
and identical maps of $S$, $B_0$ and $B_1$.

In the opposite direction, assume that there is a homomorphism $\varphi$ from $p(s(2^k,S))$ to $p(s(2^k,S'))$. By the assumption and Lemma \ref{phomo}, $\varphi$ must be map
$term(p(s(2^k,S)))$ to $term(p(s(2^k,S')))$. Because $S$ use vertices at different
levels than $B_0$ and $B_1$, all copies of $S$ must be mapped to copies of $S$. Similarly copies of $B_0$ and $B_1$ must be mapped to copies of $B_0$ or $B_1$. If $S\not \subseteq S'$ then there is position $i$ such that $i$-th letter of $s(2^k,S)$ is 1 and $i$-th letter of $s(2^k,S')$ is 0. It follows that the copy of $B_1$ corresponding to this letter would have to map to a copy of $B_0$. This contradicts with the fact that there is no homomorphism $B_1\to B_0$.
\end{proof}

\begin{lem}[folding]
\label{folding}
For a word $W$ of length $2^k$, there is a homomorphism $$\varphi:\overline{p}(WW)\to \overline{p}(W)$$ mapping $in(\overline{p}(WW))$ to $in(\overline{p}(W))$
and $term(\overline{p}(WW))$ to $in(\overline{p}(W))$.
\end{lem}
\begin{proof}
By definition $$\overline{p}(WW)=Hp(W)S\overleftarrow{p(W)}T$$ and $$\overline{p}(W)=Hp(W)T.$$ 
The homomorphism $\varphi$
maps the first copy of $p(W)$ in $\overline{p}(WW)$ to a copy of $p(W)$ in $\overline{p}(W)$, a copy of $S$
is mapped to $T$ such that the terminal vertex of $S$ maps to the initial vertex of $T$
and thus it is possible to map a copy of $\overleftarrow{p(W)}$ in $\overline{p}(WW)$ to the
same copy of $p(W)$ in $\overline{p}(W)$.
\end{proof}

We will use the folding Lemma iteratively. By composition of homomorphisms there is
also homomorphism $p(WWWW)\to p(WW)\to p(W)$. (From the path constructed from $2^k$ copies of $W$ to
$p(W)$.)
\begin{proof}[Proof (of Proposition \ref{pscarkou})]
Assume the existence of a homomorphism $\varphi$ as in Proposition \ref{pscarkou}.
First observe that $k'\leq k$ (if $k<k'$ then there is a copy of $T$ in $\overline{p}(s(2^k,S))$ would have to map
into the middle of $\overline{p}(s(2^{k'},S'))$, but there are no vertices at the level 0 in
$\overline{p}(s(2^{k'},S'))$ except for the initial and terminal vertex).

For $k=k'$ the statement follows directly from Lemma \ref{pscarkou2}.

For $k'<k$ denote by $W''$ the word that consist of $2^{k-k'}$ concatenations of $W'$. Consider a homomorphism $\varphi'$ from $p(W)$ to $p(W'')$ 
mapping $in(p(W)$ to $in(p(W''))$. $W$ and $W''$
have the same length and such a homomorphism exists by Lemma \ref{pscarkou2} if
and only if $S\subseteq S'$. Applying Lemma \ref{folding} there is a homomorphism $\varphi'': p(W'')\to p(W')$.
A homomorphism $\varphi$ can be obtained by composing $\varphi'$ and $\varphi''$.
It is easy to see that any homomorphism $\overline{p}(W)\to \overline{p}(W')$ must follow the same scheme of
``folding'' the longer path $\overline{p}(W)$ into $\overline{p}(W')$ and thus there is a homomorphism
$\varphi$ if and only if $S\subseteq S'$. We omit the details.
\end{proof}

For a periodic set $S$ denote by $S^{(i)}$ the inclusion maximal
periodic subset of $S$ with period $i$. (For example for $s(4,S)=0111$ we have $s(2,S^{(2)})=01$.)
\begin{defn}
For $S\in \Periodic$ let $i$ be the minimal integer such that $S$ has period $2^i$.
Let $\Embed{\Periodic}{\Paths}(S)$  be the concatenation of the paths $$H,$$
$$\overline{p}(s(1,S^{(1)}))\overleftarrow{\overline{p}(s(1,S^{(1)}))},$$
$$\overline{p}(s(2,S^{(2)}))\overleftarrow{\overline{p}(s(2,S^{(2)}))},$$
$$\overline{p}(s(4,S^{(4)}))\overleftarrow{\overline{p}(s(4,S^{(4)}))},$$
\centerline{\ldots,}
$$\overline{p}(s(2^{i-1},S^{(2^{i-1})}))\overleftarrow{\overline{p}(s(2^{i-1},S^{(2^{i-1})}))},$$
$$\overline{p}(s(2^i,S))\overleftarrow {\overline{p}(s(2^i,S))}.$$
\end{defn}

\begin{thm}
$\Embed{\Periodic}{\Paths}(v)$ is an embedding of $(\Periodic,\subseteq)$ to $(\Paths,\leq_\Paths)$.
\end{thm}
\begin{proof}
Fix $S$ and $S'$ in $\Periodic$ of periods $2^i$ and $2^{i'}$ respectively.

Assume that $S\subseteq S', i>i'$.  Then the homomorphism $\varphi:\Embed{\Periodic}{\Paths}(S)\to\Embed{\Periodic}{\Paths}(S')$ can be constructed
via the concatenation of homomorphisms:
$$H\to H,$$
$$\overline{p}(s(1,S^{(1)}))\overleftarrow{\overline{p}(s(1,S^{(1)}))}\to \overline{p}(s(1,S'^{(1)}))\overleftarrow{\overline{p}(s(1,S'^{(1)}))},$$
$$\overline{p}(s(1,S^{(2)}))\overleftarrow{\overline{p}(s(1,S^{(2)}))}\to \overline{p}(s(2,S'^{(2)}))\overleftarrow{\overline{p}(s(2,S'^{(2)}))},$$
\centerline{\ldots,}
$$\overline{p}(s(2^{i'-1},S^{(2^{i'-1})}))\overleftarrow{\overline{p}(s(2^{i'-1},S^{(2^{i'-1})}))}\to \overline{p}(s(2^{i'-1},S'^{(2^{i'-1})}))\overleftarrow{\overline{p}(s(2^{i'-1},S'^{(2^{i'-1})}))},$$
$$\overline{p}(s(2^{i'},S^{(2^{i'})}))\overleftarrow{\overline{p}(s(2^{i'},S^{(2^{i'})}))}\to \overline{p}(s(2^{i'},S')),$$
$$\overline{p}(s(2^{i'+1},S^{(2^{i'+1})}))\overleftarrow{\overline{p}(s(2^{i'+1},S^{(2^{i'+1})}))}\to \overline{p}(s(2^{i'},S')),$$
\centerline{\ldots,}
$$\overline{p}(s(2^{i},S))\overleftarrow{\overline{p}(s(2^{i'},S))}\to \overline{p}(s(2^{i'},S')).$$
Individual homomorphisms exists by Proposition \ref{pscarkou}.  For $i\leq i'$ the
construction is even easier.

In the opposite direction assume that there is a homomorphism
$\varphi:\Embed{\Periodic}{\Paths}(S)\to\Embed{\Periodic}{\Paths}(S')$.
$\Embed{\Periodic}{\Paths}(S)$ starts by two concatenations of $H$ and thus 
a long monotone path and using a same argument as in Lemma \ref{zacatel}, $\varphi$
must map the initial vertex of $\Embed{\Periodic}{\Paths}(S)$ to the initial vertex of
$\Embed{\Periodic}{\Paths}(S')$.  It follows that $\varphi$ preserves levels of vertices.
It follows that for every $k=1,2,4,\ldots, 2^i$, $\varphi$ must map $\overline{p}(s(k,S^{(k)})$ to $\overline{p}(s(k',S'^{(k')}))$  for some $k'\leq k, k'=1,2,4,\ldots, 2^{i'}$.  By application of
Proposition \ref{pscarkou} it follows that $S^{(k)}\subseteq S'^{(k')}$. In particular $S\subseteq S'^{(k')}$. This holds only if $S\subseteq S'$.
\end{proof}
\begin{thm}[\cite{HN-paths}]
\label{pathuniv}
The quasi-order $(\Paths,\leq_\Paths)$ contains universal partial order.
\end{thm}

In fact our new proof of Corollary \ref{pathuniv} gives the following
strengthening for rooted homomorphisms of paths. A {\em plank} $(P,r)$ is an
oriented path rooted at the initial vertex $r=in(P)$.
Given planks $(P,r)$ and $(P',r')$, a homomorphism $\varphi:(P,r)\to (P',r')$ is a homomorphism
$P\to P'$ such that $\varphi(r)=r'$.

\begin{thm}
\label{klacky}
The quasi-order formed by all planks ordered by the existence of homomorphisms contains a universal partial order.
\end{thm}

\section{Other classes}
In this section we outline techniques of proving the universality of
homomorphism order on a given class $\K$ of graphs (or relational structures
in general) by embedding the class of oriented paths.  We prove following
two results that presented an open problem for several years.
\begin{thm}
\label{degree}
Denote by $(\Delta_k,\leq_{\Delta_k})$ the class of all finite graphs  with the
maximal degree $\leq k$ ordered by the existence of a homomorphism.  The
quasi-order $(\Delta_k,\leq_{\Delta_k})$ contains a universal partial order if
and only if $k\geq 3$.
\end{thm}

\begin{thm}
  \label{thero}
  Denote by $(\K,\leq_\K)$ the class of all cubic planar graphs ordered by the existence of homomorphism. 
  The quasi-order $(\K,\leq_\K)$ contains a universal partial order.
\end{thm}
We use the indicator technique (``arrow construction'') which allows us to
replace arcs of a graph by copies of a gadget (``indicator'') in such a way
that the (global) homomorphism properties are preserved, see \cite{Pultr, NT}.
More precisely this can be done as follows:

Any graph $I$ with two distinguished vertices $a$,$b$ is called an {\em indicator}.
(We use the indicator defined by Figure \ref{strnulak1}.)
  \begin{figure}[strnulak1]
    \begin{center}
      {\def\IPEfile{strnulak1.ipe}\begingroup
  \catcode`\%=9\catcode`\!=0\catcode`\-=11\input{\IPEfile}}
    \end{center}
    \caption{$(I,a,b)$.}
    \label{strnulak1}
  \end{figure}
Given a graph $G=(V,E)$ we denote by $G*(I,a,b)$ the following graph $(W,F)$:

$$W=(E\times V(I))/\sim.$$

Thus the vertices of $(V,E)$ are equivalence classes of the equivalence $\sim$.
For a pair $(e,x)\in E\times V(I)$ its equivalence class will be denoted by
$[e,x]$.

The equivalence $\sim$ is generated by the following pairs:

$$((x,y),a)\sim((x,y'),a),$$
$$((x,y),b)\sim((x',y),b),$$
$$((x,y),b)\sim((y,z),a).$$

We put $\{[e,x],[e',x']\}\in F$ if and only if $e=e'$ and $\{x,x'\}\in E(I)$.

Indicator construction is schematically shown in Figure \ref{indikatory}.
See also Section \ref{bounded} for the indicator construction on relational
structures.
  \begin{figure}[tp]
    \begin{center}
      {\def\IPEfile{indikatory.ipe}\begingroup
  \catcode`\%=9\catcode`\!=0\catcode`\-=11\input{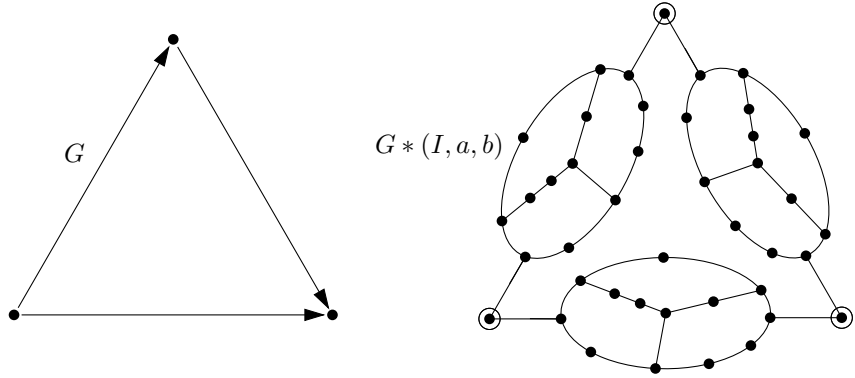}}
    \end{center}
    \caption{Construction of $G*(I,a,b)$.}
    \label{indikatory}
  \end{figure}

We have the following properties:

\begin{claim}~
  \begin{enumerate}
    \item $P*(I,a,b)$ is a planar graph with all its degrees $\leq 3$ for every path $P$.
    \item If $f:P\to P'$ is a path homomorphism then the mapping $\phi(f)$ defined by $$\phi(f)[(u,v),x]=[(f(u),f(v)),x]$$ is a homomorphism $\phi(P)\to\phi(P')$.
    \item If $g:\phi(P)\to\phi(P')$ then there exists $f:P\to P'$.
  \end{enumerate}
\end{claim}
\begin{proof}
Only the last claim needs explanation.
Put $I'=I-\{a,b\}$ (thus $I'$ is the main block of $I$).  Observe that the
only cycles in the graph $P*(I,a,b)$ of length $\leq 7$ belong to the set
$\{[a,z]:z\in V(I')\}$ from an edge $e\in P$.  In fact all non-trivial blocks
of $P*(I,a,b)$ are isomorphic to $I'$.  It is well known that $I'$ is rigid
(see e.g. \cite{NT}).  This in turn means that for any homomorphism
$G:P*(I,a,b)\to P'*(I,a,b)$ there exists a mapping $f:V(P)\to V(P')$ such that
for every edge $e=(x,y)\in E$ and $z\in V(I')$ holds $g([e,z])=[(f(x),
f(y)),z]$.  This $f$ is a desired homomorphism $P\to P'$.
(Note that this correspondence of $g$ and $f$ is not functorial; the graph $I$ fails to be rigid.)
\end{proof}

Put $\phi(P)=P*(I,a,b)$.  We proved $P\to P'$ if and only if $\phi(P)\to \phi(P')$. Note
that $\phi(P)$ is planar and that all degrees $\leq 3$.  It is a graph theory
routine to extend $\phi(P)$ to planar cubic graphs. This implies Theorem \ref{thero}.


\subsection{Series-parallel graphs}

We can use the indicator construction to obtain the following
\begin{thm}
  \label{series}
  Denote by $(\mathcal S_l,\leq_{\mathcal S_l})$  the class of all series-parallel graphs of girth $>l$.
  For every $l>0$ the quasi-order $(\mathcal S_l,\leq_{\mathcal S_l})$ contains a universal partial order.
\end{thm}

Fix $l\geq 2$. Theorem \ref{series} is proved similarly as \ref{thero} by means of the
indicator $I_l$ defined in Figure \ref{strnulak2}.
  \begin{figure}[strnulak2]
    \begin{center}
      {\def\IPEfile{strnulak2.ipe}\begingroup
  \catcode`\%=9\catcode`\!=0\catcode`\-=11\input{\IPEfile}}
    \end{center}
    \caption{$(I_1,a,b)$.}
    \label{strnulak2}
  \end{figure}
The vertices of $I_l$ are $a,b,c_1,c_2,c_3$ together with
$$v_{12}(1),\ldots,v_{12}(l-1), v'_{12}(1),\ldots,v'_{12}(l),$$
$$v_{13}(1),\ldots,v_{13}(l-1), v'_{13}(1),\ldots,v'_{13}(l),$$
$$v_{23}(1),\ldots,v_{23}(l-1), v'_{23}(1),\ldots,v'_{23}(l).$$

The edges of $I_l$ form pairs $\{a,v'_{12}(2)\}, \{v'_{23}(l-2),b\}$ and edges of paths joining vertices $c_1$, $c_2$, $c_3$: 
$$\{c_1, v_{12}\}, \{c_1, v'_{12}\},\{c_1, v_{13}\},$$
$$\{v_{ij}(k),v_{ij}(k+1)\}\hbox{ for }1\leq i<j\leq 3, k=1,\ldots,l-2,$$
$$\{v'_{ij}(k),v'_{ij}(k+1)\}\hbox{ for }1\leq i<j\leq 3,  k=1,\ldots,l-2,$$
$$\{c_1, v_{1i}(1)\},\{c_1,v'_{1i}(1)\}\hbox{ for } i=2,3,$$
$$\{c_2, v_{12}(l-1)\},\{c_2,v'_{12}(l)\},$$
$$\{c_2, v_{23}(1)\},\{c_2,v'_{23}(1)\},$$
$$\{c_1, v_{i3}(l-1)\},\{c_3,v'_{i3}(l)\}\hbox{ for } i=2,3.$$

The graph $I_l$ has girth $>2l+1$.  Put $I'=I-\{a,b\}$. $I'$ is not rigid but
it is a core graph and it has no automorphism witch maps $v'_{12}(2)$ to
$v'_{23}(l-2)$.  It follows that we may argue similarly as in the proof of Theorem
\ref{thero}.  We omit the details.

\section{Related results}

By similar techniques as presented in this chapter Lehtonen \cite{lehtonen1} shows universally of labeled partial orders ordered by homomorphisms.

Lehtonen  and Ne\v set\v ril \cite{lehtonen2} consider also the partial order defined on boolean
functions in the following way.  Each clone $\mathcal C$ on a fixed base set $A$
determines a quasiorder on the set of all operations on $A$ by the following
rule: $f$ is a $\mathcal C$-minor of $g$ if $f$ can be obtained by substituting
operations from $\mathcal C$ for the variables of $g$.  Using embedding homomorphism order
on hypergraphs, it can be shown that a clone $\mathcal C$ on $\{0,1\}$ has the property
that the corresponding $\mathcal C$ minor partial order is universal if and only if $\mathcal C$
is one of the countably many clones of clique functions or the clone of
self-dual monotone functions (using the classification of Post classes).

It seems that in most cases the homomorphism order of classes of relational
structures is either universal or fails to be universal for very simple reasons
(such as the absence of infinite chains or anti-chains).
Ne\v set\v ril and Nigussie \cite{NJared} look for minimal minor closed classes of
graphs that are dense and universal.  They show that $(\Paths,\leq_\Paths)$ is
a unique minimal class of oriented graphs which is both universal and dense.
Moreover, they show a dichotomy result for any minor closed class $\K$ of
directed trees. $\K$ is either universal or it is well-quasi-ordered.  Situation
seems more difficult for the case of undirected graphs, where such minimal
classes are not known and only partial result on series-parallel graphs was
obtained.

\chapter{Universal structures for $\Forb(\F)$}
\label{Forbchapter}

The main purpose of this chapter is to give a new proof of the existence of an (embedding-) universal
structure for the class $\Forb(\F)$, where $\F$ consists of connected finite structures of finite type
(Corollary~\ref{1thm}). Unlike \cite{CherlinShelahShi} we give a combinatorial
proof based on the amalgamation method. 

Explicit construction allows us to state the result in a stronger form (Theorem
\ref{mainthm}) for countable families $\F$.  Explicit construction also makes
it possible to describe the universal structures via forbidden embeddings
(Theorem \ref{reprezent}) and establish a number of their properties.

The techniques used in the finite presentation of the rational Urysohn space (Chapter~\ref{homprostorchapter}) can be extended to the
finite presentation of universal structures constructed here.  In
the general case the resulting construction is however too complicated to serve its
purpose as a simple and informative description of the universal structure. We show (in Section \ref{urysohnsection}) the relation to
homomorphism dualities and Urysohn spaces for special families $\F$ in order to outline how finite presentation can be constructed.

%
First let us recall the concept of lifts and shadows in a more detailed form.
The class $\Rel(\Delta), \Delta=(\delta_i: i\in I)$, $I$ finite, is fixed throughout this chapter. Unless otherwise stated all structures $\relsys{A}, \relsys{B},\ldots$ belong to $\Rel(\Delta)$.
Now let $\Delta'=(\delta'_i:i\in I')$ be a type containing type $\Delta$. (By this we mean $I\subseteq I'$ and $\delta'_i=\delta_i$ for $i\in I$.)
Then every structure $\relsys{X}\in \Rel(\Delta')$ may be viewed as structure $\relsys{A}=(A,(\rel{A}{i}: i\in I))\in \Rel(\Delta)$ together with some additional relations $\rel{X}{i}$ for $i\in I'\setminus I$. To make this more explicit these additional relations will be denoted by $\ext{X}{i}, i\in I'\setminus I$. Thus a structure $\relsys{X}\in \Rel(\Delta')$ will be written as
$$\relsys{X}=(A,(\rel{A}{i}:i\in I),(\ext{X}{i}:i\in I'\setminus I))$$

and, by abuse of notation, more briefly as
$$\relsys{X}=(\relsys{A},\ext{X}{1},\ext{X}{2},\ldots, \ext{X}{N}).$$

 We call $\relsys{X}$ a {\em lift} of $\relsys{A}$ and $\relsys{A}$ is called the {\em shadow} (or {\em projection}) of $\relsys{X}$. In this sense the class $\Rel(\Delta')$ is the class of all lifts of $\Rel(\Delta)$.
Conversely, $\Rel(\Delta)$ is the class of all shadows of $\Rel(\Delta')$. In this chapter we shall always consider types of shadows to be finite, although we allow countable types for lifts (so $I$ is finite and $I'$ countable).
 Note that a lift is also in the model-theoretic setting called an {\em expansion} and a shadow a {\em reduct}.
(Our terminology is motivated by a computer science context, see \cite{KunN}.)
We shall use letters $\relsys{A}, \relsys{B}, \relsys{C}, \ldots$ for shadows (in $\Rel(\Delta)$) and letters $\relsys{X}, \relsys{Y}, \relsys{Z}$ for lifts (in $\Rel(\Delta'))$.

For a lift $\relsys X=(\relsys{A}, \ext{X}{1},\ldots,\ext{X}{N})$, we denote by $\psi(\relsys{X})$ the relational
structure $\relsys{A}$, i.e., the shadow of $\relsys{X}$. ($\psi$
is called a {\em forgetful functor}.) Similarly, for a class $\K'$
of lifted structures we denote by $\psi(\K')$ the class of all shadows
of structures in $\K'$.


For a structure $\relsys{A}=(A,(\rel{A}{i}:i\in I))$ the {\em Gaifman graph}
(in combinatorics often called {\em 2-section}) is the graph $G$ with vertices
$A$ and all those edges which are a subset of a tuple of a relation of
$\relsys{A}$: $G=(V,E)$, where $\{x,y\}\in E$ if and only if $x\neq y$ and there exists a
tuple $\vec{v}\in \rel{A}{i},i\in I$, such that $x,y\in \vec{v}$.

A {\em cut} in $\relsys{A}$ is a subset $C$ of $A$ such that the Gaifman graph $G_\relsys{A}$ is disconnected by removing the set $C$ (i.e. if $C$ is graph-theoretic cut of $G_{\relsys A}$). By a {\em minimal cut} we always mean an inclusion-minimal cut.

If $C$ is a set of vertices then $\overrightarrow{C}$ will denote a tuple (of length $|C|$) from all elements of $R$. Alternatively, $\overrightarrow{R}$ is an arbitrary linear order of $R$.


\section{Classes omitting countable families of structures}
\label{hlavnisekce}

Let $\F$ be a fixed countable set of finite relational structures of finite type
$\Delta$. For the construction of a universal structure of $\Forb(\F)$ we use special
lifts, called $\F$-lifts.  The definition of $\F$-lift is easy and
resembles decomposition techniques standard in graph theory and thus we adopt a
similar terminology. The following is the basic notion:

\begin{defn}
For a relational structure $\relsys{A}$ and minimal cut $R$ in $\relsys{A}$,
a {\em piece} of a relational structure $\relsys A$ is a pair
$\Piece=(\relsys{P},\overrightarrow{R})$. Here $\relsys{P}$ is the structure
induced on $\relsys{A}$ by the union of $R$ and vertices of some connected
component of $\relsys{A}\setminus R$. The tuple $\overrightarrow{R}$ consists of the vertices of the cut $R$ in a (fixed) linear order.
\end{defn}

Note that from inclusion-minimality of the cut $R$ it follows that the pieces of a connected structure are always connected structures.

All pieces are thought of as rooted structures: a piece $\Piece$ is a structure $\relsys{P}$ rooted at $\overrightarrow{R}$. Accordingly, we say
 that pieces $\Piece_1=(\relsys{P}_1,\overrightarrow{R}_1)$ and
$\Piece_2=(\relsys{P}_2,\overrightarrow{R}_2)$ are {\em isomorphic} if there is a function $\varphi:P_1\to P_2$ that is a isomorphism of structures $\relsys{P}_1$ and $\relsys{P}_2$ and $\varphi$ restricted to $\overrightarrow{R}_1$ is a monotone bijection between $\overrightarrow{R}_1$ and $\overrightarrow{R}_2$ (we denote this
$\varphi(\overrightarrow{R}_1)=\overrightarrow{R}_2$).

\begin{figure}
\label{Petersoni}
\centerline{\includegraphics{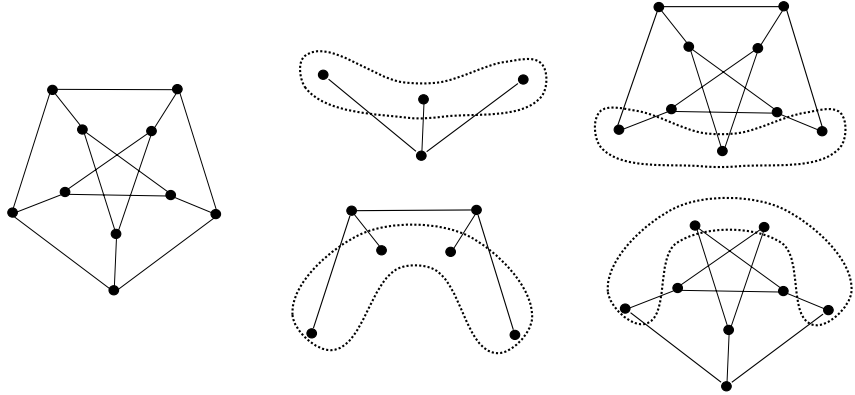}}
\caption{Pieces of the Petersen graph up to isomorphism (and permutations of roots). }
\end{figure}
Observe that for relational trees, pieces are equivalent to rooted  branches. Pieces of the Petersen graph are shown in Figure \ref{Petersoni}.

First let us prove a simple observation about pieces. We show that in most cases a ``subpiece'' of a piece is a piece.
\begin{lem}
Let $\Piece_1=(\relsys{P}_1,\overrightarrow{R}_1)$ be a piece of structure $\relsys A$ and $\Piece_2=(\relsys{P}_2,\overrightarrow{R}_2)$ a piece of $\relsys{P}_1$.  If $R_1\cap P_2\subseteq R_2$, then $\Piece_2$ is also a  piece of $\relsys A$.
\label{kousky}
\end{lem}
\begin{proof}
Denote by $\relsys{C}_1$ the connected component of $\relsys{A}\setminus R_1$ that produces $\Piece_1$.  Denote by $\relsys{C}_2$ the component of $\relsys{P}_1\setminus R_2$ that produces $\Piece_2$. 
As $R_1\cap P_2\subseteq R_2$ one can check that then $\relsys{C}_2$ is contained in $\relsys{C}_1$ and every vertex of $\relsys{A}$ connected by a tuple to any vertex of $\relsys{C}_2$ is
contained in $\relsys{P}_1$. Thus $\relsys{C}_2$ is also a connected component of
$\relsys{A}$, created after removing vertices of $R_2$.
\end{proof}

Fix an index set $I'$ and let $\Piece_i, i\in I'$, be all pieces of all relational
structures $\relsys{F}\in \F$. Notice that there are only countably many pieces.

The relational structure $\relsys X=(\relsys{A}, (\ext{X}{i}:i\in I'))$ is
called the {\em $\F$-lift} of the relational structure $\relsys A$ when the arities of
relations $\ext{X}{i},i\in I'$, correspond to
$|\overrightarrow{R}_i|$.

For a relational structure $\relsys A$ we define the {\em canonical lift $\relsys{X}=L({\relsys A})$}
by putting $(v_1,v_2,\ldots,v_l)\in \ext{X}{i}$ if and only if there is homomorphism
$\varphi$ from $\relsys{P}_i$ to $A$ such that $\varphi(\overrightarrow
{R}_i)=(v_1,v_2,\ldots,v_l)$.

\begin{thm}
\label{mainthm}
Let $\F$ be a countable set of finite connected relational structures.  Denote
by $\Lifts$ the class of all induced substructures (sublifts) of lifts $L(\relsys{A})$,
$\relsys{A}\in\Forb(\F)$.   Denote by $\Lifts_f$ the class of all finite structures
in $\Lifts$.  $\Lifts_f$ is an amalgamation class (Definition \ref{amalgamationclassdef}).
There is a generic structure $\relsys{U}$ in $\Lifts$ and its shadow $\psi(\relsys{U})$ is a universal structure for the class $\Forb(\F)$.
\label{mainth}
\end{thm}

\begin{figure}
\label{constructionh}
\centerline{\includegraphics{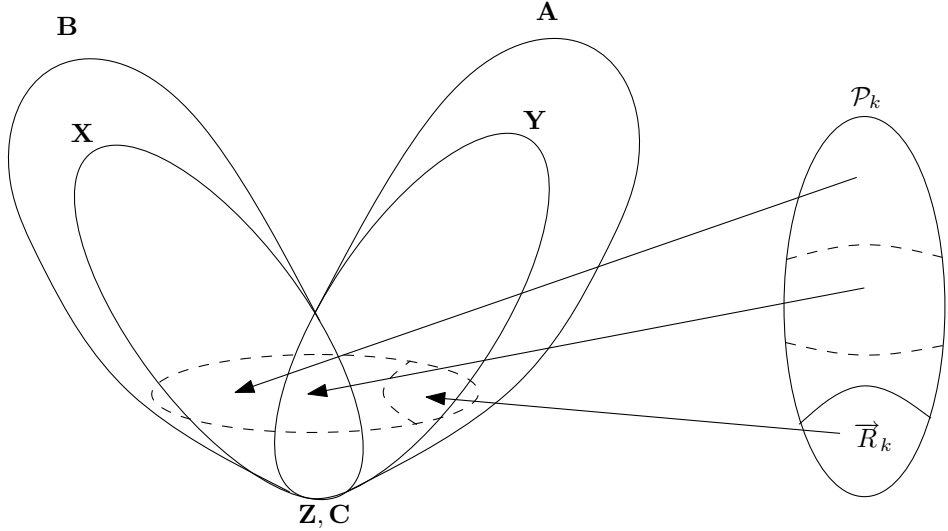}}
\caption{Construction of an amalgamation.}
\end{figure}

For $\relsys{X}\in \Lifts$ we denote by $W(\relsys{X})$ one of the structures $\relsys{A}\in\Forb(\F)$ such that the structure $\relsys{X}$ is induced on $X$ by $L(\relsys{A})$. $W(\relsys{X})$ is called a {\em witness} of the fact that $\relsys{X}$ belongs to $\Lifts$.

\begin{proof}
By definition the class $\Lifts$ (and thus also $\Lifts_f$) is hereditary, isomorphism-closed, and has the joint
embedding property. $\Lifts_f$ contains only countably many mutually non-isomorphic structures, because there are only countably
many mutually non-isomorphic structures in $\Age(\Forb(\F))$ (type $\Delta$ is finite) and thus also
countably many mutually non-isomorphic lifts. To show that $\Lifts_f$ is an amalgamation  class it remains to
verify that $\Lifts_f$ has the amalgamation property. The rest of theorem follows from Theorem \ref{fraissethm}
and the fact that $\Lifts$ is the class of all lifts younger than the ultrahomogeneous structure (lift) $\relsys{U}=\lim \Lifts_f$ (\Fraisse{} limit of $\Lifts_f$).

Consider $\relsys{X},\relsys{Y},\relsys{Z}\in \Lifts_f$. Assume that structure
$\relsys{Z}$ is a substructure induced by both $\relsys{X}$ and
$\relsys{Y}$ on $Z$ and without loss of generality assume that $X\cup Y=Z$.

Put
$$\relsys{A}=W(\relsys{X}),$$
$$\relsys{B}=W(\relsys{Y}),$$
$$\relsys{C}=\psi(\relsys{Z}).$$

Because $\Lifts$ is closed under isomorphism, we can still assume that
$\relsys{A}$ and $\relsys{B}$ are vertex-disjoint with the exception of vertices of
$\relsys{C}$.

Let $\relsys{D}$ be the free amalgamation of $\relsys{A}$ and $\relsys{B}$ over vertices of $\relsys{C}$: the vertices of $\relsys{D}$ are $A\cup B$ and there is $\vec{v}\in \rel{D}{i}$ if and only if $\vec{v}\in \rel{A}{i}$ or $\vec{v}\in \rel{B}{i}$.

We claim that the structure $$\relsys{V}=L(\relsys{D})$$ is a (not necessarily free) amalgamation of $L(\relsys{A})$ and $L(\relsys{B})$ over $\relsys{Z}$
and thus also an amalgamation of $\relsys{X}, \relsys{Y}$ over $\relsys{Z}$. The situation is depicted in Figure \ref{constructionh}.

First we show that the substructure induced by $\relsys{V}$ on $A$ is $L(\relsys{A})$ and that the substructure induced by $\relsys{V}$ on $B$ is $L(\relsys{B})$. In the other words, no new tuples to $L(\relsys{A})$ or $L(\relsys{B})$ (and thus none to $\relsys{X}$ or $\relsys{Y}$ either) have been introduced.

Assume to the contrary that there is a new tuple $(v_1,\ldots,v_t)\in \ext{V}{k}$ and among all tuples and possible choices of $k$ choose one with the minimal number of vertices in the corresponding piece $\Piece_k$. By symmetry we can assume that $v_i\in A, i=1,\ldots,t$. Explicitly, we assume
that there is a homomorphism $\varphi$ from $\relsys{P}_k$ to $\relsys{D}$ such that
$\varphi(\overrightarrow{R_k})=(v_1,v_2,\ldots, v_t)\notin \extl{L(\relsys{A})}{k}$.

The set of vertices of $\relsys{P}_k$ mapped to $L(\relsys{A})$, $\varphi^{-1}(A)$, is nonempty, because it contains all vertices of $\overrightarrow{R_k}$.  $\varphi^{-1}(B)$ is nonempty because there is no homomorphism $\varphi'$ from $\relsys{P}_k$ to $\relsys{A}$ such that $\varphi'(\overrightarrow{R_k})=(v_1,v_2,\ldots, v_t)$ (otherwise we would have $(v_1,v_2,\ldots, v_t)\in \extl{L(\relsys{A})}{k})$.

Because there are no tuples spanning both vertices $A\setminus C$ and vertices $B\setminus C$ in $\relsys D$ and because pieces are connected we also have
$\varphi^{-1}(C)$ nonempty. Additionally, the vertices of $\varphi^{-1}(C)$ form a cut of $\relsys{P}_k$.

Denote by $\relsys K_1, \relsys K_2,\ldots, \relsys K_l$ all connected components of the substructure induced on $P_k\setminus \varphi^{-1}(A)$ by  $\relsys P_k$.  For each component $\relsys K_i$, $1\leq i\leq l$, there is a vertex cut $K'_i$ of $\relsys P_k$ formed by all vertices of $\varphi^{-1}(A)$ connected to $K_i$.
This cut is always contained in $\varphi^{-1}(C)$.

Because $\Piece_k$ is piece of some $\relsys F\in \F$ and because $(\relsys K_i, \overrightarrow {K'_i})$ are pieces of $\relsys P_k$, by Lemma \ref{kousky} they are also pieces of $\relsys F$. We denote by $\Piece_{k_1}, \Piece_{k_2},\ldots, \Piece_{k_l}$ the pieces isomorphic to the pieces $(\relsys K_1, \overrightarrow
{K'_1}), (\relsys K_2, \overrightarrow {K'_2}), \ldots, (\relsys K_l,
\overrightarrow {K'_l})$ via isomorphisms $\varphi_1, \varphi_2,\ldots,
\varphi_l$ respectively.

Now we use minimality of the piece $\Piece_k$. All the pieces $\Piece_{k_i},
i=1,\ldots, l$, have smaller size than $\Piece_k$ (as $\varphi^{-1}(C)$ is a cut
of $\Piece_k)$. Thus we have that tuple $\varphi(K_i)$ of $L(\relsys{D})$ is also
 a tuple of $L(\relsys{A})$.  Thus there exists a homomorphism $\varphi'_i$
from $\relsys K_i$ to $\relsys D$ such that
$\varphi'_i(\overrightarrow {K'_i}) = \varphi (\overrightarrow {K'_i})$ for every
$i=1,2,\ldots, l$.

In this situation we define $\varphi'(x):P_k\to A$ as follows:
\begin{enumerate}
 \item $\varphi'(x)=\varphi'_i(x)$ when $x\in K_i$ for some $i=1,2,\ldots, l$.
 \item $\varphi'(x)=\varphi(x)$ otherwise.
\end{enumerate}

It is easy to see that $\varphi'(x)$ is a homomorphism from $\relsys P_k$ to $L(\relsys{A})$. This is a contradiction.

It remains to verify that $\relsys{D}\in \Forb(\F)$. We proceed analogously. Assume that $\varphi$ is a homomorphism of some $\relsys{F}\in \F$ to $\relsys{D}$. Because $\relsys{A},\relsys{B}\in \Forb(\F)$, $\varphi$ must use vertices of $\relsys{C}$ and $\varphi^{-1}(C)$ forms a cut of $\relsys{F}$.   Denote by $E$ a minimal cut contained in $\varphi^{-1}(C)$. $\varphi(E)$ must contain tuples corresponding to all pieces of $\relsys{F}$ having $E$ as roots in $\relsys{Z}$.  This is a contradiction with $\relsys{Z}\in \Lifts$.

\end{proof}


\section {Forbidden lifts ($\Forbi(\F')$ classes)}

In Theorem \ref{mainthm} we found an amalgamation class $\Lifts_f\in
\Rel(\Delta')$ of lifted objects such that the shadow of the \Fraisse{} limit
of $\Lifts_f$ is a universal object of $\Forb(\F)$. In this section we focus on
finite families $\F$ and further refine this result by giving an explicit
description of the amalgamation class $\Lifts$ in terms of forbidden
substructures. We prove that $\Lifts$ is equivalent to a class $\Forbi(\F')$
for an explicitly defined class of lifts $\F'$ (derived from the class $\F$).
This however holds only for lifts with finite types.  For infinite families
$\F$ (and thus lifts with infinite type) this is not possible: there
are even uncountably many relational systems with a single vertex $v$: every
relation may or may not contain the tuple $(v,v,\ldots, v)$. Only countably many of
them can be forbidden in $\Forbi(\F')$ for $\F'$ countable and thus the class $\Age(\Forbi(\F'))$
must contain uncountably many mutually non-isomorphic structures.

First we show a more explicit construction of a witness.

\begin{defn} 
\label{enumerace}
For a piece $\Piece_i$ such that $\overrightarrow{R_i}$ is an $n$-tuple and for an $n$-tuple $\vec{x}$ of vertices of $\relsys{A}$ we denote by $A+_{\vec{x}}\Piece_i$ the relational structure created as a disjoint union of $\relsys{A}$ and $\relsys{P}_i$ identifying vertices of $\overrightarrow{R}_i$ along $\vec{x}$ (i.e. $A+\Piece_i$ is a free amalgamation over $\overrightarrow{x}$).

We put $\relsys{X}+_{\ext{X}{i}}\Piece_i = \relsys{\psi(X)}+_{\vec x_1}\Piece_i+_{\vec x_2}\Piece_i+_{\vec x_3}\ldots+_{\vec x_k}\Piece_i$, where $\{\vec x_1,\vec x_2,\ldots, \vec x_k\}=\ext{X}{i}$.
Finally we put $UW(\relsys X)=\psi(\relsys{X})+_{\ext{X}{1}}\Piece_1+_{\ext{X}{2}}\Piece_2+_{\ext{X}{3}}\ldots+_{\ext{X}{N}}\Piece_N$.
We shall call $UW(\relsys{X})$ the {\em universal witness} of $\relsys X$.
\end{defn}
Now we develop an alternative and more explicit description of the class $\Lifts$ (introduced in Section \ref{hlavnisekce}). We preserve all the notation introduced there.
\begin{lem}
\label{UWlem}
Lift $\relsys X$ belongs to $\Lifts$ if and only if $UW(\relsys{X})\in \Forb(\F)$ and $\relsys X$ is induced on $X$ by $L(UW(\relsys {X}))$ (in the other words, $UW(\relsys{X})\in \Forb(\F)$ is a witness of $\relsys{X}$).
\end{lem}
\begin{proof}
Assume that $\relsys{X}\in \Lifts$ and also put
$$\relsys{A}=W(\relsys{X}),$$
$$\relsys{B}=UW(\relsys{X}).$$

It follows from the construction that there exists a homomorphism $\varphi:\relsys{B}\to \relsys{A}$ which is the identity on $\relsys{X}$.

If there was a homomorphism $\varphi'$ from some $\relsys{F}$ 
 to
$UW(\relsys{X})$ then, by composing with $\varphi$, there also exists a homomorphism from
$\relsys{F}$ to $W(\relsys{X})$.  This is not possible, since $\relsys{A}$ is
a witness.

Let us assume now that $\relsys{X}$ is not induced by $L(\relsys{B})$ on $X$.  From the construction  of $L(\relsys{B})$ we have trivially
that for each $\vec{v}\in \ext{X}{i}$, there is also $\vec{v}\in \extl{L(\relsys{B})}{i}$.  Assume that there is some $\vec{v}\in \extl{L(\relsys{B})}{i}$ consisting only
of vertices from $\relsys{X}$ such that $\vec{v}\notin \ext{X}{i}$.  Let $\varphi'$ be the homomorphism $\relsys{P}_i\to L(\relsys{B})$ such that $\varphi'(\overrightarrow{R_i})=\vec{v}$. Again by composing with $\varphi$
we obtain a homomorphism $\relsys{P}_i\to L(\relsys{A})$, a contradiction with $\vec{v}\notin \ext{X}{i}$. Thus $\relsys{X}$ is induced by $x$ on $L(\relsys{B})$.

In the reverse direction,
if $UW(\relsys{X})$ is a witness then $\relsys{X}\in \Lifts$.  The conditions listed in the lemma are
precisely the conditions for $UW(\relsys{X})$ to be a witness.
\end{proof}

\begin{defn}
\label{roothomo}
For a rooted structure $(\relsys{X},\overrightarrow{R})$ we define an {\rm $i$-rooted homomorphism} $(\relsys{X}, \overrightarrow{R})$ $\to \relsys{Y}$ as a homomorphism
$f:\relsys{X}\to\relsys{Y}$ such that $f(\overrightarrow{R})\in \ext{Y}{i}$ if and only if $\overrightarrow{R}\in \ext{X}{i}$.

For relational structure $\relsys{A}$ and $\relsys{X}$ a sublift of $L(\relsys{A})$, we say that $\relsys X$ is {\em $\relsys{A}$-covering} if and only if there is a homomorphism $f:\relsys{A}\to UW(\relsys{X})$. 

Similarly for piece $\Piece_i$ and $\relsys{X}$ a sublift of $L(\relsys{P}_i)$ such that $\overrightarrow{R}_i\notin \ext{X}{i}$, we say that $\relsys X$ is {\em $\Piece_i$-covering} if and only if $X$ contains all roots of $\overrightarrow{R}_i$ and there is a homomorphism $\varphi:\relsys{P}_i\to UW({\relsys{X}})$ such that $f$ is the identity on $\overrightarrow{R}_i$.
\end{defn}

Our first characterization of classes $\Lifts$ is in terms of rooted homomorphisms and coverings.

\begin{thm}
\label{reprezent}
For a fixed finite $\F$, the class $\Lifts$ (defined above before Theorem \ref{mainthm}) satisfies:

$\relsys{X}\in \Lifts$ if and only if
\begin{enumerate}
\item[(a)] there is no homomorphism $\relsys{Y}\to \relsys{X}$, where $\relsys{Y}$ is $\relsys{F}$-covering for some $\relsys{F}\in \F$,
\item[(b)] for every $i=1,\ldots, N$ and every $\Piece_i$-covering $\relsys{Z}$ there is no $i$-rooted homomorphism $f:(\relsys{Z},\overrightarrow{R}_i)\to \relsys{X}$.
\end{enumerate}
\end{thm}
\begin{figure}
\centerline{\includegraphics{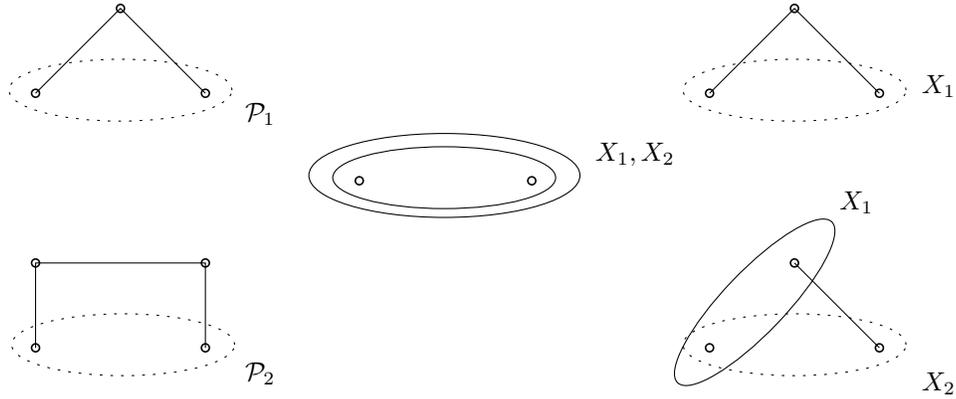}}
\caption{Pieces of the 5-cycle (up to isomorphisms and permutations of roots), inclusion-minimal $C_5$-covering sublifts, and inclusion-minimal $\Piece_1$-covering and $\Piece_2$-covering sublifts. Roots are denoted by dotted lines.}
\end{figure}

\begin{lem}
 Conditions $(a)$ and $(b)$ hold for every $\relsys{X}\in \Lifts$.
\end{lem}
\begin{proof}
Fix $\relsys{X}\in \Lifts$. Assume that $(a)$ does not hold for some $\relsys{Y}$ that is $\relsys{F}$-covering for some $\relsys{F}\in\F$. Since there is a homomorphism $\relsys{F}\to UW(\relsys{Y})$ and a homomorphism $\relsys{Y}\to\relsys{X}$ we also have a homomorphism $\relsys{F}\to UW(\relsys{Y}) \to UW(\relsys{X})$, a contradiction with Lemma \ref{UWlem}.

To show $(b)$ use a rooted analogy of the same proof.
\end{proof}

\begin{proof}[Proof of Theorem \ref{reprezent}]

Take lift $\relsys{X}$ such that $\relsys{X}\notin \Lifts$.  By Lemma \ref{UWlem}
we have one of the following cases:
\begin{itemize}
\item[I.] $\relsys{X}$ is not induced by $L(UW(\relsys{X}))$ on $\relsys{X}$.

In this case we have some homomorphism $f:\relsys{P}_i\to
UW(\relsys{X})$ such that $f(\overrightarrow{R}_i)\notin \ext{X}{i}$.  Assume that $i$
is chosen so that the number of vertices of $\relsys{P}_i$ is minimal.

Denote by $\relsys{Y}$ a maximal (non-induced) sublift of $L(\relsys{P}_i)$
such that $f$ is also a homomorphism from $\relsys{Y}$ to $\relsys{X}$.
We need to show that $\relsys{Y}$ is $\Piece_i$-covering to get a contradiction
with $(b)$.

Denote by $\relsys{C}_1,\ldots,\relsys{C}_t$ the components of $\relsys{P}_i\setminus Y$.
Now denote by $\Piece_{k_1},\ldots, \Piece_{k_t}$ the pieces corresponding to these components
and by $f_1,\ldots, f_t$ the homomorphisms $\relsys{P}_{k_j}\to \relsys{P}_i$, $j=1,\ldots,t$ mapping non-roots
of $\relsys{P}_{k_j}$ to vertices of $\relsys{C}_j$ and roots to vertices of $\relsys{Y}$.

Because $f(f_j(x))$ is a homomorphism $\relsys{P}_{k_j}\to UW(\relsys{X})$ we
have from the minimality of the counterexample $f(f_j(\overrightarrow{R}_{k_j}))\in
\ext{X}{k_j}$ and thus also $f_j(\overrightarrow{R}_{k_j})\in \ext{Y}{k_j}$.
This holds for every $j=1,\ldots,t$, and thus we also have a homomorphism $\relsys{P_n}\to UW(\relsys{Y})$ 
that is the identity on $Y$.  This prove that $\relsys{Y}$ is $\Piece_i$-covering.

\item[II.] There is a homomorphism $f$ from some $\relsys{F}\in\F$ to $UW(\relsys{X})$.

Assume that $f(F)\cap X$ is empty.  In this case there is $i$ such that $f(F)$ is contained among the vertices of a copy of $\Piece_i$ in $UW(X)$.  In this case the lift $\relsys{X}$ is covering because it contains a tuple in $\ext{X}{i}$. A contradiction.

Denote by $\relsys{Y}$ a maximal (non-induced) sublift of $L(\relsys{F})$ so $f$
is also a homomorphism $\relsys{Y}\to\relsys{X}$. Because there is a nonempty
intersection of $f(F)$ and $X$, $Y$ is nonempty.  We can show that $\relsys{Y}$
is covering by the same argument as in I, getting a contradiction with $(a)$ too.
\end{itemize}
\end{proof}

Observe that properties $(a)$ and $(b)$ directly translate to a family $\F'$ that has the property that the shadow of $\Forbi(\F')$ is $\Forb(\F)$. This leads to the desired explicit characterization of the class $\Lifts$.
\begin{thm}
\label{reprezent2}
Let $\F'$ be a class of $\F$-lifts satisfying the following:
\begin{enumerate}
\item $\relsys{X}\in \F'$ for every lift $\relsys{X}$ such that there is an $\relsys{F}$-covering lift $\relsys{Y}$ for some $\relsys{F}\in \F$ together with a
surjective homomorphism $\relsys{Y}\to\relsys{X}$,
\item $\relsys{X}\in \F'$ for every structure $\relsys{X}$ such that there is $1\leq i\leq N$ and a $\Piece_i$-covering rooted structure $(\relsys{Y}, \overrightarrow{R})$ together with a surjective $i$-rooted homomorphism $\relsys{Y}\to\relsys{X}$,
\item  $\F'$ contains no other structures.
\end{enumerate}

Then we have:
\begin{enumerate}
  \item $\F'$ is a finite family.
  \item $\Forbi(\F')=\Lifts$ and thus $\Age(\Forbi(\F'))$ is an amalgamation class. The shadow $\relsys{U}$ of the generic $\relsys{U}'=\lim \Age(\Forbi(\F'))$ is a universal structure for $\Forb(\F)$.
\end{enumerate}
\end{thm}

\begin{proof}
$\F'$ is necessarily finite, because the number of vertices of lifts $\relsys{X}\in
\F'$ is bounded by the number of vertices of structures $\relsys{A}\in \F$.
From the construction above it follows that $\Forbi(\F')$ is precisely the class of
structures satisfying conditions $(a)$ and $(b)$.
\end{proof}

We used the notion of rooted homomorphisms (and thus classes $\Forbi(\F')$) to define our lifted classes.  It is
easy to see that the classes $\Forb(\F')$  are not powerful enough to 
extend the expressive power of lifts.
\begin{lem}
\label{stupne}
Assume that there is a class $\F$ and a lifted class $\F'$ such that $\Forb(\F')$ contains a generic structure (lift) whose shadow is a universal structure of the class $\Forb(\F)$.  Then the class $\Forb(\F)$ itself contains a generic structure.
\end{lem}
\begin{proof}
Observe that all classes $\Forb(\F')$ are monotone. That is, for any $\relsys{X}\in \Forb(\F')$, a relational structure $\relsys{Y}$ created from $\relsys{X}$ by removing some of its tuples also belongs to $\Forb(\F')$.

In particular $\Forb(\F')$ is closed under constructing shadows and thus $\Forb(\F)$ may be thought of as a subclass of $\Forb(\F')$ (modulo the signature of relational structures).

Now take any $\relsys{A},\relsys{B},\relsys{C}\in \Forb(\F)$ and their lifts $\relsys{X},\relsys{Y},\relsys{Z}$ such that they contain no new tuples.  These lifts are in the class $\Forb(\F')$.  Now consider $\relsys{W}$ an amalgamation of $\relsys{X}$ and $\relsys{Y}$ over $\relsys{Z}$ and its shadow $\relsys{D}$. Then $\relsys{D}$ is an amalgamation of $\relsys{A}$ and $\relsys{B}$ over $\relsys{C}$.
\end{proof}

\section{Bounding arities}
\label{bounded}

The expressive power of lifts can be limited in several ways. For example, it is natural to 
restrict arities of the newly added relations.  It follows from the above proof that the arities of new relations
in our lifted class $\Lifts$ depend on the size of a maximal inclusion-minimal cut of the Gaifman graph of a forbidden structure.

In this section we completely characterize the minimal arity of generic lifts of classes $\Forb(\F)$.
This involves a non-trivial Ramsey-type statement stated below as Lemma \ref{sets}. As a warm-up, we first show that the generic universal graph for the class $\Forb(C_5)$ cannot be
constructed by finite monadic lifts.

Consider, for contradiction, a monadic lift $\relsys{U}'$ which is both a ultrahomogeneous relational structure and whose
shadow $\relsys{U}$ is universal for the class $\Forb(C_5)$. Since all extended
relations are monadic, we can view them as a finite coloring of vertices.
For $v\in \relsys{U}$ we shall denote by $c(v)$ the color of $v$ or, equivalently, the set of
all extended relations $\ext{U}{i}$ such that $(v)\in \ext{U}{i}$.

Since graphs in $\Forb(C_5)$ have unbounded chromatic number, we know that the
chromatic number of $\relsys{U}$ is infinite. Consider the decomposition of
$\relsys{U}$ implied by $c$. Since the range of $c$ is finite, one of the
graphs in this decomposition has infinite chromatic number. Denote this subgraph
by $\relsys{S}$.

In fact it suffices that $\relsys{S}$ is not bipartite. Thus $\relsys{S}$ contains an odd cycle. The shortest odd cycle has length $\geq 7$ and thus $\relsys{S}$ contains
 an induced
path of length 3 formed by vertices $p_1,p_2,p_3,p_4$.  Additionally there is a vertex $v$ of degree at least 2.  Because the graph is triangle free, the vertices $v_1$ and $v_2$ connected to $v$ are not connected by an edge.

From the ultrahomogeneity of $\relsys{U}'$ we know that the partial isomorphism mapping
$v_1\to p_1$ and $v_2\to p_4$ can be extended to an automorphism $\varphi$ of $\relsys{U}'$.
The vertex $\varphi(v)$ is connected to $p_1$ and $p_4$ and thus together with
$p_1,p_2,p_4$ contains either a triangle or a $5$-cycle.  It follows that the generic lift $\relsys{U}'$ cannot be monadic.

In this section we prove that there is nothing special here about arity 2 nor about the pentagon. One can determine the minimal arity of generic lifts for general classes $\Forb(\F)$. Towards this end we shall need a Ramsey-type statement, which we formulate after introducing the following:

Let $S$ be a finite set with a partition $S_1\cup S_2\cup\ldots\cup S_n$. For $v\in S$ we denote by $i(v)$ the index $i$ such that $v\in S_i$. Similarly, for a tuple $\vec{x}=(x_1,x_2,\ldots,x_t)$ of elements of $S$ we denote by $i(\vec{x})$ the tuple $(i(x_1), i(x_2),\ldots,i(x_t))$. We make use of the following:
\begin{lem}
\label{sets}
For every $n\geq 2$, $r<n$ and $K$ integers, there is a relational structure $\relsys{S}=(S,\relS)$, with vertices $S=S_1\cup S_2\cup\ldots \cup S_n$ (the sets $S_i$ are mutually disjoint) and a single relation $\relS$ of arity $2n$ with the following properties:
\begin{enumerate}
  \item Every $(v_1,u_1,v_2,u_2,\ldots, v_n,u_n)\in \relS$ satisfies $v_i\neq u_i\in S_i, i=1,\ldots, n$.
  \item For every $\vec{v},\vec{u}\in \relsys{S}$, $\vec{v}\neq\vec{u}$, $\vec{v}$ and $\vec{u}$ have at most $r$ common vertices.
  \item 
For every coloring of tuples of $S$ of size $r$ ($r$-tuples) using $K$ colors there is a $2n$-tuple $\vec{v}\in \relS$ such that every two $r$-tuples $\vec{x},\vec{x}'$ consisting of vertices of $\vec{v}$ such that $i(\vec{x})=i(\vec{x}')$ have the same color.
\end{enumerate}
\end{lem}
\begin{proof}
This statement follows from results obtained by Ne\v set\v ril and R\"odl
\cite{NRo}. Although not stated explicitly, this is a ``partite version''
of the main result of \cite{NRo}. It can also be  obtained directly by
means of the amalgamation method, see \cite{N4,NR2}. In this work this result plays
an auxiliary role only and we omit the proof.
\end{proof}
Given a relational structure $\relsys{S}=(S,\relS)$ with a relation $\relS$ of arity $2n$ and a rooted relational structure $(\relsys{A}, \overrightarrow{R})$ of type $\Delta$ with $\overrightarrow{R}=(r_1,r'_1,r_2,r'_2,\ldots, r_n,r'_n)$, we denote by $\relsys{S}*(\relsys{A},\overrightarrow{R})$ the following relational structure $\relsys{B}$ of type $\Delta$:

The vertices of $\relsys{B}$ are equivalence classes of a equivalence relation $\sim$ on $\relS \times A$
generated by the following pairs:

$$(\vec{v}, r_i)\sim (\vec{u}, r_i) \hbox{ if and only if } \vec{v}_{2i}=\vec{u}_{2i},$$
$$(\vec{v}, r'_i)\sim (\vec{u}, r'_i) \hbox{ if and only if } \vec{v}_{2i+1}=\vec{u}_{2i+1},$$
$$(\vec{v}, r_i)\sim (\vec{u}, r'_i) \hbox{ if and only if } \vec{v}_{2i}=\vec{u}_{2i+1}.$$

Denote by $[\vec{v}, r_i]$ the equivalence class of $\sim$ containing $(\vec{v}, r_i)$.

We put $\vec{v}\in \rel{B}{j}$ if and only if $\vec{v}= ([\vec{u},v_1],[\vec{u},v_2],\ldots, [\vec{u},v_t])$ for some $\vec{u}\in \relS$ and $(v_1,v_2,\ldots, v_t)\in \rel{A}{j}$.

This is a variant of the indicator construction introduced in Section \ref{degree}.  It essentially means replacing every tuple of $\relS$ by a disjoint copy of $\relsys{A}$ with roots $\overrightarrow{R}$ identified with vertices of the tuple.

For a given vertex $v$ of $\relsys{S}*(\relsys{A},\overrightarrow{R})$ such that $v=[\vec{u}, r_i]$ (or $v=[\vec{u}, r'_i]$) we shall call the vertex $v'=\vec{u}_{2i}$ (or $v'=\vec{u}_{2i+1}$, respectively) {\em the vertex corresponding to $v$ in $\relsys{S}$}.  Note that this gives the correspondence between vertices of $\relsys{S}$ and $\relsys{S}*(\relsys{A},\overrightarrow{R})$ restricted to vertices $[\vec{v}, r_i]$ and $[\vec{v}, r'_i]$.
\begin{figure}
\label{Hconstruct}
\centerline{\includegraphics{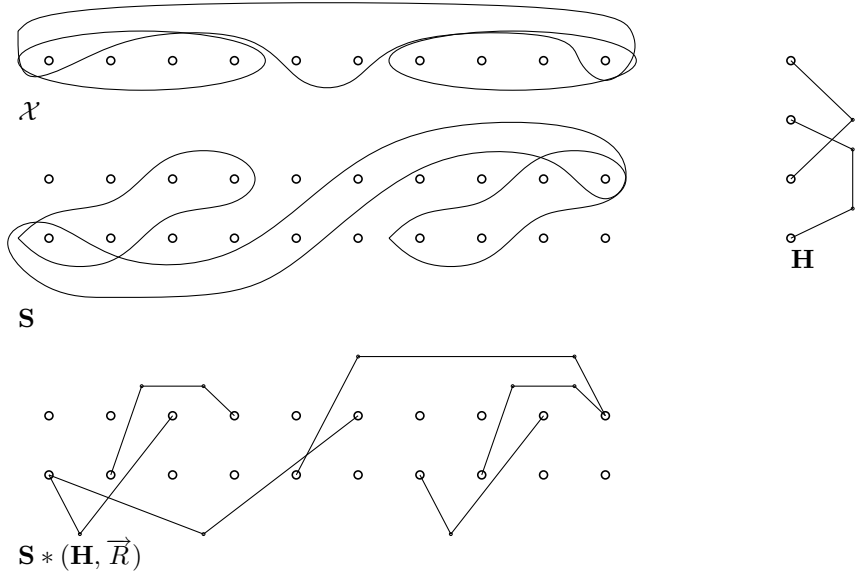}}
\caption{Sketch of the construction.}
\end{figure}

A finite family of finite relational structures is called {\em minimal}  if and only if
all structures in $\F$ are cores and there is no homomorphism between two
structures
in $\F$.

The following is the main result of this section.

\begin{thm}
\label{aritythm}
Denote by $\F$ a minimal family of finite connected relational structures. There is a lift $\K$ of the class $\Forb(\F)$ that contains new relations of arity at most $r$ with a generic structure (lift) $\relsys{U}$ if and only if all minimal cuts of $\relsys{F}\in \F$ consist of at most $r$ vertices.
\end{thm}
\begin{proof}

The construction of the lifted class $\Lifts$ in the proof of Theorem $\ref{mainthm}$ adds
relations of arities corresponding to the sizes of minimal cuts of
$\relsys{F}\in \F$, so one direction of Theorem \ref{aritythm} follows directly from the proof
of Theorem $\ref{mainthm}$.

In the opposite direction fix a class $\F$, $r\geq 1$, and a relational structure
$\relsys{F}\in \F$ containing a minimal cut $C=\{r_1,r_2,\ldots, r_n\}$ of size 
$n>r$.  Assume, for a contradiction, that there exists a lift $\K$ of the class
$\Forb(\F)$ with a generic lift $\relsys{U}$ and contains new relations
of arities at most $r$. Denote by $K$ the number of different relational structures on $r$ vertices
appearing in $\K$.

For brevity, assume that $\relsys{F}\setminus C$ has only two connected components. Denote by $\Piece_1=(\relsys{P}_1,\overrightarrow{R}_1)$ and $\Piece_2=(\relsys{P}_2,\overrightarrow{R}_2)$ the pieces generated by $C$ such that $\overrightarrow{R}_1=\overrightarrow{R}_2=(r_1,r_2,\ldots, r_n)$. (For three or and more pieces we can proceed analogously.)

Now we construct a relational structure $\relsys{H}$ as follows:
$$H= (P_1\times \{1\})\cup (P_2\times \{2\}),$$
and put
$$((v_1,1),\ldots ,(v_t,1))\in \rel{H}{i}\hbox{ if and only if }(v_1,\ldots ,v_t)\in \rel{P_1}{i},$$ 
$$((v_1,2),\ldots ,(v_t,2))\in \rel{H}{i}\hbox{ if and only if }(v_1,\ldots ,v_t)\in \rel{P_2}{i},$$ with no other tuples.  
In other words, $\relsys{H}$ is the disjoint union of $\relsys{P}_1$ and $\relsys{P}_2$.  We shall consider $\relsys{H}$ rooted by the tuple $$\overrightarrow{R}=
((r_1,1),(r_1,2),(r_2,1),(r_2,2),\ldots,(r_n,1),(r_n,2)).$$

Take the relational structure $\relsys{S}$ given by Lemma \ref{sets} and put $\relsys{D}=\relsys{S}* (\relsys{H},\overrightarrow{R})$. This construction for the two pieces of 5-cycle is shown in Figure \ref{Hconstruct}.
For a vertex $v\in S$ denote by $m(v)$ the vertex of $\relsys{D}$ corresponding to $v$ (if it exists) or an arbitrary vertex of $\relsys{D}$ otherwise.

Denote by $f$ the mapping defined by $[(a,1),\vec{x})]\mapsto a$ for $a\in P_1$ and $[(a,2),\vec{x}]\mapsto a$ for $a\in P_2$. It is easy to check that $f$ is a homomorphism $\relsys{D}\to \relsys{F}$.
Additionally, for $v\in \relsys{D}$ put $$M(v)=\{\vec{x}: ((a,t),\vec{x})\hbox{ is in equivalence class } v\}.$$
Observe that for a vertex $v\in \relsys{D}$ such that $f(v)\in C$, $M(v)$ may contain multiple tuples, while for all other vertices
$M(v)$ contains precisely one tuple.

Assume, to the contrary, that there is a homomorphism $\varphi:\relsys{F}\to\relsys{D}$.  By composition we have that $\varphi \circ f$ is a homomorphism $\relsys{F}\to\relsys{F}$.  Because $\relsys{F}$ is a core, we also know that $\varphi \circ f$ is an automorphism of
$\relsys{F}$. It follows that $\varphi$ is an injective homomorphism $\relsys{F}\to\relsys{D}$.
For $v\in \relsys{F}$, denote by $M(v)$ the set $M(\varphi(v''))$ where $v''$ is uniquely defined by $\varphi\circ f(v'')=v$.
It follows that for $v\in F\setminus C$, $M(v)$ consists of single tuple. For tuple $\vec{x}\in \rel{F}{i}$, there is a tuple $\varphi(\vec{x})\in \rel{D}{i}$ if and only if the sets $M(v), v\in \varphi(\vec{x})$, have a nonempty intersection (i.e. all belong to the single copy of some piece $\relsys{P}_i$) and thus also the sets $M(v), v\in \varphi\circ f(\vec{x})$, have a nonempty intersection.

As the relational systems $\relsys{P}_i\setminus C$ are connected, it follows
that all $M(v), v\in \relsys{P}_i\setminus C$, are equivalent singleton sets.
Denote by $\vec{x}_1$ the tuple such that $M(v)=\{\vec{x}_1\}$ for $x\in P_1\setminus C$ and
by $\vec{x}_2$ the tuple such that $M(v)=\{\vec{x}_2\}$ for $v\in P_2\setminus C$.
Because copies of pieces in $\relsys{D}$ corresponding to a single tuple $\vec{x}\in \cal
X$ are not connected, we have $\vec{x}_1\neq \vec{x}_2$.
 Finally, because every vertex in $C$ is connected to $P_1\setminus C$ by some tuple (by minimality of the cut $C$), we have $\vec{x}_1\in M(v)$ for every $v\in C$ and
analogously $\vec{x}_2\in M(v)$ for every $v\in C$.  It follows that the sets $\vec{x}_1$ and $\vec{x}_2$ overlap on the whole of $C$. Thus $\vec{x}_1$ and $\vec{x}$ overlap on $r$ or more vertices. This is a contradiction with the construction of the relational system $\relsys{S}$. It follows that there is no homomorphism $\relsys{F}\to \relsys{D}$.

There is also no homomorphism $\relsys{F'}\to \relsys{D}$ for any $\relsys{F'}\in \F, \relsys{F'}\neq \relsys{F}$, because composing such a homomorphism with $f$ would
lead to a homomorphism $\relsys{F'}\to \relsys{F}$ that does not exist.  It follows that $\relsys{D}\in \Forb(\F)$.

Take the generic lift $\relsys{U}\in \K$.  Every embedding $\Phi:\relsys{D}\to \psi(\relsys{U})$ ($\psi(\relsys{U})$ is the shadow of $\relsys{U}$) implies a $K$-coloring of $r$-tuples with elements of $D$ (colors are defined by the additional relations of $\relsys{U}$) and thus also 
a $K$ coloring of $r$-tuples of $\relsys{S}$.  Consequently, using Lemma \ref{sets}, there is a tuple $\vec{v}\in \relsys{S}$, such that $\vec{v}=(u_1,v_1,u_2,v_2,\ldots, u_n,v_n)$ and the relations added by the lift $\K$ are equivalent on $\Phi(u_i)$ and $\Phi(v_i), i=1,\ldots, n$. Thus $\relsys{U}$ induce on both sets $\{\Phi(u_1),\Phi(u_2),\ldots, \Phi(u_n)\}$ and $\{\Phi(v_1),\Phi(v_2),\ldots, \Phi(v_n)\}$ the same lift $\relsys{X}$. ($\relsys{X}$ is the lift of the relational structure induced by $\relsys{F}$ on $C$.)  Consequently, there is a partial isomorphism of $\relsys{U}$ mapping $\Phi(u_i)\to \Phi(v_i)$. By genericity of the relational structure $\relsys{U}$ this partial isomorphism extends to an automorphism $\Psi$ of $\relsys{U}$. From the construction of the relational system $\relsys{D}$ this mapping $\Psi$ sends a root of the image of piece $\Piece_1$ to the corresponding roots of the image of piece $\Piece_2$. Thus the shadow of $\relsys{U}$ contains copy of $\relsys{F}\in \F$, and this is a desired contradiction.
\end{proof}

\section {Special cases of small arities}
\label{monadic}

By Theorem \ref{aritythm} it follows that the only minimal classes of
finite relational structures $\F$ such that the class $\Forb(\F)$ has a monadic
lift that forms an amalgamation class are precisely the classes $\F$ such that
all minimal vertex cuts of the Gaifman graph of each $\relsys{F}\in \F$ have size 1.
Examples forming an amalgamation class include graphs whose blocks are all complete graphs.

Consider even more restricted classes $\F$ of structures consisting from
(relational) trees only (see Definition \ref{reltree} for relational trees). In this case we can claim a much stronger result:
there exists a finite  universal object $\relsys{D}$ which is a retract of a
universal structure $\relsys{U}$. 

\subsection{Finite dualities and constraint satisfaction problems}

\label{Duality}
A constraint satisfaction problem (CSP) is the following decision problem:

\smallskip

\noindent
Instance: A finite structure $\relsys{A}$.

\noindent
Question: Does there exist a homomorphism $\relsys{A}\to\relsys{H}$?
\smallskip

We denote by $\CSP(\relsys{H})$ the class of all finite structures $\relsys{A}$ with $\relsys{A}\to\relsys{H}$.
It is easy to see that the class $\CSP(\relsys{H})$ coincides with a particular instance of lifts and shadow. 

Recall that a {\em finite duality} (for structures of given type) is any equation
$$\Forb(\F)=CSP(\relsys{D})$$
where $\F$ is a finite set \cite{NesetrilPultr,NTardif,Hell}. $\relsys{D}$ is called the {\em dual of $\F$}. We also write $\relsys{D}_\F$ for the dual of $\F$ (it is easy to see that $\relsys{D}_\F$ is up to homomorphism-equivalence uniquely determined). The pair $(\F,\relsys{D})$ is called a {\em dual pair}.  In a sense duality is a simple constraint satisfaction problem: the existence of a homomorphism to $\relsys{D}$ (i.e.~a $\relsys{D}$-coloring) is equivalently characterized by a finite set of forbidden substructures. Dualities play a role not only in complexity problems but also in logic, model theory, the theory of partial orders and category theory. In particular, it follows from \cite{Atserias} and \cite{Rossman} that dualities coincide with those first-order definable classes which are homomorphism-closed. 

Finite dualities for monadic lifts include all classes $CSP(\relsys{H})$. We formulate this as follows:

\begin{prop}
\label{prop21}
For a class $\K$ of structures the following two statements are equivalent:
\begin{enumerate}
\item $\K=\CSP(\relsys{H})$ for finite $\relsys{H}$.
\item There exists a class $\K'$ of monadic lifts such that:
\begin{enumerate}
\item The shadow of $\K'$ is $\K$.
\item $\K'=\Forb(\F')\cap \Forbi(\relsys{K}_1)$, where $\F'$ is a finite set of monadic covering lifts of edges (i.e. every $\relsys{F}\in\F'$ contains at most one non-unary tuple.) while every vertex belongs to a unary lifted tuple.
\end{enumerate}
\end{enumerate}
\end{prop}

\begin{proof}[Proof (sketch)]
1.~obviously implies 2.

In the opposite direction construct $\relsys{H}$ as follows:  Let
$\relsys{H}_0$ be a lift with a vertex for every consistent combination of new
relations  $\ext{H_0}{i}$, and with relations $\rel{H_0}{i}$ empty.  Now construct a lift
$\relsys{H}$ on the same vertex set as $\relsys{H}_0$ with $\ext{H}{i}=\ext{H_0}{i}$.
Put tuple $\vec{x}\in \rel{H}{i}$ if and only if the structure induced by $\vec{x}$ on
$\relsys{H}_0$ with $\vec{x}$ added to $\rel{H}{i}$ is in $\Forb(\F')$. Consequently
if $\Age(\Forb(\F'))$ is an amalgamation class then $\Age(\Forb(\F))$ is amalgamation class too.
\end{proof}

In the language of dualities this amounts to saying that the classes $\CSP(\relsys{H})$ are just those classes described by shadow dualities of the simplest kind: the forbidden lifts are just vertex-colored edges.


As discussed in Section \ref{homounivsection}, finite dualities have been characterized:
\begin{thm}[\cite{NTardif}]
For every type $\Delta$ and for every finite set $\F$ of finite relational trees there exists a dual structure $\relsys{D}_\F$. Up to homomorphism-equivalence there are no other dual pairs.
\end{thm}
\label{dualpairs}
Various constructions of structure duals of given $\F$ are known \cite{NTardif2}.
It follows from this section that we have a yet another approach to this problem:
\begin{corollary}
\label{nasdual}
Let $\F$ be a set of finite relational trees of finite type, then there exists a finite set of lifted structures $\F'$ with the following properties:
\begin{itemize}
  \item[$(i)$] $\Age(\Forbi(\F'))$ is an amalgamation class (and thus there is universal $\relsys{U}'\in \Forbi(\F')$),
  \item[$(ii)$] all lifts in $\Forbi(\F')$ are monadic,
  \item[$(iii)$] $\psi(\relsys{U'})=\relsys{U}$ is universal for $\K$,
  \item[$(iv)$] $\relsys{U}'$ has a finite retract $\extsys{D}_\F$ and consequently $\psi(\extsys{D}_\F)=\relsys{D}_\F$ is a dual of $\F$.
\end{itemize}
\end{corollary}

\begin{proof}[Proof]
Observe that the inclusion-minimal cuts of a relational tree
all have size 1. Thus for a fixed family $\F$ of finite relational trees
 Theorem \ref{mainthm} establishes the existence of a monadic lift
that gives a generic structure $\relsys{U}'$ whose shadow is (homomorphism-) universal for $\Forb(\F)$.

This structure $\relsys{U}'$ is countable.  To get a dual, we find finite
$\relsys{X}\in \Lifts$ which is a retract of $\relsys{U}'$ and for which there is still a homomorphism $\relsys{Y}\to \relsys{U}'$ if and only if there is a homomorphism $\relsys{Y} \to \relsys{X}$.

The set $\F'$ is given by Theorem  $\ref{reprezent}$.
 Observe that every inclusion-minimal covering set of every piece of a tree is induced by a single
tuple and thus the class $\Lifts$ is defined by forbidden (rooted) homomorphisms of structures covered by  single tuple.
This means that the generic structure $\relsys{U}'$ has a finite retract defined by
all consistent combinations of new relations of its vertices.
\end{proof}

Note that it is also possible to construct $\relsys{D}_\F$ in a finite way
without using the \Fraisse{} limit: for every possible combination of new
relations on a single vertex, create a single vertex of $\relsys{D}_\F$ and then
keep adding tuples as long as possible so that $\relsys{D}_\F$ is still in $\Lifts$ (similarly
as in the proof of Proposition \ref{prop21}).

Finally, let us remark that one can prove that $\extsys{U}$ (and thus also $\relsys{U}$) has a finite presentation.

\label{aplikace}

%
%

\subsection {Forbidden cycles and Urysohn spaces (binary lifts)}
\label{urysohnsection}

We briefly turn our attention to binary lifts.  This relates some of the earliest results on universal graphs with recently intensively studied
Urysohn spaces.

We shall consider a finite family $\F$ consisting of graphs of odd cycles of lengths
$3,5,7,\ldots, l$.  As shown by \cite{Mekler} (see also \cite{CherlinShelahShi}) these families have universal
graphs in $\Forbi(\F)$ and, as shown by \cite{CherlinShi}, these are the only
classes defined by forbidding a finite set of  cycles.
These classes also form especially easy families of pieces.  In fact each piece
is an undirected path of length at most $l$, where $l$ is the length of the longest cycle
in $\F$ with both ends of the path being roots.  This allows a particularly easy
description of the lifted structure.

We use the following definition which is motivated by metric spaces.
When specialized to graphs, this definition is analogous to (the corrected form) of an
$s$-structure \cite{Mekler}. However this approach also gives a new easy description (i.e. finite presentation) of the lifted structure by the same construction that was used for Urysohn space in Section \ref{homprostorchapter}.  

\begin{defn}
\label{evenoddp}
A pair $(a,b)$ is considered to be an {\em even-odd pair} if $a$ is an even non-negative integer or $\omega$, and $b$ is odd non-negative integer or $\omega$.

For even-odd pairs $(a,b)$ and $(c,d)$ we say that $(a,b)\leq (c,d)$ if and only if $a\leq c$ and $b\leq d$.  Consider $a+\omega=\omega$ and $\omega+b=\omega$.  Put $$(a,b)+(c,d) = (\min(a+c,b+d),\min (a+d,b+c)).$$

For a set $S$, a function $d$ from $S$ to even-odd pairs is called an {\em even-odd distance function on $S$} if the following conditions are satisfied:
\begin{enumerate}
\item $d(x, y) = (0,b)$, $b$ is any odd number or $\omega$, if and only if   $x = y$,
\item $d(x, y) = d(y, x)$,
\item $d(x, z) \leq d(x, y) + d(y, z)$.

\end{enumerate}
Finally a pair $(S,d)$ where $d$ is an even-odd distance function for $S$ is called an {\em even-odd metric space}.
\end{defn}

Note that the even-odd metric spaces differ from the usual notion of the metric space
primarily by the fact that the ordering of values of the distance function is not
linear, but forms a 2-dimensional partial order.  
Some basic results about metric spaces are valid even in this setting.

An even-odd metric space can form a stronger version of the distance metric on the
graph.  For a graph $G$ we can put $d(x,y)=(a,b)$ where $a$ is length of the
shortest walk of even length connecting $x$ and $y$, while $b$
is the length of the shortest walk of odd length. 

The even-odd distance metric specifies the length of all possible walks: for a graph
$G$ and an even-odd distance metric $d$ we now have a walk connecting $x$ and
$y$ of length $a$ if and only if $d(x,y)=(b,c)$ such that $b\leq a$ for $a$ even or $c\leq
a$ for $a$ odd.

It is well known that the generic metric space exists for several classes of
metric spaces \cite{d9,Nguyen}. (See also Chapter \ref{homprostorchapter}.)  Analogously we have:
\begin{lem}
\label{Ulema}
There exists a generic even-odd metric space $\Urysohn_{eo}$.
\end{lem}
\begin{proof}
 We prove that the class $\M$ of all finite even-odd metric spaces is an amalgamation class.

To show that $\M$ has the amalgamation property, 
 take a free amalgamation $\relsys{D}$ of even-odd metric spaces $\relsys{A}$,
$\relsys{B}$ over $\relsys{C}$. This amalgamation is not an even-odd metric space, since some distances are not defined.

We can however define a walk from $v_1$ to $v_t$ of length $l$ in $\relsys{D}$ as a sequence of vertices $v_1,v_2,v_3,\ldots, v_t$
and distances $d_1,d_2,d_3,\ldots, d_{t-1}$ such that $\sum_{i=1}^{t-1} d_i=l$ and $d_i$ is present in
the even-odd pair $d(v_i,v_{i+1})$ for $i=1,\ldots,t-1$.

We produce the even-odd metric space $\Urysohn_{eo}$ on the same vertex set as
$\relsys{D}$, where the distance between vertices $a,b\in E$ is the even-odd pair
$(l,l')$ such that $l$ is the smallest even value such that there exists a walk
joining $a$ and $b$ of length $l$ in $\relsys{D}$, and $l'$ is the smallest odd value such
that there exists a walk from $a$ to $b$ of length $l'$.

It is easy to see that $\Urysohn_{eo}$ is an even-odd metric space (every triangle
inequality is supported by the existence of a walk) and other properties of
the amalgamation class follow from the definition.
\end{proof}

The graphs containing no odd cycle up to length $l$ can be axiomatized by a simple
condition on their even-odd distance metric.  Denote by $\K_l$ the class of all countable even-odd metric spaces such that there are no vertices $x,y$ such that $d(x,y)=(a,b)$ with $a+b\leq l$.  The existence of the generic even-odd metric space $\Urysohn_l=(U_l,d_l)$ for class $\K_l$ is a simple consequence of Lemma \ref{Ulema}. In fact $\Urysohn_l$ is a subspace of $\Urysohn_{eo}$ induced by all those vertices $v$ of $\Urysohn_{eo}$ satisfying $d(v,v)=(0,b)$ and $b>l$.

\begin{thm}
For a metric space $\Urysohn_l=(U_l,d_l)$ denote by $G_l=(U_l,E_l)$ the graph on the vertex set $U_l$ where $\{x,y\}\in E_l$ if and only if $d(x,y)=(a,1)$.

For every choice of odd integer $l\geq 3$, $G_l$ is a universal graph for the class $\Forb(C_l)$.
\end{thm}
\begin{proof}
The graph $G_l$ does not contain any odd cycle up to length $l$ due to the fact that any two
vertices $x,y$ on an odd cycle of length $k$ have the distance $d(x,y)=(a,b)$ where
$a+b$ is at most $k$.

Now consider any countable graph $G=(V,E)$ omitting odd cycles of length at most
$l$. Construct the corresponding even-odd distance metric space $(V,d_G)$. By the universality argument,  $(V,d_G)$ is subspace of $\Urysohn_l$ and thus $G$ is a subgraph of $G_l$.
\end{proof}

The explicit construction of the rational Urysohn space, as described in Chapter \ref{homprostorchapter}, can be carried over to
even-odd metric spaces. This is captured by the following definition.

\begin{defn}
\label{defUf}
The vertices of $\U$ are functions $f$ such that:
\begin{enumerate}
\item[(1)] The domain $D_f$ of $f$ is a finite (possibly empty) set of functions and $\emptyset$.
\item[(2)] The range of $f$ consist of even-odd pairs.
\item[(3)] For every $g\in D_f$ and $h\in D_{g}$, we have $h\in D_f$.
\item[(4)] $D_f$ using metric $d_{\U}$ defined below forms an even-odd metric
space.
\item[(5)] $f$ defines an extension of the even-odd metric space on vertices $D_f$ by adding a new vertex. This means that $f(\emptyset)=(0,x)$ and for every
$g, h\in D_f$ we have $f(g)+f(h)\leq d_{\U}(g,h)$ and $f(g)\geq f(h)+d_\U(g,h)$.

\end{enumerate}
The metric $d_{\U}(f,g)$ is defined by:
\begin{enumerate}
  \item if $f=g$ then
$d_{\U}(f,g)=f(\emptyset),$
  \item if $f\in D_g$ then
$d_{\U}(f,g)=g(f),$
  \item if $g\in D_f$ then 
$d_{\U}(f,g)=f(g),$
  \item if none of above hold then
$d_{\U}(f,g)=min_{h\in D_f\cap D_g} f(h)+g(h).$

The minimum is taken elementwise on pairs.
\end{enumerate}
 \end{defn}

\begin{thm}\label{genericT}
$(\U,d_\U)$ is the generic even-odd metric space.
\end{thm}

\section{Indivisibility results}
A pair $(A,B)$ is called a {\em partition of a structure} $\relsys{R}$ if $A$ and $B$ are disjoint sets of vertices of $\relsys{R}$ and $A\cup B=R$.  We denote by $\relsys{A}$ the structure induced on $A$ by $\relsys{R}$ and by $\relsys{B}$ the structure induced on $B$ by $\relsys{R}$.

A structure $\relsys{R}$ is {\em weakly indivisible} if for every partition
$(A,B)$ of $R$ for which some finite induced substructure of $\relsys{R}$ does
not have copy in $\relsys{A}$, there exists a copy of $\relsys{R}$ in
$\relsys{B}$.

For a minimal finite family $\F$ of finite structures, we call structure
$\relsys{A}$ a {\em minimal homomorphic image} of $\relsys{F}\in \F$ if and only if
$\relsys{A}$ is a homomorphic image of $\relsys{F}$ and every proper substructure
of $\relsys{A}$ is in $\Forb(\F)$.

The weak indivisibility of ultrahomogeneous structures has been studied in \cite{Sauer}.
In this section we briefly discuss basic (in)divisibility results on universal
structures for classes $\Forb(\F)$.  

We say that a class $\K$ has the {\em free vertex amalgamation} property if, for any $\relsys{A},\relsys{B}\in \K$, and relational structure $\relsys{C}$ consisting of a single vertex and embeddings $\alpha:\relsys{C}\to\relsys{A}$ and $\beta:\relsys{C}\to\relsys{B}$, there is $(\relsys{D},\gamma,\delta)$, $\relsys{D}\in \K$, that is a free amalgamation of $(\relsys{A},\relsys{B},\relsys{C},\alpha,\beta)$.

\begin{thm}[\cite{Sauer}]
\label{Sauerthm}
Let $\relsys{H}$ be a ultrahomogeneous structure such that $\Age(\relsys{H})$ has free vertex amalgamation property and contains unique (up to isomorphism) structure on single vertex.  Then $\relsys{H}$ is weakly indivisible. 
\end{thm}

The construction of universal structures as shadows of ultrahomogeneous structures makes
this result particularly easy to apply to obtain indivisibility results for
universal structures for classes $\Forb(\F)$. This leads to the following partial classification of classes $\F$ that do admit a weakly indivisible structure universal for $\Forb(\F)$.
\begin{thm}
Fix a finite minimal family of connected finite structures $\F$.  
\begin{enumerate}
\item The class $\Forb(\F)$ contains a universal structure that is weakly indivisible if every vertex-minimal cut $C$ of every homomorphic image $\relsys{A}$ of $\relsys{F}\in \F$ is of size at least 2
and additionally the structure induced by $\relsys{A}$ on $C$ is connected and has no cuts of size 1.
\item All universal structures $\relsys{U}$ in $\Forb(\F)$ are
divisible if there is a structure $\relsys{A}$ which is a minimal homomorphic image of $\relsys{F}\in \F$ such
that $\relsys{A}$ contains a cut $C$ of size 1.
\end{enumerate}
\end{thm}

\begin{proof}

To prove $2.$, fix $\relsys{A}$, a minimal homomorphic image of $\relsys{F}\in \F$
that has a vertex cut $C$ of size 1.  Denote by $\Piece_1,\Piece_2,\ldots,\Piece_n$
all the pieces of $\relsys{A}$ generated by the cut $C$.  Fix $\relsys{U}$, the universal
structure for $\Forb(\F)$. Denote by $U_i$, $i=1, \ldots, n$ the set of all
vertices $v$ of $\relsys{U}$ such that there is a rooted homomorphism from
$\Piece_i$ to $\relsys{U}$ mapping a root of $\Piece_i$ to $u$.

The structure induced on $U_i$ by $\relsys{U}$ is not universal as it does not contain a homomorphic image of $\relsys{P}_i$.
 Similarly, the structure induced on
$U\setminus \bigcup_{i=1}^n U_i$ is not universal since there is no
homomorphic image of $\relsys{P}_1$.  Consequently, $\relsys{U}$ is divided
into finitely many substructures such that none is universal, resulting in 
the divisibility of $\relsys{U}$.

$1.$ follows from the weak indivisibility of the class $\Lifts$.  To apply Theorem
\ref{Sauerthm} we only need to show that the class $\Lifts$ admits a free vertex
amalgamation.  This follows directly from the construction of the amalgamation in the proof
of Theorem $\ref{mainthm}$. The amalgamation constructed is not free in general:
every new tuple $\vec{v}$ added to an extended relation $i$ has the property that there is a homomorphism from the
structure induced by $R_i$ on $\relsys{P}_i$ into the vertices
of tuple $\vec{v}$.   But since we have a free amalgamation of the shadow and since all cuts of all homomorphic
images do not have cuts of size 1, we have the free vertex amalgamation property.
\end{proof}

\section{Lifted classes with free amalgamation}
The explicit construction of lifts provided by Theorem $\ref{mainthm}$ allows
more insight into their structure. In this section we give an answer to a problem of
Atserias \cite{Atseriaspriv} which asks whether there always exists a lift
of a class $\Forb(\F)$ with the free amalgamation property.  The answer is negative in general.
We can however precisely characterize families $\F$ with this property.

Recall that structure is {\em irreducible}  if it does not have a cut (alternatively, any two distinct vertices are contained in a tuple of $\relsys{A})$.
 
\begin{thm}
Let $\F$ be a minimal family of finite connected relational structures.  Then the following statements are equivalent:
\begin{enumerate}
\item There exists class $\K'$ such that:
\begin{enumerate}
  \item[(a)] $\Age(\K')$ is an amalgamation class,
  \item[(b)] $\K'$ is closed under free amalgamation,
  \item[(c)] the shadow of $\K'$ is $\Forb(\F)$.
  \item[(d)] $\K$ contains a generic structure.
\end{enumerate}
\item Every minimal cut in $\relsys{F}\in \F$ induces an irreducible substructure.
\end{enumerate}

\end{thm}
\begin{proof}
To show that $2.$ implies $1.$ it suffices to verify that for such classes $\F$ the amalgamation $\relsys{V}$ constructed in the proof of Theorem $\ref{mainthm}$ is the free amalgamation of $\relsys{X}$ and $\relsys{Y}$ over $\relsys{Z}$.  The amalgamation is constructed as $L(\relsys{D})$, where $\relsys{D}$ is the free amalgamation of shadows of $\relsys{X}, \relsys{Y}, \relsys{Z}$. Now for every tuple $\vec{v}\in \ext{V}{i}$ we have a homomorphism $\varphi:P_i\to \relsys{D}$. Because $\relsys{P_i}$ induces on vertices $\overrightarrow{R}_i$ an irreducible relational structure, the map must correspond to the shadow of $\relsys{A}$ or $\relsys{B}$ and thus there are no new edges in $\relsys{V}$.

In the opposite direction, assume that  $\F$ and
a class $\K'$ satisfying $(a)$, $(b)$, $(c)$ and $(d)$ are given.

Define a class $\overline{\K}'$ as the class of all $\relsys{A}\in \K'$
such that for each tuple $\vec{v}\in \ext{A}{i}$ the relational
structure induced by $\psi(\relsys{A})$ on $\vec{v}$ is irreducible.

We claim that $\overline{\K}'$ also satisfies $(a)$, $(b)$, $(c)$ and $(d)$.
 Assume the contrary.  Then, for some $i$, there is $\relsys{A}\in \K'$ and $\vec{v}\in \ext{A}{i}$ such that structure induced by
$\psi(\relsys{A})$ on $\vec{v}$ is reducible and there is no $\relsys{B}\in \overline{\K}'$ such that the shadow of $\relsys{A}$ is the same as the shadow of $\relsys{B}$. Without loss of generality we can assume that $\relsys{A}$
is a counterexample with the minimal number of tuples.  Denote by $v_1, v_2$ subsets
of vertices of $\vec{v}$ such that the free amalgamation of structures induced on
$v_1$ and $v_2$ by structure $\psi(\relsys{A})$ over vertices $v_1\cup v_2$ is
equivalent to the structure induced on $\vec{v}$ by structure $\psi(\relsys{A})$.

Now construct $\relsys{B}$ as the free amalgamation of the structure induced on $(A\setminus v)\cup v_1$ and on $(A\setminus v)\cup v_2$ by $\relsys{A}$ over vertices
$v_1\cap v_2$. Because $\K'$ is an amalgamation class, we have $\relsys{B}\in \K'$.  The shadow of $\relsys{B}$ is
equivalent to the shadow of $\relsys{A}$ and either $\relsys{B}\in \overline{\K}'$ or
$\relsys{B}$ is a smaller counterexample, a contradiction with minimality of
$\relsys{A}$.

Now take $\relsys{F}\in \F$ such that there is a vertex-minimal cut $C$ and
the structure $\relsys{C}$ induced on $C$ by $\relsys{F}$ is not irreducible.  By
Theorem \ref{aritythm} we know that the arity of the lift $\overline{\K}'$ must be at least $|C|$.
While the lift $\overline{\K}'$ can have unbounded arity, from the fact that the images of
$C$ are reducible, the arity of the lift $\overline{\K}'$ on images of $\relsys{C}$
is strictly smaller than $C$.  The proof of Theorem \ref{aritythm} only deals with
extended tuples on images of cuts $C$ and thus we have a contradiction.
\end{proof}

\chapter{Conclusion (summary and open problems)}

\section{Finite presentations of ultrahomogeneous structures}
In Part I we exhibited finite presentations of several ultrahomogeneous structures
as provided by the classification programme (Section \ref{Genericsection}).  We
gave finite presentations of all ultrahomogeneous undirected graphs
(Chapter~\ref{graphschapter}), all partial orders
(Chapter~\ref{homposetchapter}), all ultrahomogeneous tournaments
(Chapter~\ref{generictournaments}) and the rational Urysohn metric space (Chapter \ref{homprostorchapter}).

%
There are a number of ways to continue research in this direction. Naturally
one might seek further positive examples.
\begin{prob}
\label{grpr}
Which ultrahomogeneous structures are finitely presented? In particular which
ultrahomogeneous directed graphs are finitely presented?
\end{prob}
To determine precisely which ultrahomogeneous directed graphs are finitely presented,
one would clearly need to provide a condition on a set $\F$ of finite tournaments
that would imply the existence of a finite presentation of the universal directed
graph for the class $\Forbi(\F)$. To decide this problem one needs to know how
complex the finite presentation can be.  This leads to:
\begin{prob}
Give a more precise formulation of the notion of a finite presentation.
\end{prob}
We hope to attack this problem in \cite{novyclanek}.

A related question is the following:
\begin{prob}[Cameron \cite {Debbie}]
Is there a simpler finite presentation of the generic partial order? (Simpler than $(\HF,\leq_\HF)$ given in Definition \ref{hfdef}.)
\end{prob}
Even if $(\HF,\leq_\HF)$ fits very well with our notion
of a finite presentation, in several ways it can be considered inferior to
the finite presentations of the Rado graph (Section \ref{explicitmodels}).  It is significantly less
streamlined and in a
way it ``just partially encodes the amalgamation process.''  This can be considered necessary
(the definition of ordinal numbers can be considered similarly faulty)
but still there is hope that some well-established mathematical structure
will be shown to give the generic partial order in a similarly easy way
to the variants of representations of the Rado graph.

A similar question can be raised about the finite presentation of the rational
Urysohn space.  In addition, we have only given a presentation of the rational Urysohn
space, not finite presentations for all ultrahomogeneous metric spaces.
Clearly our finite presentation can be easily modified, for example, for the generic metric
space $\Urysohn_\N$ where the distances are integers. More generally we can ask:

\begin{prob}
\label{prpr}
Give a finite presentation of the generic metric space $\Urysohn_S$ for a given $S\subset[0,+\infty]$ satisfying the 4-values condition (see Definition \ref{4values}).
\end{prob}
  Similarly as in Problem \ref{grpr}, we may require a simple representation of
$S$.  Problem \ref{prpr} is harder because the construction of the rational
metric space uses several properties of $\Q$, such as the fact that $\Q$ is
closed under addition and subtraction.

Generalizing even further, the construction of the Urysohn metric space can be
modified for relational structures with axiomatization similar to that for metric spaces. In Chapter \ref{homprostorchapter} we gave an analogous finite presentation of the generic partial order.  In Chapter
\ref{Forbchapter} we gave a finite presentation of universal graphs for classes
$\Forb(C_l)$ where $C_l$ is a cycle of length $l>3$.
\begin{prob}
Give a finite presentation of universal structures for (some of) the classes
$\Forb(\F)$ constructed in Chapter \ref{Forbchapter}.
\end{prob}

\section{Finitely presented universal structures}
In Part II we looked for well-known finitely presented structures and tried to
prove their universality. In Chapter \ref{posetschapter} we gave a
catalogue of structures that are known to induce a universal partial order.  The
catalogue of finite presentations of partial orders can always be extended.
\begin{prob}
Find more examples of mathematical structures that form universal partial
orders.
\end{prob}

A particularly interesting special case is the following:
\begin{prob}
Denote by $\mathcal R$ the class of all recursive languages $NP$.  For
recursive languages $A, B\in \mathcal R$ put $A\leq_\mathcal R B$ if and only if $A$ is
polynomial-time reducible to $B$.  Does the quasi-order $(\mathcal R,\leq_\mathcal
R)$ contain a universal partial order?
\end{prob}
Assuming that $P$ is not equal to $NP$, the density of $(\mathcal R,\leq_\mathcal R)$
was shown in \cite{Diaz}. The way to embed any finite partial order is shown
in \cite{Pangrac}. It is still not known whether every countable partial order can be embedded
in $(\mathcal R,\leq_\mathcal R)$.

In Chapter \ref{cestickychapter} we focused on the universality of the
homomorphism order of restricted classes of relational structures. The
universality of oriented paths implies the universality of many other classes.
\begin{prob}
Find more classes of relational structures where the homomorphism order is universal (or give a reason why it is not).
\end{prob}

The study of the universality of the homomorphism order was originally motivated by
the study of embeddings of categories.  It is natural to ask when a representation
of a partial order can be strengthened to the representation of a category.

\begin{prob}
For which classes of relational structures $\K$ does there exist an embedding $\varphi$
of the generic partial order $(P,\leq_P)$ into $(\K,\leq_h)$ such that for any
$x,y\in P$ with $x\leq_P y$ there is only one homomorphism $\varphi(x)\to
\varphi(y)$? (Such a $\varphi$ is then an embedding of $(P,\leq_P)$ as a thin category.)
\end{prob}
It was shown in \cite{Pultr} that there is an embedding of any category
representable by sets and functions into the class of undirected graphs (with
homomorphisms as morphisms). Also it is known that there is no such embedding
into topologically restricted classes (such as planar graphs) as these classes
fail to represent all groups (see \cite{B}) and monoids (in case of bounded
degrees, see \cite{BP}).  Our embedding of the universal partial order into the
class of oriented paths ordered by the existence of a homomorphism shows that
embedding of partial orders is noticeably easier than embedding of categories
in general.

In Chapter \ref{Forbchapter} we gave a combinatorial proof of the existence of
universal structures for the classes $\Forb(\F)$, $\F$ family of finite connected structures. Based on the
explicit description of such universal structures we showed the relation to
homomorphism dualities and Urysohn spaces, as well as described some additional
properties.  A natural development of the main result of this chapter would be
to give the following:

\begin{prob}
Extend the techniques of construction of universal graphs in Chapter
\ref{Forbchapter} to all classes with local failure of amalgamation
(reproving \ref{Convigtonthm} in a combinatorial way).
\end{prob}
Local failure of amalgamation is a more combinatorial condition for the
existence of a universal structure for a given class than  finiteness of algebraic closure.
It may be interesting to extend this condition to a necessary and sufficient condition
for the existence of an $\omega$-categorical universal graph.

We also gave only a partial classification of divisibility results on the
$\omega$-categorical graphs for classes $\Forbi(\F)$.

\section{Classification programmes}
Perhaps the most challenging problem is to complete the classification
programmes.  We gave an overview of those programmes in Chapter 1,
so here we give just a short summary:

\subsection{The classification of ultrahomogeneous structures}  
The classification of ultrahomogeneous structures of type $\Delta=(2)$ has been
completed.  Here one has to refer to the fundamental works of Schmerl
\cite{Schmerl}, Lachlan \cite{La1,La2}, Lachlan and Woodrow \cite{lachlan} and
Cherlin \cite{Cherlin}. See Section \ref{classection}. 

Among other types of structures where the classification programme is completed
are, for example, homogeneous permutations \cite{CameronP}, colored partial
orders \cite{Sousa}, and 2-graphs \cite{Truss}.  Initial work has also been
done to classify structures of type $\Delta=(3)$ in \cite{res6}.

Recall that a relational structure $\relsys{S}$ is ultrahomogeneous if any
isomorphism between finite induced substructures of $\relsys{S}$ can be
extended to an automorphism of $\relsys{S}$. Various weaker notions of
ultrahomogeneity are discussed.  For instance it is possible to bound the size
of the substructures by a given constant $n$. This results in the notion of an {\em
$n$-homomorphism}.  Alternatively, only special classes of substructures can be
considered, resulting in the notion of {\em connected-homogeneity} or {\em
distance-transitivity}.  A structure $\relsys{S}$ is {\em set-homogeneous}, if
for every two substructures $\relsys{A}$ and $\relsys{B}$ that are isomorphic
there exists an automorphism $\varphi$ such that $\varphi(\relsys{A})=\relsys{B}$
(so we do not require all isomorphisms of $\relsys{A}$ and $\relsys{B}$ to
extend to automorphisms of $\relsys{S}$).

An interesting recent variant is suggested in \cite{CameronN}.  Consider
classes of relational structures that arise when the definition of ultrahomogeneity
is changed slightly, by replacing `isomorphism' by
`homomorphism' or `monomorphism'. We say that a structure $\relsys{S}$ belongs
to the class {\bf XY} if every x-morphism from an finite induced substructure of
$\relsys{S}$ into $\relsys{S}$ extends to a y-morphism from $\relsys{S}$ to
$\relsys{S}$; where $({\mathbf X},x)$ and $({\mathbf Y},y)$ can be $({\mathbf
I}, iso)$, $({\mathbf M}, mono)$, or $({\mathbf H}, homo)$. The classes that
arise are {\bf IH}, {\bf HM},  {\bf HH},  {\bf IM},  {\bf MM}.  A classification
of partial orders was given in \cite{Masul2, Debbie} and tournaments with loops
\cite{Masul1}.  The results seem to suggest that these relaxed variants of
ultrahomogeneity are easier to work with.

\subsection{The classification of universal structures}
In Section \ref{univsectioncl} we gave an overview of known results about
the existence of a universal structure for a given class $\K$. We outlined known sufficient
conditions for the existence of such a structure (results of \cite{CherlinShelahShi}
and \cite{Covington}) as well as several known examples. While several classes
have been characterized, these are all just special cases and no complete
classification is known. See \cite{CherlinShelah, CherlinShelahShi, restruct}
for a summary of the known results and suggestions for future research.

\subsection{The classification of Ramsey classes}
In Section \ref{ramseysection} we outlined the classification programme of
Ramsey classes based on the classification of ultrahomogeneous
structures. This programme has been suggested by Ne\v set\v ril as a realistic project
despite the fact that proving ages of ultrahomogeneous structures to be Ramsey
already presents interesting and difficult problems (see \cite{HN-Posets}).

\end{document}